\documentclass[letterpaper,12pt]{amsart}
\usepackage[foot]{amsaddr}

\usepackage[english]{babel}
\usepackage[utf8]{inputenc}
\usepackage[T1]{fontenc}

\usepackage{mathrsfs}
\usepackage{amsmath}
\usepackage{amsfonts}
\usepackage{amssymb}
\usepackage{amsthm}
\usepackage{hyperref}
\usepackage{indentfirst}
\usepackage{enumerate}
\usepackage{float}
\usepackage{multicol}
\usepackage{graphicx}
\usepackage{geometry}
\usepackage{color}
\usepackage[usenames,dvipsnames,svgnames,table,x11names]{xcolor}
\usepackage{dsfont}

\usepackage[titletoc]{appendix}

\usepackage{comment}

\usepackage{makecell}

\usepackage{xspace}

\usepackage[skip=10pt]{caption}

\newcommand{\new}{$\dagger$\xspace}

\emergencystretch=\maxdimen
\usepackage{pgf,tikz,pgfplots}
\pgfplotsset{compat=1.14}
\usetikzlibrary{arrows}
\usetikzlibrary{arrows}
\usepackage{rotating}
\usetikzlibrary{positioning}

 \usepackage{tikz}
\usepackage{xstring}
\usepackage{xparse}
\usepackage{etoolbox}
\usepackage{expl3}
\usepackage{pgfkeys}
\usepackage{pgfopts}
\usepackage{mathtools}

\usepackage{diagbox}

\usepackage{textcomp}

\usepackage[all,knot,arc,import,poly]{xy}

\newcommand\restr[2]{{
  \left.\kern-\nulldelimiterspace 
  #1   \vphantom{\big|} 
  \right|_{#2} 
  }}

\newtheorem{thm}{Theorem}[section]
\newtheorem{rmk}[thm]{Remark}

\newtheorem{prop}[thm]{Proposition}
\newtheorem{cor}{Corollary}[thm]
\newtheorem{lema}[thm]{Lemma}
\newtheorem{defi}[thm]{Definition}

\newtheorem{exe}[thm]{Example}
\newtheorem{exen}[thm]{\new Example}

\usepackage{array}
\newcolumntype{C}[1]{>{\centering\arraybackslash}p{#1}}

\everymath{\displaystyle}

\usepackage{multicol}
\usepackage{multirow}

\title{Explicit Constructions of Halphen Pencils}
\author{Aline Zanardini}
\address{Mathematics Department, University of Pennsylvania, USA}
\email{alinez@math.upenn.edu}
\date{\today}

\begin{document}

\maketitle

\begin{abstract}
We construct rational elliptic surfaces of index two by explicitly constructing their associated Halphen pencils in the projective plane $\mathbb{P}^2$. For each of the types of singular fibers that occur we construct at least one example having that type of fiber and in fact, for some, we construct all possible examples. We establish a precise dictionary between the fibers in a rational elliptic surface and the corresponding plane curves and, in particular, we study the singularities of the curves appearing in a Halphen pencil. 
\end{abstract}

\setcounter{tocdepth}{3}
\tableofcontents

\section{Introduction}

Elliptic surfaces play an  important role in many questions from different areas of  Mathematics and also in theoretical Physics. They are central in the classification of algebraic surfaces, they appear in the construction of exotic four manifolds and they are also present in the formulation of F-theory. Examples of elliptic surfaces include Enriques surfaces, Dolgachev surfaces, all surfaces of Kodaira dimension one and many rational surfaces. In this paper we are interested in the latter.

We say a rational surface $Y$ is a rational elliptic surface if $Y$ admits a fibration $f:Y \to \mathbb{P}^1$ whose generic fiber is a smooth curve of genus one. We do not necessarily assume the existence of a global section. If $Y$ is a rational elliptic surface, then there exists some $m\geq 1$, called the index of the fibration, so that $f$ is given by the anti-pluricanonical system $|-mK_Y|$. Moreover, $m=1$ if and only if $f$ admits a global section and whenever $m>1$ there exists exactly one multiple fiber in $Y$, this of multiplicity $m$ (see e.g. \cite[Chapter V, \S 6]{dc}). 

Rational elliptic surfaces admitting a global section have been widely studied under many different points of view. Different compactifications for their moduli space have been constructed and are well understood \cite{av}, \cite{kdw}, \cite{kdtwisted}, \cite{hl}, \cite{stab},\cite{mirW}; their automorphism groups have been classified \cite{tolga2},\cite{tolga1}; and all possible configurations of singular fibers are known \cite{list},\cite{perssonlist}. It is also known that these surfaces can be realized from a pencil of cubic curves in the plane (by blowing-up their nine base points) and explicit examples having a Mordell-Weil group with some particular rank have been considered in \cite[Theorem 5.6.2]{dc}, \cite{fusi},\cite{pastro} and \cite{salgado}.

Nevertheless,  there are not many explicit constructions in the literature for those rational elliptic surfaces that do not admit a global section. The goal of this paper is to provide such explicit constructions. For each of the types of singular fibers that occur we construct at least one example having that type of fiber and in fact, for some, we construct all possible examples. Our approach is purely geometric.

Similar to the $m=1$ case, rational elliptic surfaces of any index $m$ can be realized as a nine point blow-up of $\mathbb{P}^2$, where the nine points are base points of a certain pencil of plane curves. These are called Halphen pencils (of index $m$), after the French mathematician Georges Henri Halphen who first studied these objects in \cite{halphen}. 

Concretely, if $f:Y\to \mathbb{P}^1$ is a rational elliptic surface, then there exists a birational map $\pi: Y \to \mathbb{P}^2$ so that $f\circ \pi^{-1}$ is a pencil of plane curves of degree $3m$ having nine (possibly infinitely near) singular base points of multiplicity $m$. In particular, any fiber $F$ corresponds to a plane curve $B$ of degree $3m$, namely $\pi(F)$ \cite[Theorem 5.6.1]{dc}. 

An important ingredient in our approach is the study of the singularities of a plane curve occurring in a Halphen pencil. The log canonical threshold (lct) plays an important role. We establish some precise relations between the log canonical thresholds of the pairs $(Y,F)$ and $(\mathbb{P}^2,B)$, which provide us with bounds for the lct of the latter. We prove the following results, where $M_B$ (resp. $M_F$) denotes the largest multiplicity of a component of $B$ (resp. $F$):

\begin{thm}
If $F$ is any (non-multiple) fiber of $Y$, then the corresponding plane curve $B$ is such that
\[
lct(\mathbb{P}^2,B)\leq \frac{1}{M_B}\leq 2lct(Y,F)
\]
and these inequalities do not depend on the index $m$ of the fibration.

Further,
\begin{enumerate}[(i)]
\item if $m>1$ and $F$ is reduced, then $B$ is reduced and we have
\[
\frac{1}{m}<lct(\mathbb{P}^2,B)\leq lct(Y,F)
\]
\item if $M_F\geq m$ and $F$ is not reduced, then
\[
lct(Y,F)\leq lct(\mathbb{P}^2,B) 
\]
\end{enumerate}
\end{thm}

In fact we establish a dictionary between the fibers in the surface and the corresponding plane curves. When $m=2$ and $F$ is of type $II^*,III^*$ or $IV^*$ we obtain the following complete characterization for the plane curve $B$:

\begin{thm}
A fiber $F$ of type $II^*$ can only be realized by one of the following plane curves:
\begin{enumerate}[(i)]
\item a triple conic 
\item a nodal cubic and an inflection line, with the line taken with multiplicity three 
\item two triples lines
\item a conic and a tangent line, with the line taken with multiplicity four 
\item a line with multiplicity five and another line 
\end{enumerate}
If $F$ is of type $III^*$, then $F$ can only be realized by one of the following curves:
\begin{enumerate}[(i)]
\item a double line, a cubic and another line 
\item a double conic and another conic 
\item a triple conic
\item two triple lines 
\item a triple line, a double line and another line 
\item a triple line, a conic and a line 
\item a triple line and a cubic 
\item a conic and a line, with the line taken with multiplicity four 
\item a line with multiplicity four and two other lines  
\end{enumerate}
And whenever $F$ is of type $IV^*$ we have that $F$ can only be realized by one of the curves below:
\begin{enumerate}[(i)]
\item a double conic and a conic 
\item a double line, a conic and two lines 
\item a double line, a cubic and a line 
\item a double line and two conics
\item two double lines and two lines 
\item two double lines and a conic 
\item a double conic and two lines 
\item a triple conic 
\item a triple line, a conic and a line 
\item a triple line, a double line and another line  
\item a triple line and three lines 
\item a triple line and a cubic 
\end{enumerate}

Conversely, we can construct a Halphen pencil of index two, $\lambda B + \mu(2C)=0$, where $B$ is any one of the curves above and the corresponding (non-multiple) fiber is of type $II^*,III^*$ or $IV^*$. 
\end{thm}

And for any index $m$ we prove Proposition \ref{notredintro} below, providing a new proof for a result of Miranda \cite[Lemma 6.4]{stab}.

\begin{prop}
If $F$ is of type $II^*,III^*$ or $IV^*$, then $B$  cannot be reduced.
\label{notredintro}
\end{prop}

The many examples we construct complement (and in some sense complete) the few existing ones considered in \cite{fuji90},\cite{fuji98},\cite{kimu1} and \cite{kimu2}.  Some examples of rational elliptic surfaces $Y$ of index two have also been constructed independently by Antonio Laface \cite{exlaface}. His approach, however, builds on the analysis of lattices by considering the $(-1)$ curves of $Y$ as integer points in a polytope inside $Pic(Y)$ and then using information about the intersection form. 

Surprisingly, Halphen pencils have appeared in \cite{squares} in the solution of a problem in Diophantine geometry and a generalization to higher dimensions has been considered in \cite{cheltsov1} and \cite{cheltsov2}. Other possible applications include the study of certain $K3$ surfaces \cite{dpn}, \cite{zhang} and the construction of: F-theory compactifications \cite{kimu1},\cite{kimu2}, discrete Painlev\'{e} equations \cite{sakai} and  a moduli space for rational elliptic surfaces of index two \cite{azstab}. 

In a forthcoming paper \cite{azstab} we will use the constructions of Halphen pencils presented here and the results from Section \ref{hp} to study the stability, in the sense of geometric invariant theory (GIT), of Halphen pencils under the action of $SL(3)$ and hence, to construct a compactification for the moduli space of rational elliptic surfaces of index two.  This was our original motivation and it also explains why we have exhibited all possible examples of Halphen pencils yielding rational elliptic surfaces with fibers of type $II^*,III^*$ or $IV^*$ (Theorems \ref{allpossibleiistar}, \ref{allpossibleiiistar} and \ref{allpossibleivstar}). In \cite{azstab} we show these types of fiber are associated with unstability.

The work of Miranda in \cite{stab} describes the GIT stability conditions for pencils of plane cubics, which leads to a compactification of the moduli space of rational elliptic surfaces with section. Such compactification agrees with the one obtained by the same author in \cite{mirW}, where the surfaces are described by equations, namely their Weierstrass models.

When the existence of a global section is not assumed, there is no analogue for the Weierstrass model and even the dimensions of the parameter spaces involved are much higher, which makes it much harder to solve the classification problem. Tools from Birational Geometry and the results obtained in this paper have helped us to overcome such difficulties.

\subsection*{Organization}

The paper is organized as follows: We begin, in Section \ref{lct},  by presenting some basic background material from Birational Geometry that will be needed later. Next, in Section \ref{res} we introduce the geometric objects we are interested in, namely rational elliptic surfaces and their associated Halphen pencils. Section \ref{hp} is devoted to establishing the dictionary between the curves in a Halphen pencil and the fibers in the corresponding elliptic surface. In particular, we study the log canonical thresholds of the plane curves that can occur in a Halphen pencil. In Section \ref{iistariiistar} we prove Theorems \ref{allpossibleiistar}, \ref{allpossibleiiistar} and \ref{allpossibleivstar} that completely characterize all possible examples of Halphen pencils of index two yielding fibers of type $II^*,III^*$ and $IV^*$. Then in Section \ref{constructions} we present many new constructions of Halphen pencils of index two. These are summarized in Tables \ref{examplesi7i8} through \ref{examplesiistar}. A more detailed and explicit geometric description is given right after in Section \ref{gd}. We work over $\mathbb{C}$.

\subsection*{Acknowledgments}
 I would like to thank my advisor, Antonella Grassi, for the many insightful conversations, the numerous enriching suggestions on earlier drafts and the encouragement for writing this paper. I also would like to thank  Antonio Laface for helpful conversations. This work is part of my PhD thesis and it was partially supported by a Dissertation Completion Fellowship at the University of Pennsylvania. This project originated from work supported by NSF Grant No. DMS-1440140 while the author was a program associate at MSRI during the Spring 2019 semester.

\section{The Log Canonical Threshold}
\label{lct}

We first recall some necessary background notions in Birational Geometry concerning log canonical pairs. We refer to \cite{singpairs} for a more detailed exposition. 

Let $X$ be a normal algebraic variety and let $\Delta=\sum d_iD_i \subset X$ be a $\mathbb{Q}$-divisor, i.e. a $\mathbb{Q}$-linear combination of prime divisors.

\begin{defi}
Given any birational morphism $\mu: \tilde{X}\to X$, with $\tilde{X}$ normal, we can write
\[
K_{\tilde{X}} \equiv \mu^*(K_X+\Delta)+\sum a_E E
\]
where $E\subset \tilde{X}$ are distinct prime divisors, $a_E \doteq a(E,X,\Delta)$ are the discrepancies of $E$ with respect to $(X,\Delta)$ and a nonexceptional divisor $E$ appears in the sum if and only if $E=\mu_{*}^{-1}D_i$ for some $i$ (in that case with coefficient $a(E,X,\Delta)=-d_i$).
\end{defi}

\begin{defi}
A \textbf{log resolution} of the pair $(X,\Delta)$ consists of a proper birational morphism $\mu: \tilde{X}\to X$ such that $\tilde{X}$ is smooth and $\mu^{-1}(\Delta)\cup Exc(\mu)$ is a divisor with global normal crossings.
\label{logres}
\end{defi}

\begin{defi}
We say $(X,\Delta)$ is  \textbf{log canonical} if $K_X+\Delta$ is $\mathbb{Q}$-Cartier and given any log resolution $\mu:\tilde{X}\to X$ we have
\[
K_{\tilde{X}} \equiv \mu^*(K_X+\Delta)+\sum a_E E
\]
with all $a_E\geq - 1$.
\end{defi}

\begin{rmk}
In particular, if $X$ is smooth and $\Delta=d_iD_i$ is simple normal crossings, then $(X,\Delta)$ is log canonical if and only if $d_i\leq 1$ for all $i$.
\label{snc}
\end{rmk}

\begin{defi}
The number
\[
lct(X,\Delta) \doteq \sup \{\,\,t\,\,;\,\, (X,t\Delta)\,\, \mbox{is log canonical}\}
\]
is called the \textbf{log canonical threshold} of $(X,\Delta)$. 
\end{defi}

We can also consider a local version:
\[
lct_p(X,\Delta) \doteq \sup \{\,\,t\,\,;\,\, (X,t\Delta)\,\, \mbox{is log canonical in an open neighborhood of $p$}\}
\]
where $p\in X$ is a closed point.

\begin{lema}
Given a log resolution $\mu: \tilde{X}\to X$, write
\[
K_{\tilde{X}}=\mu^*K_X+\sum a_iE_i \qquad \mbox{and} \qquad \mu^*\Delta=\tilde{\Delta}+\sum b_iE_i
\]
where $\tilde{\Delta}=d_i\tilde{D}_i$ (resp. $\tilde{D}_i$) denotes the strict transform of $\Delta$ (resp. $D_i$) under $\mu$ and $E_i \subset \tilde{X}$ are the exceptional divisors of $\mu$.
Then
\[
lct(X,\Delta)=\min\left\{\frac{1+a_i}{b_i},\frac{1}{d_i},1\right\}
\]
\label{complct}
\end{lema}

\section{Rational Elliptic Surfaces with a Multiple Fiber}
\label{res}

We now introduce the geometric objects we are interested in studying in this paper, namely rational elliptic surfaces and their associated Halphen pencils. We point the reader to \cite[Chapter V, \S 6]{dc} for more details.

Let $Y$ be a smooth and projective surface and let $f:Y\to C$ be a fibration (a surjective proper flat morphism) such that the generic fiber is a smooth genus one curve.  And, further, assume $Y$ is relatively minimal, meaning there are no $(-1)-$curves in any fiber. We will refer to this data as an \textbf{elliptic surface}\footnote{in the literature this is often referred to as a genus one fibration }, even though the generic fiber of $f$ does not necessarily have the structure of an elliptic curve. 

Any elliptic surface has finitely many singular fibers and the configuration of all non-multiple singular fibers is exactly the same as the one in the associated Jacobian fibration (see Section \ref{jac}). The possible non-multiple singular fibers have been classified by Kodaira and N\'{e}ron \cite{kodaira1},\cite{kodaira2} and \cite{neron} and Table \ref{fibers}  below gives the full classification. Over a field of characteristic zero, any multiple fiber is of type $I_n$ for some $n\geq 0$ \cite[Proposition 5.1.8]{dc} .

{\renewcommand{\arraystretch}{1.5}

\begin{table}[H]
\centering
\begin{tabular}{|c  | c | c |}
\hline 
\bf{Kodaira Type} & \bf{Number of Components} & \bf{Dual Graph}\\
\hline 
$I_0$ & 1 (smooth) & \begin{tikzpicture}[line cap=round,line join=round,>=triangle 45,x=1.0cm,y=1.0cm]
\clip(-.2,-.2) rectangle (.2,.2);
\begin{scriptsize}
\draw [fill=black] (0.,0.) circle (2.5pt);
\end{scriptsize}
\end{tikzpicture}\\
$I_1$ & 1 (with a node) & \begin{tikzpicture}[line cap=round,line join=round,>=triangle 45,x=1.0cm,y=1.0cm]
\clip(-.2,-.2) rectangle (.2,.2);
\begin{scriptsize}
\draw [fill=black] (0.,0.) circle (2.5pt);
\end{scriptsize}
\end{tikzpicture}\\
$I_n$ & $n\geq 2$ & $\tilde{A}_{n-1}$\\
$II$ & 1 (with a cusp) & \begin{tikzpicture}[line cap=round,line join=round,>=triangle 45,x=1.0cm,y=1.0cm]
\clip(-.2,-.2) rectangle (.2,.2);
\begin{scriptsize}
\draw [fill=black] (0.,0.) circle (2.5pt);
\end{scriptsize}
\end{tikzpicture}\\
$III$ & 2 & $\tilde{A}_1$ \\
$IV$ & 3 & $\tilde{A}_2$ \\
$I_n^*$ & $n+5$ & $\tilde{D}_{4+n}$\\
$IV^*$ & $7$ & $\tilde{E}_6$\\
$III^*$ & $8$ & $\tilde{E}_7$\\
$II^*$ & $9$ & $\tilde{E}_8$\\
\hline
\end{tabular}
\caption{Kodaira's Classification}
\label{fibers}
\end{table}
}

Given $f: Y \to C$ as before, we define the \textbf{index} of the fibration, and denote it by $d_{Y}$, as the positive generator of the ideal
\[
\{D\cdot Y_{\eta} \,\,; \,\, D\in \text{Pic}(Y)\} \trianglelefteq \mathbb{Z}
\]
where $Y_{\eta}$ is a generic fiber. Since $Y$ is projective $d_Y$ is always finite.

Note that $d_{Y}=1$ if and only if $f$ admits a section, if and only if the generic fiber has the structure of a (smooth) elliptic curve.

In this paper we are interested in the situation where $Y$ is rational and $d_Y>1$. One can show that when $Y$ is rational then $C\simeq \mathbb{P}^1$ (Lur\"{o}th's Theorem) and $f$ is given by the linear system $|-mK_Y|$, where $m=d_Y$ (Proposition \ref{Halphen}). In particular, $K_Y^2=0$ and $Y$ can be obtained as a nine-point blow-up of $\mathbb{P}^2$. 

\subsection{Halphen Pencils}

\begin{defi}
A \textbf{Halphen pencil of index $m$} is a pencil of plane curves of degree $3m$  with nine (possibly infinitely near) base points of multiplicity $m$.  
\end{defi}

Such geometric object is in one-to-one correspondence with rational elliptic surfaces:

\begin{prop}[{\cite[Theorem 5.6.1]{dc}}, {\cite[Main Theorem 2.1]{fuji90}}]
Let $f:Y \to \mathbb{P}^1$ be a rational elliptic surface of index $m$ and let $F$ be a choice of  a fiber of $f$, then there exists a birational map $\pi: Y \to \mathbb{P}^2$ so that $f\circ \pi^{-1}$ is a Halphen pencil of index $m$ and, moreover, $B\doteq \pi(F)$  is a plane curve of degree $3m$:
\[
\xymatrix{
F\subset Y \ar@{-->}[rr]^-\pi \ar[dr]_f & & \mathbb{P}^2 \supset B\doteq \pi(F) \ar@{-->}[dl]^{f \circ \pi^{-1}}\\
&\mathbb{P}^1 & 
}
\] 
Conversely, given a Halphen pencil of index $m$, taking the minimal resolution of its base points we obtain a rational elliptic surface of index $m$.
\label{Halphen}
\end{prop}

Moreover, since the canonical bundle formula (see e.g. \cite[Theorem 12.1]{barth}) implies Lemma \ref{l1} below we have that any Halphen pencil of index $m$ contains exactly one cubic of multiplicity $m$, which corresponds to the unique multiple fiber in the associated rational elliptic surface. In fact the cubic corresponds to a fiber of type $I_n$ for some $n\leq 9$ \cite[Proposition 5.1.8]{dc}. And if none of the base points are singular points of the cubic, then we can further restrict to $n \leq 3$ (Lemma \ref{Cnotsingular}).

\begin{lema}[{\cite[Proposition 5.61,(iii)]{dc}}]
If $f:Y \to \mathbb{P}^1$ is a rational elliptic surface, then $f$ has at most one multiple fiber.
\label{l1}
\end{lema}

\subsection{The Jacobian}
\label{jac}

If $f: Y \to C$ is an elliptic surface of index $m>1$, then the generic fiber $Y_{\eta}$ is a smooth genus one curve over the function field of $C$ that has no rational points over this field (the fibration has no sections). If we let $Jac(Y_{\eta})$ denote the corresponding Jacobian variety of divisors of degree $0$ on $Y_{\eta}$ that is, the connected component of the identity of $Pic(Y_{\eta})$, then we can construct an elliptic surface $J \to C$ with a section whose generic fiber $J_{\eta}$ is isomorphic to $Jac(Y_{\eta})$.  The fibration $J \to C$ is called the associated Jacobian fibration (to $f: Y \to C$). 

If $Y$ is rational, then $J$ is also rational \cite[Proposition 5.6.1 (ii)]{dc} and one can prove the following:

\begin{thm}[{\cite[Corollary 5.4.7]{dc}}]
Let $J \to \mathbb{P}^1$ be a rational elliptic surface with section. Given $m \geq 1$ and a closed point $p \in \mathbb{P}^1$ such that $J_p$ is of type $I_n, 0 \leq n\leq 9$, there exists a rational elliptic surface $Y \to \mathbb{P}^1$ of index $m$ with unique multiple fiber $Y_p=m\overline{Y}_p$ satisfying $\overline{Y}_p\simeq J_p$. Moreover, $[Y]$ is an element of order $m$ in $H^{1}(\mathbb{P}^1,\mathcal{J})$, the group of isomorphism classes of torsors over the generic fiber $J_{\eta}$.
\label{tateshaf}
\end{thm}

Moreover, the Shioda-Tate formula (see e.g. \cite[Corollary 6.13]{schuttshioda}) applied to the associated Jacobian fibration $J \to \mathbb{P}^1$ implies that the possible singular fibers occurring on $J$ (hence on $Y$) can have at most 9 irreducible components. In particular, following Kodaira's classification, if $F$ is a singular fiber of a rational elliptic surface $Y \to \mathbb{P}^1$, then $F$ is of type $I_n$ for $n\leq 9, II,III,IV,I_n^*$ for $n\leq 4,II^*,III^*$ or $IV^*$. In fact, given any integer $m>1$ any type in this list can be realized by some rational elliptic surface $Y \to \mathbb{P}^1$ of index $m$. More precisely, 

\begin{prop}[{\cite[Corollary 5.6.6]{dc}}]
 If $Y_p$ is a non-multiple fiber of a rational elliptic surface $Y \to \mathbb{P}^1$ of index $m$, then $b_2(Y_p)\leq 9$ and any Kodaira type satisfying this condition can be realized.
\label{mainexe}
\end{prop}

Such statement, however, does not provide explicit constructions, which we do in this paper.

\section{Halphen Pencils and Rational Elliptic Surfaces of index $m$}
\label{hp}

In this section we will establish a dictionary between the curves in a Halphen pencil and the fibers in the corresponding rational elliptic surface. In particular, we will provide a description of the singularities of  a plane curve in a Halphen pencil. But first we need to introduce some notations and deduce some equations.

We will fix a Halphen pencil of index $m$ and we will denote it by $\mathcal{P}$. The corresponding rational elliptic surface will be denoted by $f:Y \to \mathbb{P}^1$ and $\pi: Y \to \mathbb{P}^2$ will denote the blow-up at the nine base points of $\mathcal{P}$.

If $F$ is any (non-multiple) fiber of $Y$ we will denote by $B$ the corresponding plane curve of degree $3m$, i.e. $\pi(F)$. Further, $mC$ will denote the unique multiple cubic of $\mathcal{P}$ and $mE$ will denote the unique multiple fiber of $f$.

Because $-K_Y$ is nef, every smooth rational curve $R$ on $Y$ has self-intersection $R^2\geq -2$ (adjunction formula). This implies we can write the set of base points of $\mathcal{P}$ as in \cite[Section 2]{dolgcant}: 
\begin{equation}
\{P_1^{(1)},\ldots,P_1^{(a_1)},\ldots,P_k^{(1)},\ldots,P_k^{(a_k)}\}
\label{basepts}
\end{equation}
where $a_1+\ldots+a_k=9$, $P_j^{(1)}$ are points in $\mathbb{P}^2$ and $P_j^{(i+1)}$ is infinitely near to the previous point $P_j^{(i)}$ (of order 1). 

Moreover, if $C$ is smooth and we choose a flex point as the origin for the group law $\oplus$ on $C$, then  \cite{dolgcant}:
\[
a_1P_1^{(1)}\oplus \ldots \oplus a_kP_k^{(1)}=\varepsilon_m
\]
where $\varepsilon_m$ is a torsion point of order $m$ in $C$ (w.r.t $\oplus$).

Expressing the base points of $\mathcal{P}$ as in (\ref{basepts}) is the same as saying that each exceptional curve
\[
E_j \doteq \pi^{-1}(P_j^{(1)})
\]
consists of a chain of $(-2)$ curves of length $(a_j-1)$ with one more $(-1)$  curve at the end of the chain. The latter a multisection of degree $m$.

Thus, whenever we write
\begin{equation}
F=\overline{F}+d_1^{(1)}E_1^{(1)}+\ldots+d_1^{(a_1-1)}E_1^{(a_-1)}+\ldots+d_k^{(1)}E_k^{(1)}+\ldots+d_k^{(a_k-1)}E_k^{(a_k-1)}
\label{eqF}
\end{equation}
where $\overline{F}$ denotes the strict transform of $B$ under $\pi$ and each $E_j^{(i)}$ is the $\pi$-exceptional divisor over the base point $P_j^{(i)}$; we have the following (dual) picture for the components of $E_j$ appearing in the fiber $F$:

\begin{figure}[H]
\centering
\begin{tikzpicture}[line cap=round,line join=round,>=triangle 45,x=1.0cm,y=1.0cm]
\clip(-3.,-1) rectangle (7.,1.);
\draw [line width=1.5pt] (-2.,0.)-- (1.,0.);
\draw [line width=1.5pt] (3.,0.)-- (6.,0.);
\begin{scriptsize}
\draw [fill=black] (-2.,0.) circle (2.5pt);
\draw[color=black] (-2.,.5) node {$d_j^{(1)}E_j^{(1)}$};
\draw [fill=black] (0.,0.) circle (2.5pt);
\draw[color=black] (0.,.5) node {$d_j^{(2)}E_j^{(2)}$};
\draw [color=black] (2.,0.) node {$\ldots$};
\draw [fill=black] (4.,0.) circle (2.5pt);
\draw [fill=black] (6.,0.) circle (2.5pt);
\draw[color=black] (6.,.5) node {$d_j^{(a_j-1)}E_j^{(a_j-1)}$};
\end{scriptsize}
\end{tikzpicture}
\caption{Chains of exceptional rational curves appearing in $F$}
\label{chains}
\end{figure}
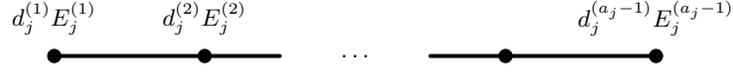

Because the chains $E_j$ are disjoint from each other, it follows that: 

\begin{lema}
If we color the nodes of the dual graph of $F$ corresponding to the components coming from $B$ in blue and the nodes corresponding to the exceptional components $d_j^{(i)}E_j^{(i)}$ in black, then every black node is connected to at most two other black nodes.
\label{nodes}
\end{lema}

This simple observation has some interesting consequences like Propositions  \ref{notreduced} and \ref{mult3} below. In Section \ref{iistariiistar} we also use Lemma \ref{nodes} repeatedly in order to characterize which curves $B$ can yield a fiber of type $II^*,III^*$ or $IV^*$ when $m=2$.

\begin{prop}
If  $F$ is of type $II^*, III^*$ or $IV^*$, then $B\doteq \pi(F)$ cannot be reduced.
\label{notreduced}
\end{prop}

\begin{proof}
If $B$ were reduced, then coloring the dual graph of $F$ as in Lemma \ref{nodes} we would obtain a black node  which is connected to more than two black nodes.
\end{proof}

\begin{prop}
If $F$ is of type $II^*$, then $M_{B}\geq 3$, where $M_B$ denotes the largest multiplicity of a component of $B$.
\label{mult3}
\end{prop}

\begin{proof}
Again, we look at the dual graph of $F$. Assuming $M_B <3$ contradicts Lemma \ref{nodes}.
\end{proof}


Writing $F$ as in (\ref{eqF}) we can further deduce Equation (\ref{numbercomp}) below, which computes the number of components of $F$.

\begin{prop}
If $n_F$ and $n_B$ denote the number of components of $F$ and $B$, respectively, then
\begin{equation}
n_F=n_B+\sum_{j=1}^{k}(a_j-1)-n_{E\backslash C}=n_B+\sum_{j=1}^{k}a_j-k-n_{E\backslash C}=n_B+9-k-n_{E\backslash C}
\label{numbercomp}
\end{equation}
where $n_{E \backslash C}$ denotes the difference between the number of components of $E$ and the number of components of $C$.
\end{prop}

The type of the multiple fiber $mE$ imposes restrictions on the numbers $n_{E\backslash C}$ appearing in Equation \ref{numbercomp} above. For instance, whenever  $m>1$ we have that

\begin{lema}
If $F$ is of type $IV^*$, then $n_{E\backslash C}\in \{0,1,2\}$ and if $F$ is of type $III^*$ (resp. $II^*$), then $n_{E\backslash C}\in \{0,1\}$ (resp. $n_{E\backslash C}=0$). 
\label{nEC}
\end{lema}

\begin{proof}
If $m>1$ and $F$ is of type $IV^*,III^*$ or $II^*$, then the classification in \cite{list} tells us the unique multiple fiber $mE$ of $Y$ can be realized as the strict transform of $mC$. If $F$ is of type $IV^*$, then $E$ is of type $I_0,I_1,I_2$ or $I_3$. Whereas if $F$ is of type $III^*$ (resp. $II^*$), then $E$ is of type $I_0,I_1$ or $I_2$ (resp. $I_0$ or $I_1$). 
\end{proof}

\begin{rmk}
If $m=1$ and $B$ is any given curve in $\mathcal{P}$, then we can always take the other generator of $\mathcal{P}$ to be a smooth cubic. In particular, we can always assume that $n_{E\backslash C}=0$ in Equation (\ref{numbercomp}).
\end{rmk}

We can also write
\[
K_Y=\pi^*K_{\mathbb{P}^2}+b_1^{(1)}E_1^{(1)}+\ldots+b_1^{(a_1)}E_1^{(a_1)}+\ldots+b_k^{(1)}E_k^{(1)}+\ldots+b_k^{(a_k)}E_k^{(a_k)}
\]
and
\[
 \pi^*B=\overline{F}+c_1^{(1)}E_1^{(1)}+\ldots+c_1^{(a_1)}E_1^{(a_1)}+\ldots+c_k^{(1)}E_k^{(1)}+\ldots+c_k^{(a_k)}E_k^{(a_k)}
\]

and we know how to compute each of the multiplicities $b_j^{(i)}\doteq b_j^{(i)}(B), c_j^{(i)}\doteq c_j^{(i)}(B)$ and $d_j^{(i)}\doteq d_j^{(i)}(B)$ rather explicitly.

For any base point $P_j^{(1)}$, the induced pencil on the surface obtained by blowing-up $P_j^{(1)}$ is 
\[
({\pi_j^{(1)}})^*\mathcal{P}-mE_j^{(1)}
\]
where $\pi_j^{(1)}$ is the blow-up map. In particular, given any curve $B$ of $\mathcal{P}$, the induced member is
\[
B_j^{(1)}+(m_{P_j^{(1)}}(B)-m)E_j^{(1)}
\]
where $B_j^{(1)}$ is the strict transform of $B$ under $\pi_j^{(1)}$ and $m_{P_j^{(1)}}(B)$ denotes the multiplicity of the point $P_j^{(1)}$ on the curve $B$.

In other words, $d_j^{(1)}=m_{P_j^{(1)}}(B)-m$ and, more generally,
\[
d_j^{(i)}=d_j^{(i-1)}+m_{P_j^{(i)}}(B)-m 
\]
where $m_{P_j^{(i)}}(B)$ denotes the multiplicity of the point $P_j^{(i)}$ on the strict transform of the curve $B$ under the blow-up of $P_j^{(1)},\ldots,P_j^{(i-1)}$.

On the other hand, we also know that $c_j^{(1)}=m_{p_j^{(1)}}(B)$  and 
\begin{equation}
c_j^{(i)}=c_j^{(i-1)}+m_{P_j^{(i)}}(B)=m_{P_j^{(1)}}(B)+\ldots+m_{P_j^{(i)}}(B)
\label{formulacj}
\end{equation}

Thus, 
\begin{equation}
d_j^{(i)}=c_j^{(i-1)}+m_{P_j^{(i)}}(B)-i\cdot m=c_j^{(i)}-i\cdot m=m_{P_j^{(1)}}(B)+\ldots+m_{P_j^{(i)}}(B)-i\cdot m 
\label{formuladjcj}
\end{equation}

In particular, 

\begin{equation}
d_j^{(i)} \leq i\cdot (m_{P_j^{(1)}}(B)-m)\leq i\cdot 2m
\label{bounddj}
\end{equation}

And the condition $d_j^{(a_j)}=0$ implies 
\begin{equation}
m_{P_j^{(1)}}(B)+\ldots+m_{P_j^{(a_j)}}(B)=a_j\cdot m
\label{formulasummults}
\end{equation}

Therefore, whenever $C$ is smooth at the base point $P_j^{(1)}$, using Noether's formula \cite{fulton} we obtain
\begin{equation}
I_{P_j^{(1)}}(B,C)=a_j\cdot m
\label{formulaintBC}
\end{equation}
where $I_{P_j^{(1)}}(B,C)$ denotes the intersection multiplicity of $B$ and $C$ at the point $P_j^{(1)}$.

Lastly, we have $b_j^{(i)}=i$ for all $j=1,\ldots k$ and $i=1,\ldots, a_j$.

\subsection{The (unique) multiple cubic}
\label{multiplecubic}

The cubic $C$ is smooth at every base point of $\mathcal{P}$ if and only if $\pi$ restricts to an isomorphism $E\simeq C$. This implies any $\pi-$exceptional curve must be either a multisection or a component of $F$. 

We also prove a partial converse of this statement:

\begin{lema}
For any index $m$ and any type of fiber we have
\[
d_j^{(1)}>0 \Rightarrow m_{P_j^{(1)}}(mC)=m
\]
That is, if the exceptional curve $E_j^{(1)}$ appears as a component in $F$ (with multiplicity $d_j^{(1)}>0$)  then $C$ is smooth at the point $P_j^{(1)}$. 
\label{Csmoothdj}
\end{lema}

\begin{proof}
If the exceptional curve $E_j^{(1)}$ appears as a component in $F$, then $mE_j^{(1)}$ cannot appear as a component of the multiple fiber $mE$. Hence $m_{P_j^{(1)}}(mC)-m=0$.
\end{proof}

\begin{cor}
If $C$ is singular at a base point $P_j^{(1)}$, then $m_{P_j^{(1)}}(B)=m$. Moreover, at the point $P_j^{(1)}$ the curve $B$ consists of a single component (branch) with multiplicity $m$.
\end{cor}

\begin{proof}
It follows from Lemma \ref{Csmoothdj} that if $C$ is singular at a base point $P_j^{(1)}$, then $E_j^{(1)}$ is not a component of $F$, hence $d_j^{(1)}=0$, which further implies $m_{P_j^{(1)}}(B)=m$. The last statement is obvious, otherwise one would need to blow-up more than one point lying in $E_j^{(1)}$ in order to separate $\mathcal{P}$.
\end{proof}

Since we are working over a field of characteristic zero, the unique multiple fiber $mE$ can only be of multiplicative type, i.e. of type $I_n$. If $n\leq 3$, then $mE$ can be realized as the strict transform (under $\pi$) of the unique multiple cubic $mC$. But if $n> 3$, then, necessarily, $C$ must be singular at a base point of $\mathcal{P}$. In other words, 

\begin{lema}
If $Y$ contains a multiple fiber of type $I_n$ with $4 \leq n \leq 9$, then $C$ is singular at a base point of $\mathcal{P}$.
\label{Cnotsingular}
\end{lema}

\begin{proof}
If $C$ is smooth at every base point of $\mathcal{P}$, then the corresponding multiple fiber on $Y$ is given by
\[
m\overline{C}+m\cdot \sum_{i,j} (m_{P_j^{(i)}}(C)-1)E_j^{(i)}=m\overline{C}
\] 
where $\overline{C}$ is the strict transform of $C$ under $\pi$ and $m_{P_j^{(i)}}(C)=1$ is the multiplicity of $P_j^{(i)}$ on the strict transform of $C$ under the blow-up of $P_j^{(1)},\ldots,P_j^{(i-1)}$. That is,  the multiple fiber of $Y$ is simply given by the strict transform of $mC$. But each of the fibers $I_n (4 \leq n \leq 9)$ have at least four components and hence the corresponding multiple fiber cannot be realized as strict transforms of a multiple cubic in the plane, a contradiction.
\end{proof}

When $C$ is singular at a base point of $\mathcal{P}$, it is also useful and interesting to understand how singular it can be.

\begin{prop}
For any index $m$ we have that $lct(\mathbb{P}^2,mC)=\frac{1}{m}$.
\label{lctmc}
\end{prop}

\begin{proof}
If $C$ is irreducible, then there is nothing to prove. Otherwise, we claim that $C$ consists of either a conic and a line intersecting it transversally or three distinct lines in general position (i.e. not concurrent at a point).

Clearly $C$ cannot be non-reduced so we must exclude the following three cases:
\begin{enumerate}[(a)]
\item a cusp
\item a conic and a tangent line
\item three concurrent lines
\end{enumerate}

Because the unique multiple fiber of $Y$ can only be of type $I_n, n\leq 9$, in any of the above cases the singular point of $C$ must be a base point of the pencil $\mathcal{P}$. Moreover, since (c) can be obtained as soon as one blow-up (of the tangency point) is performed in a cubic as in (b) and, in turn, (b) can be obtained as soon as one blow-up (of the cusp) is performed in a cubic as in (a), it suffices to consider only case (c). But blowing-up the concurrency point yields a component with multiplicity $2m$, which is an absurd. Such component is not a multisection of degree $m$ and it cannot be a component in the multiple fiber either.
\end{proof}

\begin{prop}
If $F$ is of type $IV^*$ or $III^*$, then $C$ is singular at most one base point of $\mathcal{P}$.
\end{prop}

\begin{proof}
From the proof of Proposition \ref{lctmc} we know that $C$ is reduced and either $C$ is irreducible or it consists of a conic and a line intersecting transversally or  three lines in general position. Moreover,  from the classification in \cite{list} we also know that if $F$ is of type $IV^*$ (resp. $III^*$), then the multiple fiber $mE$ can only be of type $I_0,I_1,I_2$ or $I_3$ (resp. $I_0,I_1$ or $I_2$). Now, if $C$ were singular at more than one base point of $\mathcal{P}$, then $C$ would  necessarily consist of a conic and a line intersecting transversally and the two intersecting points would be base points of $\mathcal{P}$. But then we would need to  blow-up each of those two points at least twice, which would yield at least two more components in the multiple fiber. That is, $mE$ would be of type $I_n$ with $n\geq 4$, a contradiction.
\end{proof}

\begin{rmk}
If $F$ is of type $II^*$, then $C$ must be smooth at every base point of $\mathcal{P}$, because $E$ is of type $I_0$ or $I_1$ \cite{list}. In particular, $E$ (hence $C$) is irreducible, $\pi$ restricts to an isomorphism $E\simeq C$ and $C$ cannot be singular at any base point of $\mathcal{P}$. 
\label{rmkCiistar}
\end{rmk}

\subsection{The singularities of $B$ and the log canonical threshold} We are now ready to study the singularities of the curve $B$ in terms of the type of the fiber $F$. We investigate the multiplicities of $B$ at the base points of $\mathcal{P}$ and we compute bounds for the log canonical threshold of the pair $(\mathbb{P}^2,B)$ by establishing some relations between the log canonical thresholds of the pairs $(Y,F)$ and $(\mathbb{P}^2,B)$. 

We begin by proving the following Lemma:

\begin{lema}
If $\mathcal{P}$ does not contain an infinitely near point as a base point (i.e $a_j=1$ for all $j=1,\ldots,k$), then $k=9$ and 
\[
F=\overline{F}+\sum_{j=1}^{9} (m_{P_j^{(1)}}(B)-m)E_j^{(1)}=\overline{F}
\] 
\label{induced}
\end{lema}

\begin{proof}
If $a_j=1$ for all $j=1,\ldots,k$, then it is clear that $k=9$, since $a_1+\ldots+a_k=9$. Moreover, $0=d_j^{(a_j)}=d_j^{(1)}=m_{P_j^{(1)}}(B)-m$ for all $j=1,\ldots,9$.
\end{proof}

\begin{cor}
Let $S_F$ denote the sum of all the multiplicities of the components of a fiber $F$ and let $n_F$ denote the number of its components. If either $S_F>3m$ or $n_F>3m$, then $\mathcal{P}$ must contain an infinitely near point as a base point. In particular, there exists some $1\leq j\leq k$ so that $a_j>1$ and $d_j^{(1)}\geq 1$.
\label{inftybadfiber}
\end{cor}

\begin{proof}
If $\mathcal{P}$ does not contain an infinitely near point as a base point, then Lemma \ref{induced} tells us $F$ is the strict transform of a member of $\mathcal{P}$, which implies both  $S_F\leq 3m$ and $n_F \leq 3m$. 
\end{proof}

\begin{cor}
Using the same notations as in Corollary \ref{inftybadfiber}, if a fiber $F$ is such that $S_F>3m$ or $n_F>3m$, then there exists a base point $P_j^{(1)}$ in $\mathcal{P}$ such that $m_{P_j^{(1)}}\geq m+1$.
\end{cor}

\begin{proof}
By Corollary \ref{inftybadfiber} there exists some $j$ so that $d_j^{(1)}\geq 1$ and the result follows from the equality $d_j^{(1)}=m_{P_j^{(1)}}(B)-m$.
\end{proof}

We also prove the following:

\begin{lema}
Every base point $P_j^{(1)}$ of $\mathcal{P}$ is such that $m_{P_j^{(1)}}(B)\leq \min\{M_F+m, 3m\}$, where $M_F$ denotes the largest multiplicity of a component of $F$.
\label{mbasept}
\end{lema}

\begin{proof}
If follows from the fact that $B$ has degree $3m, d_j^{(1)}=m_{P_j^{(1)}}(B)-m$ and $d_j^{(1)}\leq M_F$.
\end{proof}

\begin{cor}
If $F$ is non-reduced and $m \leq M_F$, then every base point $P_j^{(1)}$ of $\mathcal{P}$ is such that $m_{P_j^{(1)}}\leq 2M_F$. 
\label{2mf}
\end{cor}

\begin{prop}
If $F$ is reduced, then $B$ is reduced and, 
\[
 \frac{1}{m+1} < lct(\mathbb{P}^2,B)=\min\left\{lct_{P_j^{(1)}}(\mathbb{P}^2,B) ,lct(Y,F)\right\}\leq lct(Y,F)\leq lct(Y,\overline{F})
\]
\label{lctreduced}
\end{prop}

\begin{proof}
We first show the equality. We have that $lct(\mathbb{P}^2,B)=\min_P\{ lct_P(\mathbb{P}^2,B)\}$, where $P$ runs over the singular points of $B$. But any singular point of $B$ is either a base point of $\mathcal{P}$ of it is not a base point and hence it must satisfy $lct_P(\mathbb{P}^2,B)=lct_P(Y,\overline{F})$. Moreover,  $lct_P(Y,\overline{F})=lct(Y,F)$, because either $F$ is of type $II,III$ or $IV$ and $F$ contains a unique singular point, namely (the strict transform of) $P$; or $F$ is of type $I_n, 1\leq n \leq 9$ and every singular point of $F$ is an ordinary node and we have that $lct_P(Y,\overline{F})=lct(Y,F)=1$.

Now, because $F$ is reduced we have
\[
 \frac{1}{m+1}\leq \frac{1}{2} < lct(Y,F)
\]

On the other hand, for any $j$ we can find some $i$ so that 
\[
lct_{P_j^{(1)}}(\mathbb{P}^2,B)=\frac{1+b^{(i)}}{c_j^{(i)}}=\frac{1+i}{c_j^{(i)}}\geq\frac{1+i}{i\cdot m+1}>\frac{1}{m+1}
\]
where we have used Equation (\ref{formuladjcj}) together with the fact that each $d_j^{(i)}$ satisfies $d_j^{(i)}\leq 1$. More precisely, 
\[
c_j^{(i)}= m_{P_j^{(1)}}+\ldots+m_{P_j^{(i)}}=i\cdot m+d_j^{(i)}\leq i\cdot m+1
\]

Finally, it is clear that (see e.g. \cite[Theorem 8.20]{singpairs}) $lct(Y,F)\leq lct(Y,\overline{F})$ because 
\[
F=\overline{F}+\sum_{i,j}d_j^{(i)}E_j^{(i)}
\]
\end{proof}

\begin{prop}
If $m>1$ and $F$ is reduced, then $lct(\mathbb{P}^2,B)>\frac{1}{m}$.
\label{lctreducedm1}
\end{prop}

\begin{proof}
It follows from the proof of Proposition \ref{lctreduced} by observing that $lct(Y,F)>\frac{1}{2}\geq \frac{1}{m}$ and
\[
\frac{1+i}{i\cdot m+1}>\frac{1}{m}
\]
\end{proof}

\begin{prop}
If  $F$ is non-reduced and $m\leq M_F$, where $M_F$ denotes the largest multiplicity of a component of $F$,  then 
\[
lct(Y,F)\leq lct(\mathbb{P}^2,B)\leq lct(Y,\overline{F})
\]
\label{lctnonreduced}
\end{prop}

\begin{proof}
If $F$ is non-reduced, then $\pi: Y \to \mathbb{P}^2$ is a log resolution of the pair $(\mathbb{P}^2,B)$ (see Definition \ref{logres}). It follows from Lemma \ref{complct} that
\begin{equation}
lct(\mathbb{P}^2,B)=\min_{i,j}\left\{\frac{1+b_j^{(i)}}{c_j^{(i)}},\frac{1}{M_B}\right\}\leq \frac{1}{M_B}=lct(Y,\overline{F})
\label{formulalctB}
\end{equation}
where $M_B$ denotes the largest multiplicity of a component of $B$.

If $lct(\mathbb{P}^2,B)=1/M_B$ there is nothing to prove, since $M_B\leq M_F$ and $lct(Y,F)=1/M_F$. 

Thus, assume there exists  some $i$ and some $j$ such that 
\[
lct(\mathbb{P}^2,B)=\frac{1+b_j^{(i)}}{c_j^{(i)}}< \frac{1}{M_F}\leq \frac{1}{M_B}
\]

If $i=1$, then
\[
\frac{1+b_j^{(i)}}{c_j^{(i)}}=\frac{2}{m_{P_j^{(1)}}}< \frac{1}{M_F} \iff m_{P_j^{(1)}} > 2 M_F
\]
which contradicts Corollary \ref{2mf}.

Similarly, if $i=2$, then 
\[
\frac{1+b_j^{(i)}}{c_j^{(i)}}=\frac{3}{m_{P_j^{(1)}}+m_{P_j^{(2)}}}< \frac{1}{M_F} \iff m_{P_j^{(1)}}+m_{P_j^{(2)}} > 3 M_F
\]
but $m_{P_j^{(1)}}+m_{P_j^{(2)}}=d_j^{(2)}+2m\leq M_F+2m\leq 3M_F$

Otherwise, using Equation \ref{formuladjcj}, we can write $c_j^{(i)}=d_j^{(i)}+i\cdot m$. Then,
\begin{eqnarray*}
\frac{1+b_j^{(i)}}{c_j^{(i)}}=\frac{1+b_j^{(i)}}{d_j^{(i)}+i\cdot m}< \frac{1}{M_F} &\iff& M_F (1+ b_j^{(i)})< d_j^{(i)}+i\cdot m\\
&\iff& M_F (1+ i)< d_j^{(i)}+i\cdot m
\end{eqnarray*}
which is a contradiction because $M_F\geq d_j^{(i)}$ and $M_F\geq m$.
\end{proof}

\begin{rmk}
Note that Equation (\ref{formulalctB}) in the proof of Proposition \ref{lctnonreduced} holds for any index $m$. In particular, if $F$ is of type $I_n^*,II^*,III^*$ or $IV^*$, then we also have that (see e.g. \cite{lct})
\[
\frac{1}{m_{P_{j_{max}}^{(1)}}}\leq lct(\mathbb{P}^2,B) 
\]
 where $m_{P_{j_{max}}^{(1)}}\doteq \max_j m_{P_j^{(1)}}(B)$.
\end{rmk}

Then Propositions \ref{notreduced} and \ref{mult3} allow us to further prove:

\begin{prop}
For any index $m$ we have $lct(Y,\overline{F})\leq 2lct(Y,F)$.
\label{lctstrict}
\end{prop}

\begin{proof}
By contradiction, assume $1/M_B>2lct(Y,F)$. If $F$ does not contain a component with multiplicity $\geq 3$, then $2lct(Y,F)\geq 1$ and we conclude $M_B<1$, a contradiction. If $F$ is of type $III^*$ or $IV^*$, then $B$ must be reduced (i.e., $M_B=1$) and if $F$ is of type $II^*$, then we conclude $M_B<3$, contradicting Propositions \ref{notreduced} and \ref{mult3}.
\end{proof}

\begin{rmk}
Note that when $F$ is of type $II^*,III^*$ or $IV^*$, then the inequality $1/M_B\leq 2lct(Y,F)$ implies Propositions \ref{notreduced} and \ref{mult3}.
\end{rmk}

In particular, combining Propositions \ref{lctreduced}, \ref{lctnonreduced} and \ref{lctstrict} we obtain:

\begin{cor}
For any index $m$ we have $lct(\mathbb{P}^2,B)\leq 2 lct(Y,F)$.
\label{lctB2lctF}
\end{cor}

\section{The curve $B$ when $F$ is of type $II^*,III^*$ or $IV^*$}
\label{iistariiistar}

We will now assume $m=2$ and $\mathcal{P}$ is a Halphen pencil of index two and we will only consider fibers $F$ of type $II^*,III^*$ and $IV^*$. The results obtained in the previous sections allow us to completely characterize which sextics can and which cannot yield these types of fiber. 

We adopt the same notations as in Section \ref{hp} and our strategy can be summarized as follows: given  $F$ we know the number of its components $n_F$. Assuming we also know the number $n_B$ of components of $B$ we can compute $k$, the number of base points \footnote{not counting infinitely near points} in $B$, from Equation \ref{numbercomp}, which we recall next 
\[
n_F=n_B+9-k-n_{E\backslash C}
\] 
and Lemma \ref{nEC}, where $n_{E\backslash C}$ denotes the difference between the number of components of $E$ and the number of components of $C$.

There are exactly $k-n_{E\backslash C}$ disjoint chains of rational curves in $F$ as in Figure \ref{chains}. And, moreover, together with the strict transform of $B$ under $\pi$ these are all the components of $F$. Thus, analyzing how the dual graph of $F$ must look like we can decide whether the components coming from $B$ and these chains could possibly yield the given fiber. 

The desired configuration of rational curves imposes restrictions on how the curves $B$ and $C$ can intersect and how the components of $B$ must intersect. Since $B$ and $C$ can only intersect at base points of $\mathcal{P}$ we can use Equation (\ref{formulaintBC}), which we also recall below
\[
I_{P_j^{(1)}}(B,C)=a_j\cdot m
\]

The desired configuration also imposes restrictions on the multiplicities $d_j^{(1)}$ of the components  $E_j^{(1)}$ appearing in $F$. Recall we have (Equation (\ref{formuladjcj}))
\[
d_j^{(1)}=m_{P_j^{(1)}}(B)-2
\]
In particular, we know what $m_{P_j^{(1)}}(B)$, the multiplicity of $B$ at the base point $P_j^{(1)}$, must be.

Moreover, every time we consider the dual graph of $F$ we can color the components coming from $B$ in blue and in black we indicate the missing components as in Lemma \ref{nodes}. Then the possible configurations are those where the components in black are arranged in exactly $k-n_{E\backslash C}$ disjoint chains as in Figure \ref{chains}. In particular, every black node can only be connected to at most two other black nodes (Lemma \ref{nodes}).

These considerations not only give us an algorithm to decide whether a sextic $B$ can yield the desired type of fiber but they also give us an algorithm to construct all possible examples for a given type of fiber, which we do in Section \ref{constructions}.

\subsection{Non-Examples}

\begin{prop}
If $F$ is of type $II^*$, then $B$ does not consist of any of the following curves:
\begin{enumerate}[(i)]
\item a line with multiplicity $6$
\item a line with multiplicity four and a double line 
\item  a triple line, a double line and another line
\end{enumerate}
\label{notiistar}
\end{prop}

\begin{proof}
If $B=6L$, where $L$ is a line, then Equation (\ref{numbercomp}) implies $k=1$.  And $\pi$ consists of blowing-up a single point $P_1$ in $B$ nine times. In particular, $L$ must be an inflection line of $C$, with the flex at $P_1$. But then blowing-up $\mathbb{P}^2$ at the nine base points $P_1^{(1)},\ldots,P_1^{(9)}$ of $\mathcal{P}$ would yield a (dual) configuration of rational curves in $F$, which is different from the (affine) Dynkin diagram $\tilde{E}_8$, the dual graph of a fiber of type $II^*$. 

If $B$ consists of a line with multiplicity $4$ and another line with multiplicity $2$ and $F$ is of type $II^*$,  then $n_F=9,B_B=2,n_{E\backslash C}=0$, hence $k=2$ and the only possible picture is the following one:

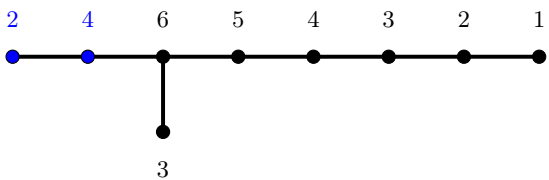
\begin{figure}[H]
\centering
\begin{tikzpicture}[line cap=round,line join=round,>=triangle 45,x=1.0cm,y=1.0cm]
\clip(-3.,-2.5) rectangle (5.5,1.);
\draw [line width=1.5pt] (-2.,0.)-- (5.,0.);
\draw [line width=1.5pt] (0.,0.)-- (0.,-1.);
\begin{scriptsize}
\draw [fill=black] (0.,-1.) circle (2.5pt);
\draw[color=black] (0.,-1.5) node {3};
\draw [fill=blue] (-2.,0.) circle (2.5pt);
\draw[color=blue] (-2.,.5) node {2};
\draw [fill=blue] (-1.,0.) circle (2.5pt);
\draw[color=blue] (-1.,.5) node {4};
\draw [fill=black] (0.,0.) circle (2.5pt);
\draw[color=black] (0.,.5) node {6};
\draw [fill=black] (1.,0.) circle (2.5pt);
\draw[color=black] (1.,0.5) node {5};
\draw [fill=black] (2.,0.) circle (2.5pt);
\draw[color=black] (2.,.5) node {4};
\draw [fill=black] (3.,0.) circle (2.5pt);
\draw[color=black] (3.,0.5) node {3};
\draw [fill=black] (4.,0.) circle (2.5pt);
\draw[color=black] (4,.5) node {2};
\draw [fill=black] (5.,0.) circle (2.5pt);
\draw[color=black] (5.,0.5) node {1};
\end{scriptsize}
\end{tikzpicture}
\end{figure}

In particular, up to relabeling, $\pi$ is the blow-up of $\mathbb{P}^2$ at the nine points $P_1^{(1)},\ldots,P_1^{(8)}, P_2^{(1)}$ and, moreover, the two lines do not intersect at a base point of $\mathcal{P}$. 

But $C$ must intersect each of the two lines at least once (at a base point) and $k=2$. So both lines must be inflection lines of $C$. In particular, we need to blow-up each of the base points at least three times in order to separate $\mathcal{P}$. That is, $a_1\geq3$ and $a_2\geq 3$, a contradiction.

Finally, if $B$ consists of a triple line, a double line and another line, it follows from Equation (\ref{numbercomp}) that $k=3$ and $\pi$ is the blow-up of $\mathbb{P}^2$ at the nine points
\[
P_1^{(1)},\ldots,P_1^{(a_1)}, P_2^{(1)},\ldots,P_2^{(a_2)},P_3^{(1)},\ldots,P_3^{(a_3)} 
\]

If we consider the dual graph of $F$, the picture must be one of the following:
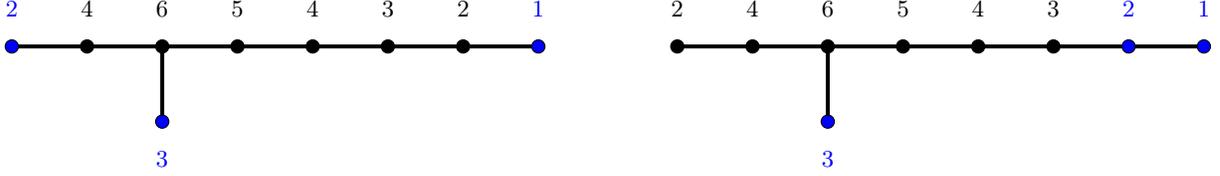
\[
\hspace*{-0.5cm}
\begin{matrix}
\begin{tikzpicture}[line cap=round,line join=round,>=triangle 45,x=1.0cm,y=1.0cm]
\clip(-3.,-2.5) rectangle (5.5,1.);
\draw [line width=1.5pt] (-2.,0.)-- (5.,0.);
\draw [line width=1.5pt] (0.,0.)-- (0.,-1.);
\begin{scriptsize}
\draw [fill=blue] (0.,-1.) circle (2.5pt);
\draw[color=blue] (0.,-1.5) node {3};
\draw [fill=blue] (-2.,0.) circle (2.5pt);
\draw[color=blue] (-2.,.5) node {2};
\draw [fill=black] (-1.,0.) circle (2.5pt);
\draw[color=black] (-1.,.5) node {4};
\draw [fill=black] (0.,0.) circle (2.5pt);
\draw[color=black] (0.,.5) node {6};
\draw [fill=black] (1.,0.) circle (2.5pt);
\draw[color=black] (1.,0.5) node {5};
\draw [fill=black] (2.,0.) circle (2.5pt);
\draw[color=black] (2.,.5) node {4};
\draw [fill=black] (3.,0.) circle (2.5pt);
\draw[color=black] (3.,0.5) node {3};
\draw [fill=black] (4.,0.) circle (2.5pt);
\draw[color=black] (4,.5) node {2};
\draw [fill=blue] (5.,0.) circle (2.5pt);
\draw[color=blue] (5.,0.5) node {1};
\end{scriptsize}
\end{tikzpicture}
&
\begin{tikzpicture}[line cap=round,line join=round,>=triangle 45,x=1.0cm,y=1.0cm]
\clip(-3.,-2.5) rectangle (5.5,1.);
\draw [line width=1.5pt] (-2.,0.)-- (5.,0.);
\draw [line width=1.5pt] (0.,0.)-- (0.,-1.);
\begin{scriptsize}
\draw [fill=blue] (0.,-1.) circle (2.5pt);
\draw[color=blue] (0.,-1.5) node {3};
\draw [fill=black] (-2.,0.) circle (2.5pt);
\draw[color=black] (-2.,.5) node {2};
\draw [fill=black] (-1.,0.) circle (2.5pt);
\draw[color=black] (-1.,.5) node {4};
\draw [fill=black] (0.,0.) circle (2.5pt);
\draw[color=black] (0.,.5) node {6};
\draw [fill=black] (1.,0.) circle (2.5pt);
\draw[color=black] (1.,0.5) node {5};
\draw [fill=black] (2.,0.) circle (2.5pt);
\draw[color=black] (2.,.5) node {4};
\draw [fill=black] (3.,0.) circle (2.5pt);
\draw[color=black] (3.,0.5) node {3};
\draw [fill=blue] (4.,0.) circle (2.5pt);
\draw[color=blue] (4,.5) node {2};
\draw [fill=blue] (5.,0.) circle (2.5pt);
\draw[color=blue] (5.,0.5) node {1};
\end{scriptsize}
\end{tikzpicture}
\end{matrix}
\]

If either the line with multiplicity three is not an inflection line of $C$ or the three lines are not concurrent at a base point, then we need to blow-up at least two points lying in the triple line and, in order to separate $\mathcal{P}$, we need to blow-up each of these points at least twice. This means that in $F$ we would have at least two disjoint chains  connected to the blue component of multiplicity $3$, where each chain has length $\geq 1$, excluding both pictures.

Thus, the only possibility is for the line with multiplicity three to be an inflection line of $C$ and for the three lines to be concurrent at a base point, which excludes the second picture. It is routine to check that this case does not yield the first picture either. 
\end{proof}

\begin{prop}
If $F$ is of type $III^*$, then $B$ does not consist of a double line and a (rational) quartic.
\label{not2-1,1-4}
\end{prop}

\begin{proof}
By contradiction, assume $B=2L+Q$, where $L$ is a line and $Q$ is a rational quartic. Then from Equation (\ref{numbercomp}) and Lemma \ref{nEC} we know that either $k=3$ ($n_{E\backslash C}=0$) or $k=2$ ($n_{E\backslash C}=1$). In any case, considering the dual graph of $F$, we conclude the picture must be the following one:

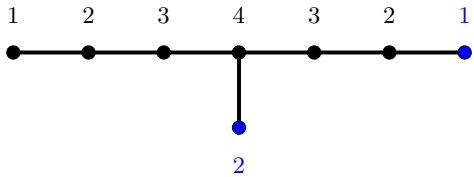
\begin{figure}[H]
\centering
\begin{tikzpicture}[line cap=round,line join=round,>=triangle 45,x=1.0cm,y=1.0cm]
\clip(-4.,-2.) rectangle (4.,1.);
\draw [line width=1.5pt] (-3.,0.)-- (3.,0.);
\draw [line width=1.5pt] (0.,0.)-- (0.,-1.);
\begin{scriptsize}
\draw [fill=blue] (0.,-1.) circle (2.5pt);
\draw[color=blue] (0.,-1.5) node {$2$};
\draw [fill=black] (-3.,0.) circle (2.5pt);
\draw[color=black] (-3.,.5) node {$1$};
\draw [fill=black] (-2.,0.) circle (2.5pt);
\draw[color=black] (-2.,.5) node {$2$};
\draw [fill=black] (-1.,0.) circle (2.5pt);
\draw[color=black] (-1.,.5) node {$3$};
\draw [fill=black] (0.,0.) circle (2.5pt);
\draw[color=black] (0.,.5) node {$4$};
\draw [fill=black] (1.,0.) circle (2.5pt);
\draw[color=black] (1.,0.5) node {$3$};
\draw [fill=black] (2.,0.) circle (2.5pt);
\draw[color=black] (2.,.5) node {$2$};
\draw [fill=blue] (3.,0.) circle (2.5pt);
\draw[color=blue] (3.,0.5) node {$1$};
\end{scriptsize}
\end{tikzpicture}
\end{figure}

Thus, up to relabeling, we have that $Q\cap L=\{P_1^{(1)}\}$ and we also have that either $a_1=7$ and $a_2=a_3=1$ ($k=3$) or $a_1=7$ and $a_2=2$ ($k=2$). Moreover, if the former holds then $d_1^{(7)}=d_2^{(1)}=d_3^{(1)}=0$ and if the latter holds, then $d_1^{(7)}=d_2^{(1)}=d_2^{(2)}=0$.

In any case we know from Equation (\ref{formulaintBC}) that we must have
\begin{equation}
I_{P_1^{(1)}}(B,C)=a_1\cdot m =14
\label{eq1}
\end{equation}
and since $I_{P_1^{(1)}}(B,C)=I_{P_1^{(1)}}(Q,C)+2I_{P_1^{(1)}}(L,C)$, it follows that if $k=3$, then either
\begin{itemize}
\item $I_{P_1^{(1)}}(Q,C)=12$ and $I_{P_1^{(1)}}(L,C)=1$ or
\item $I_{P_1^{(1)}}(Q,C)=10$ and $I_{P_1^{(1)}}(L,C)=2$ or
\item $I_{P_1^{(1)}}(Q,C)=8$ and $I_{P_1^{(1)}}(L,C)=3$
\end{itemize}

The conclusion is that $Q$ must be singular at $P_1^{(1)}$ because $Q$ must be rational and any other base point lying in $Q$ must be an ordinary double point. Note that there is only one chain of exceptional curves intersecting the reduced component in blue.

But if $Q$ is singular at $P_1^{(1)}$, then $m_{P_1^{(1)}}(B)>3$, hence $d_1^{(1)}>1$. From the picture, this further implies we have $d_1^{(1)}=2$ and  $E_1^{(1)}$ meets the strict transform of $Q$ under $\pi$. The only possibility then is for $E_1^{(6)}$ to appear with multiplicity one in $F$ and we must have $E_1^{(6)}\cdot E_1^{(7)}=2$. But the two curves meet transversally at a single point.  

Finally, if $k=2$ then there are no other base points lying in neither $Q$ nor $L$ (besides $P_1^{(1)}$). Thus, the intersection multiplicity of $Q$ (resp. $L$) and $C$ at $P_1^{(1)}$ is twelve (resp. three). But then
\[
I_{P_1^{(1)}}(B,C)=I_{P_1^{(1)}}(Q,C)+2I_{P_1^{(1)}}(L,C)=18
\]
contradicting Equation \ref{eq1}.
\end{proof}

\begin{prop}
If $F$ is of type $III^*$, then $B$ does not consist of a double line and two conics.
\label{not2-1,1-2,1-2}
\end{prop}

\begin{proof}
By contradiction, assume $B=Q_1+Q_2+2L$, where the $Q_i$ are conics and $L$ is a line. Then from Equation (\ref{numbercomp}) and Lemma \ref{nEC} we know that either $k=4$ ($n_{E\backslash C}=0$) or $k=3$ ($n_{E\backslash C}=1$). In any case the picture of the dual graph of $F$ must be:

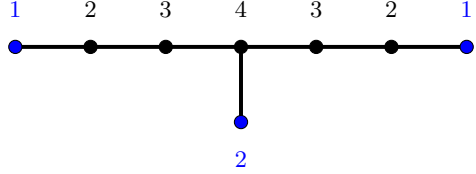
\begin{figure}[H]
\centering
\begin{tikzpicture}[line cap=round,line join=round,>=triangle 45,x=1.0cm,y=1.0cm]
\clip(-4.,-2.) rectangle (4.,1.);
\draw [line width=1.5pt] (-3.,0.)-- (3.,0.);
\draw [line width=1.5pt] (0.,0.)-- (0.,-1.);
\begin{scriptsize}
\draw [fill=blue] (0.,-1.) circle (2.5pt);
\draw[color=blue] (0.,-1.5) node {$2$};
\draw [fill=blue] (-3.,0.) circle (2.5pt);
\draw[color=blue] (-3.,.5) node {$1$};
\draw [fill=black] (-2.,0.) circle (2.5pt);
\draw[color=black] (-2.,.5) node {$2$};
\draw [fill=black] (-1.,0.) circle (2.5pt);
\draw[color=black] (-1.,.5) node {$3$};
\draw [fill=black] (0.,0.) circle (2.5pt);
\draw[color=black] (0.,.5) node {$4$};
\draw [fill=black] (1.,0.) circle (2.5pt);
\draw[color=black] (1.,0.5) node {$3$};
\draw [fill=black] (2.,0.) circle (2.5pt);
\draw[color=black] (2.,.5) node {$2$};
\draw [fill=blue] (3.,0.) circle (2.5pt);
\draw[color=blue] (3.,0.5) node {$1$};
\end{scriptsize}
\end{tikzpicture}
\end{figure}

Thus, up to relabeling , all the  curves $Q_1,Q_2$ and $L$ must intersect at the base point $P_1^{(1)}$ and we have that $L\cap Q_1=L\cap Q_2=\{P_1^{(1)}\}$ that is, $L$ is tangent to both $Q_1$ and $Q_2$ at $P_1^{(1)}$ so $Q_1$ and $Q_2$ must be tangent at $P_1^{(1)}$.

It is routine to check  that blowing-up $P_1^{(1)},\ldots,P_1^{(6)}$ does not yield the desired chain of exceptional curves. Therefore, we do not obtain a fiber of type $III^*$.

\end{proof}

\begin{prop}
If $F$ is of type $III^*$, then $B$ does not consist of a double conic and a double line.
\label{not2-2,2-1}
\end{prop}

\begin{proof}
By contradiction, assume $B=2Q+2L$, where $Q$ is a conic and $L$ is a line. Then from Equation (\ref{numbercomp}) and Lemma \ref{nEC} we know that either $k=3$ ($n_{E\backslash C}=0$) or $k=2$ ($n_{E\backslash C}=1$). In any case the picture of the dual graph of $F$ must be the following one:

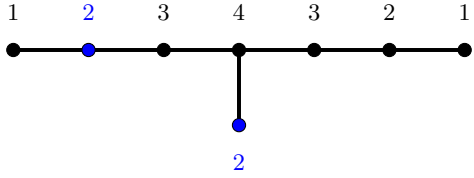
\begin{figure}[H]
\centering
\begin{tikzpicture}[line cap=round,line join=round,>=triangle 45,x=1.0cm,y=1.0cm]
\clip(-4.,-2.) rectangle (4.,1.);
\draw [line width=1.5pt] (-3.,0.)-- (3.,0.);
\draw [line width=1.5pt] (0.,0.)-- (0.,-1.);
\begin{scriptsize}
\draw [fill=blue] (0.,-1.) circle (2.5pt);
\draw[color=blue] (0.,-1.5) node {$2$};
\draw [fill=black] (-3.,0.) circle (2.5pt);
\draw[color=black] (-3.,.5) node {$1$};
\draw [fill=blue] (-2.,0.) circle (2.5pt);
\draw[color=blue] (-2.,.5) node {$2$};
\draw [fill=black] (-1.,0.) circle (2.5pt);
\draw[color=black] (-1.,.5) node {$3$};
\draw [fill=black] (0.,0.) circle (2.5pt);
\draw[color=black] (0.,.5) node {$4$};
\draw [fill=black] (1.,0.) circle (2.5pt);
\draw[color=black] (1.,0.5) node {$3$};
\draw [fill=black] (2.,0.) circle (2.5pt);
\draw[color=black] (2.,.5) node {$2$};
\draw [fill=black] (3.,0.) circle (2.5pt);
\draw[color=black] (3.,0.5) node {$1$};
\end{scriptsize}
\end{tikzpicture}
\end{figure}

Thus, up to relabeling we have $d_1^{(1)}\geq 1$ and $d_2^{(1)}\geq 1$. Since $k=2$ would imply $a_1=6$ and $a_2=3$ we see that we must have $k=3$ and $a_1=6,a_2=2,a_3=1$. Moreover,  $L$ must be tangent to $Q$ at $P_1^{(1)}$.  Otherwise, $E_1^{(1)}$ would appear with multiplicity $2$ in $F$, but the picture tells us such curve appears with multiplicity $3$ since $d_1^{(1)}=m_{P_1^{(1)}}(B)-2\geq 2$.

It is routine to check  that blowing-up $P_1^{(1)},\ldots,P_1^{(6)}$ does not yield the desired chain of exceptional curves. Thus, we do not obtain a fiber of type $III^*$.
\end{proof}

\begin{prop}
If $F$ is of type $III^*$, then $B$ does not consist of a double conic and two lines.
\label{not2-2,1-1,1-1}
\end{prop}

\begin{proof}
By contradiction, assume $B=2Q+L_1+L_2$, where $Q$ is a conic and each $L_i$ is a line. Then from Equation (\ref{numbercomp}) and Lemma \ref{nEC} we know that either $k=4$ ($n_{E\backslash C}=0$) or $k=3$ ($n_{E\backslash C}=1$). In any case, considering the dual graph of $F$, the picture must be: 

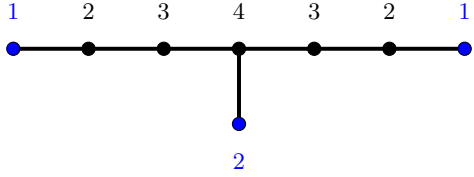
\begin{figure}[H]
\centering
\begin{tikzpicture}[line cap=round,line join=round,>=triangle 45,x=1.0cm,y=1.0cm]
\clip(-4.,-2.) rectangle (4.,1.);
\draw [line width=1.5pt] (-3.,0.)-- (3.,0.);
\draw [line width=1.5pt] (0.,0.)-- (0.,-1.);
\begin{scriptsize}
\draw [fill=blue] (0.,-1.) circle (2.5pt);
\draw[color=blue] (0.,-1.5) node {$2$};
\draw [fill=blue] (-3.,0.) circle (2.5pt);
\draw[color=blue] (-3.,.5) node {$1$};
\draw [fill=black] (-2.,0.) circle (2.5pt);
\draw[color=black] (-2.,.5) node {$2$};
\draw [fill=black] (-1.,0.) circle (2.5pt);
\draw[color=black] (-1.,.5) node {$3$};
\draw [fill=black] (0.,0.) circle (2.5pt);
\draw[color=black] (0.,.5) node {$4$};
\draw [fill=black] (1.,0.) circle (2.5pt);
\draw[color=black] (1.,0.5) node {$3$};
\draw [fill=black] (2.,0.) circle (2.5pt);
\draw[color=black] (2.,.5) node {$2$};
\draw [fill=blue] (3.,0.) circle (2.5pt);
\draw[color=blue] (3.,0.5) node {$1$};
\end{scriptsize}
\end{tikzpicture}
\end{figure}

Therefore, up to relabeling, we have that either $a_1=6$ and  $a_2=a_3=a_4=1$ ($k=4$) or $a_1=6,a_2=2$ and $a_3=1$ ($k=3$). If the former holds, then $d_2^{(1)}=d_3^{(1)}=d_4^{(1)}=0$ and if the latter holds, then $d_2^{(1)}=d_2^{(2)}=d_3^{(1)}=0$. But then, in any case, we must have that both $Q\cap L_1=\{P_1^{(1)}\}$ and  $Q\cap L_2=\{P_1^{(1)}\}$. That is, $L_1$ and $L_2$ are both tangent lines to $Q$ at $P_1^{(1)}$, which is an absurd.
\end{proof}

\begin{prop}
If $F$ is of type $III^*$, then $B$ does not consist of two double lines and a conic.
\label{not2-1,2-1,1-2}
\end{prop}

\begin{proof}
By contradiction, assume $B=2L_1+2L_2+Q$, where each $L_i$ is a line and $Q$ is a conic. Then from Equation (\ref{numbercomp}) and Lemma \ref{nEC} we know that either $k=4$ ($n_{E\backslash C}=0$) or $k=3$ ($n_{E\backslash C}=1$). In both cases the picture of the dual graph of $F$ must be one of the following:

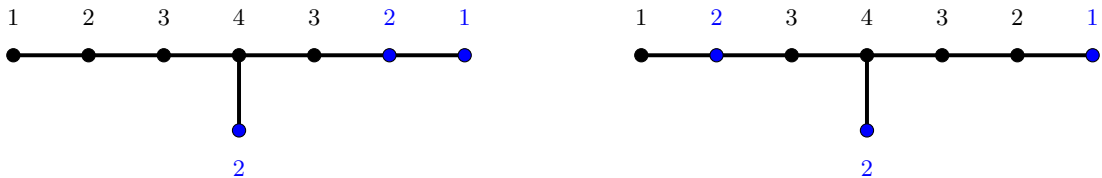
\[
\begin{matrix}

\begin{tikzpicture}[line cap=round,line join=round,>=triangle 45,x=1.0cm,y=1.0cm]
\clip(-4.,-2.) rectangle (4.,1.);
\draw [line width=1.5pt] (-3.,0.)-- (3.,0.);
\draw [line width=1.5pt] (0.,0.)-- (0.,-1.);
\begin{scriptsize}
\draw [fill=blue] (0.,-1.) circle (2.5pt);
\draw[color=blue] (0.,-1.5) node {$2$};
\draw [fill=black] (-3.,0.) circle (2.5pt);
\draw[color=black] (-3.,.5) node {$1$};
\draw [fill=black] (-2.,0.) circle (2.5pt);
\draw[color=black] (-2.,.5) node {$2$};
\draw [fill=black] (-1.,0.) circle (2.5pt);
\draw[color=black] (-1.,.5) node {$3$};
\draw [fill=black] (0.,0.) circle (2.5pt);
\draw[color=black] (0.,.5) node {$4$};
\draw [fill=black] (1.,0.) circle (2.5pt);
\draw[color=black] (1.,0.5) node {$3$};
\draw [fill=blue] (2.,0.) circle (2.5pt);
\draw[color=blue] (2.,.5) node {$2$};
\draw [fill=blue] (3.,0.) circle (2.5pt);
\draw[color=blue] (3.,0.5) node {$1$};
\end{scriptsize}
\end{tikzpicture}
&
\begin{tikzpicture}[line cap=round,line join=round,>=triangle 45,x=1.0cm,y=1.0cm]
\clip(-4.,-2.) rectangle (4.,1.);
\draw [line width=1.5pt] (-3.,0.)-- (3.,0.);
\draw [line width=1.5pt] (0.,0.)-- (0.,-1.);
\begin{scriptsize}
\draw [fill=blue] (0.,-1.) circle (2.5pt);
\draw[color=blue] (0.,-1.5) node {$2$};
\draw [fill=black] (-3.,0.) circle (2.5pt);
\draw[color=black] (-3.,.5) node {$1$};
\draw [fill=blue] (-2.,0.) circle (2.5pt);
\draw[color=blue] (-2.,.5) node {$2$};
\draw [fill=black] (-1.,0.) circle (2.5pt);
\draw[color=black] (-1.,.5) node {$3$};
\draw [fill=black] (0.,0.) circle (2.5pt);
\draw[color=black] (0.,.5) node {$4$};
\draw [fill=black] (1.,0.) circle (2.5pt);
\draw[color=black] (1.,0.5) node {$3$};
\draw [fill=black] (2.,0.) circle (2.5pt);
\draw[color=black] (2.,.5) node {$2$};
\draw [fill=blue] (3.,0.) circle (2.5pt);
\draw[color=blue] (3.,0.5) node {$1$};
\end{scriptsize}
\end{tikzpicture}
\end{matrix}
\]

If we have the first configuration (from the left to the right), then, up to relabeling, we have that $L_1\cap L_2=\{P_1^{(1)}\}$, $P_1^{(1)}\notin Q$ and we know from Equation (\ref{formulaintBC}) that the intersection multiplicity of $B$ and $C$ at $P_1^{(1)}$ is equal to 
\[
I_{P_1^{(1)}}(B,C)=2I_{P_1^{(1)}}(L_1,C)+2I_{P_1^{(1)}}(L_2,C)=a_1\cdot m =12
\]
Thus, it must be the case that  $I_{P_1^{(1)}}(L_1,C)=I_{P_1^{(1)}}(L_2,C)=3$. That is, both $L_1$ and $L_2$  are inflection lines of $C$ at $P_1^{(1)}$, which is an absurd.

Now, if we have the second configuration, then up to relabeling, $L_1\cap L_2=\{P_1^{(1)}\}$, we have $P_1^{(1)}\in Q$ and there are at most two base points lying in $Q$. Thus, one of the lines, say $L_1$, must be tangent to $Q$ at $P_1^{(1)}$ and since both lines cannot be tangent to $Q$ at $P_1^{(1)}$ we must have $P_2^{(1)}\in Q$, $I_{P_2^{(1)}}(Q,C)=2$ and $I_{P_1^{(1)}}(Q,C)= 4$. 

Now, Equation (\ref{formulaintBC}) gives the intersection multiplicity of $B$ and $C$ at $P_1^{(1)}$ is equal to 
\[
I_{P_1^{(1)}}(B,C)=I_{P_1^{(1)}}(Q,C)+2I_{P_1^{(1)}}(L_1,C)+2I_{P_1^{(1)}}(L_2,C)=a_1\cdot m =12
\]
and it follows that $I_{P_1^{(1)}}(L_1,C)= 3$ and $I_{P_1^{(1)}}(L_2,C)= 1$, since $L_2$ cannot be tangent to $C$ at $P_1^{(1)}$. It is routine to check  that blowing-up $P_1^{(1)},\ldots,P_1^{(5)}$ does not yield the desired chain of exceptional curves.

\end{proof}

\begin{prop}
If $F$ is of type $III^*$, then $B$ does not consist of three double lines.
\label{not2-1,2-1,2-1}
\end{prop}

\begin{proof}
By contradiction, assume $B=2L_1+2L_2+2L_3$, where each $L_i$ is a line. The picture of the dual graph of $F$ can only  be:

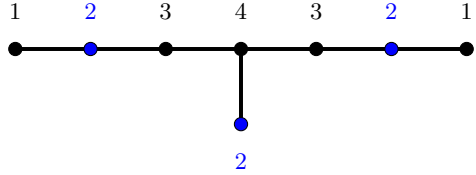
\begin{figure}[H]
\centering
\begin{tikzpicture}[line cap=round,line join=round,>=triangle 45,x=1.0cm,y=1.0cm]
\clip(-4.,-2.) rectangle (4.,1.);
\draw [line width=1.5pt] (-3.,0.)-- (3.,0.);
\draw [line width=1.5pt] (0.,0.)-- (0.,-1.);
\begin{scriptsize}
\draw [fill=blue] (0.,-1.) circle (2.5pt);
\draw[color=blue] (0.,-1.5) node {$2$};
\draw [fill=black] (-3.,0.) circle (2.5pt);
\draw[color=black] (-3.,.5) node {$1$};
\draw [fill=blue] (-2.,0.) circle (2.5pt);
\draw[color=blue] (-2.,.5) node {$2$};
\draw [fill=black] (-1.,0.) circle (2.5pt);
\draw[color=black] (-1.,.5) node {$3$};
\draw [fill=black] (0.,0.) circle (2.5pt);
\draw[color=black] (0.,.5) node {$4$};
\draw [fill=black] (1.,0.) circle (2.5pt);
\draw[color=black] (1.,0.5) node {$3$};
\draw [fill=blue] (2.,0.) circle (2.5pt);
\draw[color=blue] (2.,.5) node {$2$};
\draw [fill=black] (3.,0.) circle (2.5pt);
\draw[color=black] (3.,0.5) node {$1$};
\end{scriptsize}
\end{tikzpicture}
\end{figure}

Thus, the three lines $L_1,L_2$ and $L_3$ must be concurrent at the base point $P_1^{(1)}$. But if that is the case, then $E_1^{(1)}$ would appear with multiplicity $4$ in $F$.

\end{proof}

\begin{prop}
If $F$ is of type $III^*$, then $B$ does not consist of two double lines and two other lines.
\label{not2-1,2-1,1-1,1-1}
\end{prop}

\begin{proof}
By contradiction, assume $B=2L_1+2L_2+L_3+L_4$, where each $L_i$ is a line. If we consider the dual graph of $F$, the picture is the following one:

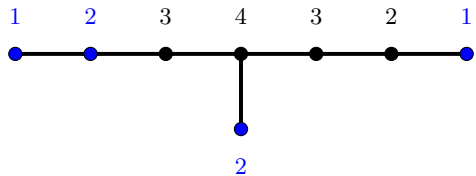
\begin{figure}[H]
\centering
\begin{tikzpicture}[line cap=round,line join=round,>=triangle 45,x=1.0cm,y=1.0cm]
\clip(-4.,-2.) rectangle (4.,1.);
\draw [line width=1.5pt] (-3.,0.)-- (3.,0.);
\draw [line width=1.5pt] (0.,0.)-- (0.,-1.);
\begin{scriptsize}
\draw [fill=blue] (0.,-1.) circle (2.5pt);
\draw[color=blue] (0.,-1.5) node {$2$};
\draw [fill=blue] (-3.,0.) circle (2.5pt);
\draw[color=blue] (-3.,.5) node {$1$};
\draw [fill=blue] (-2.,0.) circle (2.5pt);
\draw[color=blue] (-2.,.5) node {$2$};
\draw [fill=black] (-1.,0.) circle (2.5pt);
\draw[color=black] (-1.,.5) node {$3$};
\draw [fill=black] (0.,0.) circle (2.5pt);
\draw[color=black] (0.,.5) node {$4$};
\draw [fill=black] (1.,0.) circle (2.5pt);
\draw[color=black] (1.,0.5) node {$3$};
\draw [fill=black] (2.,0.) circle (2.5pt);
\draw[color=black] (2.,.5) node {$2$};
\draw [fill=blue] (3.,0.) circle (2.5pt);
\draw[color=blue] (3.,0.5) node {$1$};
\end{scriptsize}
\end{tikzpicture}
\end{figure}

In particular, up to relabeling, the two double lines together with one of the other lines, say $L_3$, must be concurrent at the point $P_1^{(1)}$ and, further, $C$ must be smooth at $P_1^{(1)}$, since $d_1^{(1)}>0$ (see Lemma \ref{Csmoothdj}).

In particular, Equation (\ref{formulaintBC}) gives us
\begin{equation}
I_{P_1^{(1)}}(B,C)=a_1\cdot m =10
\label{eq2}
\end{equation}
where $I_{P_1^{(1)}}(B,C)$ denotes the intersection multiplicity of $B$ and $C$ at the point $P_1^{(1)}$, but 
\[
I_{P_1^{(1)}}(B,C)=2I_{P_1^{(1)}}(L_1,C)+2I_{P_1^{(1)}}(L_2,C)+I_{P_1^{(1)}}(L_3,C)\leq 2\cdot 3 +1+1=8 
\]
contradicting Equation \ref{eq2}.

\end{proof}

\begin{prop}
If $F$ is of type $III^*$, then $B$ does not consist of a line with multiplicity four and a double line.
\label{not4-1,2-1}
\end{prop}

\begin{proof}
By contradiction, assume $B=4L+2L'$, where $L$ and $L'$ are two distinct lines.  It follows from Equation (\ref{numbercomp}) and Lemma \ref{nEC} that either $k=3$ ($n_{E\backslash C}=0$) or  $k=2$ ($n_{E\backslash C}=1$). Thus, if we look at the dual graph of $F$, the picture is:

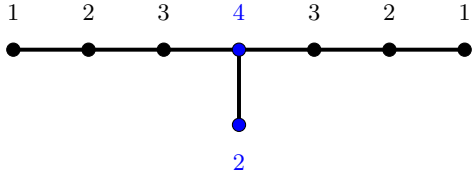
\begin{figure}[H]
\centering
\begin{tikzpicture}[line cap=round,line join=round,>=triangle 45,x=1.0cm,y=1.0cm]
\clip(-4.,-2.) rectangle (4.,1.);
\draw [line width=1.5pt] (-3.,0.)-- (3.,0.);
\draw [line width=1.5pt] (0.,0.)-- (0.,-1.);
\begin{scriptsize}
\draw [fill=blue] (0.,-1.) circle (2.5pt);
\draw[color=blue] (0.,-1.5) node {$2$};
\draw [fill=black] (-3.,0.) circle (2.5pt);
\draw[color=black] (-3.,.5) node {$1$};
\draw [fill=black] (-2.,0.) circle (2.5pt);
\draw[color=black] (-2.,.5) node {$2$};
\draw [fill=black] (-1.,0.) circle (2.5pt);
\draw[color=black] (-1.,.5) node {$3$};
\draw [fill=blue] (0.,0.) circle (2.5pt);
\draw[color=blue] (0.,.5) node {$4$};
\draw [fill=black] (1.,0.) circle (2.5pt);
\draw[color=black] (1.,0.5) node {$3$};
\draw [fill=black] (2.,0.) circle (2.5pt);
\draw[color=black] (2.,.5) node {$2$};
\draw [fill=black] (3.,0.) circle (2.5pt);
\draw[color=black] (3.,0.5) node {$1$};
\end{scriptsize}
\end{tikzpicture}
\end{figure}

and the two lines do not intersect at a base point of $\mathcal{P}$. In particular, up to relabeling, we have $P_1^{(1)},P_2^{(1)}\in L$ and $P_1^{(1)},P_2^{(1)}\notin L'$. So that the equality $d_j^{(1)}=m_{P_{j}^{(1)}}(B)-2$ implies both $E_1^{(1)}$ and $E_2^{(1)}$ appear with multiplicity $2$ in $F$, a contradiction.
\end{proof}

\begin{lema}
Let $C:c=0$ be a smooth cubic. Let $P\in C$ be a flex point and let $L:l=0$ be the corresponding inflection line. If $Q: q=0$ is a quartic such that $I_P(Q,L)=3$, then 
\[
q=lc'-l'c
\]
where $c'$ has degree three and $P$ does not lie in the line $l'=0$. 
\label{q1}
\end{lema}

\begin{proof}
We can choose coordinates in $\mathbb{P}^2$ so that $C$ is the cubic $y^2z-x(x-z)(x-\alpha\cdot z)=0$ for some $a \in \mathbb{C}\backslash \{0,1\}$. Then we can assume $P$ is the point $(0:1:0)$ and $L$ is the line $z=0$.

Since $I_P(Q,L)=3$ we have that $q(x,y,0)=\alpha x^4+\beta x^3y$ for some $\alpha,\beta \in \mathbb{C}$ with $\beta \neq 0$.

Letting $q'(x,y,z)$ be the polynomial 
\[
q+(\alpha x+\beta y+\gamma z)(y^2z-x(x-z)(x-a\cdot z))
\]
it follows that $q'(x,y,0)=0$.

Thus, there exists $c'$ a polynomial of degree three such that $q'=lc'$. Letting $l'$ be the line $\alpha x+\beta y+\gamma z=0$ it follows that
\[
q=lc'-l'c
\]
and $p$ is not a point in the line $l'=0$, since $\beta\neq 0$.
\end{proof}

\begin{cor}
Let $Q,C,L$ and $P$ be as in Lemma \ref{q1}. If $I_P(Q,C)=6$, then 
\[
c'=lq''-c
\]
where $q''$ has degree two and does not vanish on $P$.
\label{q1cor1}
\end{cor}

\begin{proof}
Since we can write $q=lc'-l'c$, the condition $I_P(Q,C)=6$ implies $I_p(C',C)=3$, where $C'$ is the cubic $c'=0$. Thus, $c'(x,y,0)=\alpha x^3$ for some $\alpha \neq 0$ and, up to a change of coordinates we can assume $\alpha=1$. Letting $c''$ be the polynomial $c'+c$, it follows that $c''(x,y,0)=0$. And we conclude there exists $q''$ a polynomial of degree two such that $c''=lq''$. By assumption, such polynomial cannot vanish on $P$.
\end{proof}

Combining Lemma \ref{q1} and Corollary \ref{q1cor1} we obtain:

\begin{cor}
Let $Q,C,L$ and $P$ be as in Lemma \ref{q1}. If $I_P(Q,C)=6$, then
\[
q=l^2q''-(l'+l)c
\]
\label{q1cor2}
\end{cor}

\begin{cor}
Let $C,L$ and $P$ be as in Lemma \ref{q1}. The space of quartics $Q$ such that $I_P(Q,L)=3$ and $I_P(Q,C)=6$ has dimension eight.
\label{q1cor3}
\end{cor}

\begin{proof}
It follows from the proofs of Lemma \ref{q1} and Corollary \ref{q1cor1} (or Corollary \ref{q1cor3}). The quartic $Q$ is determined by $9$ coefficients (the coefficients of $q''$ and the coefficients of $l'$).
\end{proof}

\begin{lema}
Let $C:c=0$ be a smooth cubic. Let $P\in C$ be a flex point and let $L:l=0$ be the corresponding inflection line. Choose three distinct points in $C$, say $P_1,P_2,P_3$ different from $p$. There does not exist a quartic $Q$ with nodes at $P_1,P_2,P_3$ and such that $I_P(Q,L)=3$ and $I_P(Q,C)=6$.
\label{q2}
\end{lema}

\begin{proof}
By Corollary \ref{q1cor3}, the space of quartics $Q$ satisfying the conditions $I_P(Q,L)=3$ and $I_P(Q,C)=6$ has dimension eight. Asking for $Q$ to have nodes at three points imposes nine more conditions. Alternatively, one can check that a quartic with equation in the form of Corollary \ref{q1cor2} does not have three nodes.
\end{proof}

\begin{prop}
If $F$ is of type $IV^*$, then $B$ does not consist of a double line and a rational quartic.
\label{notivstar1}
\end{prop}

\begin{proof}
If $B$ consists of a double line and a rational quartic and $F$ is of type $IV^*$, then the picture of the dual graph of $F$ is one of the following:

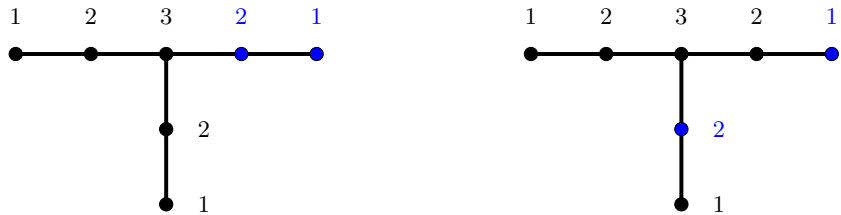
\[
\begin{matrix}
\begin{tikzpicture}[line cap=round,line join=round,>=triangle 45,x=1.0cm,y=1.0cm]
\clip(-3.5,-3.) rectangle (3.,1.);
\draw [line width=1.5pt] (-2.,0.)-- (2.,0.);
\draw [line width=1.5pt] (0.,0.)-- (0.,-2.);
\begin{scriptsize}
\draw [fill=black] (0.,-2.) circle (2.5pt);
\draw[color=black] (0.5,-2.) node {$1$};
\draw [fill=black] (0.,-1.) circle (2.5pt);
\draw[color=black] (0.5,-1.) node {$2$};
\draw [fill=black] (-2.,0.) circle (2.5pt);
\draw[color=black] (-2.,.5) node {$1$};
\draw [fill=black] (-1.,0.) circle (2.5pt);
\draw[color=black] (-1.,.5) node {$2$};
\draw [fill=black] (0.,0.) circle (2.5pt);
\draw[color=black] (0.,.5) node {$3$};
\draw [fill=blue] (1.,0.) circle (2.5pt);
\draw[color=blue] (1.,0.5) node {$2$};
\draw [fill=blue] (2.,0.) circle (2.5pt);
\draw[color=blue] (2.,.5) node {$1$};
\end{scriptsize}
\end{tikzpicture}
&
\begin{tikzpicture}[line cap=round,line join=round,>=triangle 45,x=1.0cm,y=1.0cm]
\clip(-3.5,-3.) rectangle (3.,1.);
\draw [line width=1.5pt] (-2.,0.)-- (2.,0.);
\draw [line width=1.5pt] (0.,0.)-- (0.,-2.);
\begin{scriptsize}
\draw [fill=black] (0.,-2.) circle (2.5pt);
\draw[color=black] (0.5,-2.) node {$1$};
\draw [fill=blue] (0.,-1.) circle (2.5pt);
\draw[color=blue] (0.5,-1.) node {$2$};
\draw [fill=black] (-2.,0.) circle (2.5pt);
\draw[color=black] (-2.,.5) node {$1$};
\draw [fill=black] (-1.,0.) circle (2.5pt);
\draw[color=black] (-1.,.5) node {$2$};
\draw [fill=black] (0.,0.) circle (2.5pt);
\draw[color=black] (0.,.5) node {$3$};
\draw [fill=black] (1.,0.) circle (2.5pt);
\draw[color=black] (1.,0.5) node {$2$};
\draw [fill=blue] (2.,0.) circle (2.5pt);
\draw[color=blue] (2.,.5) node {$1$};
\end{scriptsize}
\end{tikzpicture}
\end{matrix}
\]

In any case one can deduce the rational quartic must have three nodes from how the curves $B$ and $C$ must intersect. 

In the first picture (from the left to the right ), we can show that the quartic must intersect $C$ at a flex  point with multiplicity $6$ and the corresponding inflection line (of $C$) must intersect the quartic at the flex with multiplicity three; and at a fourth point which is not in $C$. Thus, all the missing base points must lie in the quartic, but not in the line. Since the quartic is reduced these points must be ordinary double points. In particular, up to relabeling, $a_1=6, a_2=a_3=a_4=1$.

 Similarly, in the second picture the quartic must intersect $C$ at a flex  point with multiplicity $4$ and the corresponding inflection line (of $C$) must intersect the quartic at the flex with multiplicity three; and at a fourth point where the quartic is tangent to $C$ with multiplicity two. Again, we conclude the missing base points must be ordinary double points. But then, up to relabeling, we would have $a_1=5,a_2=2,a_3=a_4=a_5=1$, contradicting the fact that $\sum a_i=9$.

Thus, the only possible picture is the first one. Since Equation (\ref{numbercomp}) tells us $C$ must be smooth, the result follows from Lemma \ref{q2}.
\end{proof}

\begin{prop}
If $F$ is of type $IV^*$, then $B$ does not consist of three double lines.
\label{notivstar2}
\end{prop}

\begin{proof}
If $B$ consists of three double lines and $F$ is of type $IV^*$, then the picture of the dual graph of $F$ is the following one:

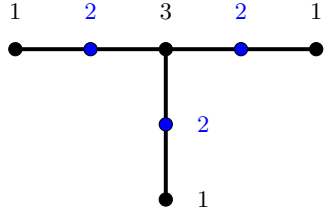
\begin{figure}[H]
\centering
\begin{tikzpicture}[line cap=round,line join=round,>=triangle 45,x=1.0cm,y=1.0cm]
\clip(-3.5,-3.) rectangle (3.,1.);
\draw [line width=1.5pt] (-2.,0.)-- (2.,0.);
\draw [line width=1.5pt] (0.,0.)-- (0.,-2.);
\begin{scriptsize}
\draw [fill=black] (0.,-2.) circle (2.5pt);
\draw[color=black] (0.5,-2.) node {$1$};
\draw [fill=blue] (0.,-1.) circle (2.5pt);
\draw[color=blue] (0.5,-1.) node {$2$};
\draw [fill=black] (-2.,0.) circle (2.5pt);
\draw[color=black] (-2.,.5) node {$1$};
\draw [fill=blue] (-1.,0.) circle (2.5pt);
\draw[color=blue] (-1.,.5) node {$2$};
\draw [fill=black] (0.,0.) circle (2.5pt);
\draw[color=black] (0.,.5) node {$3$};
\draw [fill=blue] (1.,0.) circle (2.5pt);
\draw[color=blue] (1.,0.5) node {$2$};
\draw [fill=black] (2.,0.) circle (2.5pt);
\draw[color=black] (2.,.5) node {$1$};
\end{scriptsize}
\end{tikzpicture}
\end{figure}

Thus, up to relabeling, the three lines must be concurrent at the base point $P_1^{(1)}$. But then $d_1^{(1)}=m_{P_1^{(1)}}(B)-2=4$ and $E_1^{(1)}$ appears with multiplicity four in $F$, which is an absurd.
\end{proof}

\begin{prop}
If $F$ is of type $IV^*$, then $B$ does not consist of a double conic and a double line.
\label{notivstar3}
\end{prop}

\begin{proof}
If $B$ consists of a double conic and a double line and $F$ is of type $IV^*$, it follows from Equation (\ref{numbercomp}) and Lemma \ref{nEC} that either $k=4 \, (n_{E\backslash C}=0)$, or $k=3 \, (n_{E\backslash C}=1)$, or $k=2 \, (n_{E\backslash C}=2)$. Since the picture of the dual graph of $F$ can only be the following one:

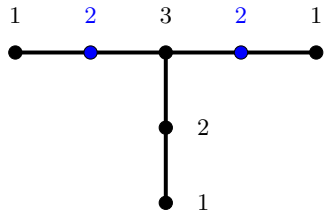
\begin{figure}[H]
\centering
\begin{tikzpicture}[line cap=round,line join=round,>=triangle 45,x=1.0cm,y=1.0cm]
\clip(-3.5,-3.) rectangle (3.,1.);
\draw [line width=1.5pt] (-2.,0.)-- (2.,0.);
\draw [line width=1.5pt] (0.,0.)-- (0.,-2.);
\begin{scriptsize}
\draw [fill=black] (0.,-2.) circle (2.5pt);
\draw[color=black] (0.5,-2.) node {$1$};
\draw [fill=black] (0.,-1.) circle (2.5pt);
\draw[color=black] (0.5,-1.) node {$2$};
\draw [fill=black] (-2.,0.) circle (2.5pt);
\draw[color=black] (-2.,.5) node {$1$};
\draw [fill=blue] (-1.,0.) circle (2.5pt);
\draw[color=blue] (-1.,.5) node {$2$};
\draw [fill=black] (0.,0.) circle (2.5pt);
\draw[color=black] (0.,.5) node {$3$};
\draw [fill=blue] (1.,0.) circle (2.5pt);
\draw[color=blue] (1.,0.5) node {$2$};
\draw [fill=black] (2.,0.) circle (2.5pt);
\draw[color=black] (2.,.5) node {$1$};
\end{scriptsize}
\end{tikzpicture}
\end{figure}

we must have $k=4$ and, up to relabeling, $a_1=4,a_2=2,a_3=2$ and $a_4=1$. Moreover, since the two blue components do not intersect, any intersection point between the line and the conic must be a base point of $\mathcal{P}$. But then for at least some $j=1,\ldots,4$ we must have $d_j^{(1)}=m_{P_j^{(1)}}(B)-2=2$, which implies $E_j^{(1)}$ must appear with multiplicity $2$ in $F$. Thus, we cannot possibly obtain the desired type of fiber.

\end{proof}

\subsection{All possible examples}

At last we can completely characterize the curve $B$ whenever $F$ is of type $II^*,III^*$ or $IV^*$.

When $F$ is of type $II^*$ the characterization is as follows:

\begin{thm}
If $F$ is of type $II^*$, then the sextic $B$ consists of one of the following (non-reduced) curves:
\begin{enumerate}[(i)]
\item a triple conic (Example \ref{iistartripleconic})
\item a nodal cubic and an inflection line, with the line taken with multiplicity three (Example \ref{iistartriplelinecubic})
\item two triples lines (Example \ref{iistartwotriplelines})
\item a conic and a tangent line, with the line taken with multiplicity four (Example \ref{iistarline4})
\item a line with multiplicity five and another line (Example \ref{iistarline5})
\end{enumerate}
\label{allpossibleiistar}
\end{thm}

\begin{proof}
A fiber of type $II^*$ has exactly two components with multiplicity three, two components with multiplicity two and only one component with multiplicity one. Thus, combining Propositions \ref{notreduced}, \ref{mult3} and \ref{notiistar}  we conclude the only possibilities for $B$ are the ones listed above.

If $B$ consists of a nodal cubic and a triple line, then the line must be an inflection line of the cubic because of the following reasoning. From Equation (\ref{numbercomp}) we know that $k=2$ so we need to blow-up exactly two points in $B$. But all the components of $F$ are supported in smooth curves so we also know that we have to blow-up the node. That is, the node is a base point. But then the second base point must lie in both the triple line and the cubic and that must be their only intersection point.

Similarly, in case $(iv)$ we can deduce how the components of $B$ must intersect from Equations (\ref{numbercomp}) and (\ref{formulaintBC}), the dual graph of $F$ and Lemma \ref{nodes}.
\end{proof}

When $F$ is of type $III^*$ we obtain the following description for the curve $B$:

\begin{thm}
If $F$ is of type $III^*$, then $B$ consists of one of the following curves:
\begin{enumerate}[(i)]
\item a double line, a cubic and another line (Example \ref{iiistar2-1,1-3,1-1})
\item a double conic and another conic (Example \ref{iiistar2-2,1-2})
\item a triple conic (Example \ref{exeiiistar})
\item two triple lines (Example \ref{iiistar3-1,3-1})
\item a triple line, a double line and another line (Examples \ref{iiistar3-1,2-1,1-1} and \ref{iiistar3-1,2-1,1-1conc} )
\item a triple line, a conic and a line (Example \ref{iiistar3-1,1-2,1-1})
\item a triple line and a cubic (Example \ref{iiistar3-1,1-3})
\item a conic and a line, with the line taken with multiplicity four (Example \ref{iiistar1-2,4-1}) 
\item a line with multiplicity four and two other lines  (Example \ref{iiistar4-1,1-1,1-1})
\end{enumerate}
\label{allpossibleiiistar}
\end{thm}

\begin{proof}
It follows from Propositions \ref{notreduced} and \ref{not2-1,1-4} through \ref{not4-1,2-1}.

\end{proof}

Finally, when $F$ is of type $IV^*$ the description for the curve $B$ is as follows:

\begin{thm}
If $F$ is of type $IV^*$, then $B$ consists of one of the following curves:
\begin{enumerate}[(i)]
\item a double conic and a conic (Example \ref{ivstar2-2,1-2})
\item a double line, a conic and two lines (Example \ref{ivstar2-1,1-2,1-1,1-1})
\item a double line, a cubic and a line (Example \ref{ivstar2-1,1-3,1-1})
\item a double line and two conics (Example \ref{ivstar2-1,1-2,1-2})
\item two double lines and two lines (Example \ref{ivstar2-1,2-1,1-1,1-1})
\item two double lines and a conic (Example \ref{ivstar2-1,2-1,1-2})
\item a double conic and two lines (Example \ref{ivstar2-2,1-1,1-1})
\item a triple conic (Example \ref{exeivstar})
\item a triple line, a conic and a line (Example \ref{ivstar3-1,1-2,1-1})
\item a triple line, a double line and another line (Example \ref{ivstar3-1,2-1,1-1}) 
\item a triple line and three lines (Example \ref{ivstar3-1,1-1,1-1,1-1})
\item a triple line and a cubic (Example \ref{ivstar3-1,1-3})
\end{enumerate}
\label{allpossibleivstar}
\end{thm}

\begin{proof}
Combining Propositions \ref{notreduced}, \ref{notivstar1},\ref{notivstar2} and \ref{notivstar3}we conclude the only possibilities for $B$ are the ones listed above.

\end{proof}

 In Sections \ref{constructions} and \ref{gd} we show the converse statements of Theorems \ref{allpossibleiistar}, \ref{allpossibleiiistar} and \ref{allpossibleivstar}  also hold. More precisely, we explicitly construct Halphen pencils of index two, $\lambda B + \mu(2C)=0$, where $B$ is one of the possible curves we have listed and the corresponding (non-multiple) fiber is of type $II^*,III^*$ or $IV^*$. 

\section{Constructions of Halphen Pencils of index two}
\label{constructions}

Following the algorithm outlined in the previous section, for each of the types of singular fibers that occur (see Proposition \ref{mainexe}) we can construct  explicit examples of a rational elliptic surface $f:Y\to \mathbb{P}^1$ of index $m=2$ having that type of singular fiber. In view of Proposition \ref{Halphen}, these are obtained by explicitly constructing the corresponding Halphen pencils $\mathcal{P}$.

We use the same notations as in Section \ref{hp}. In particular $\mathcal{P}$ is the pencil\footnote{with the exception of Example \ref{exelaface}} $\lambda B + \mu (2C)=0$, where $C$ is the unique cubic through the nine base points and $B$ corresponds to a (non-multiple) singular fiber $F$ of $f$ having the desired Kodaira type. We further assume $C$ is smooth at a base point and  use an additive notation when referring to both the components of a curve in the pencil $\mathcal{P}$ and the components of the corresponding fiber.

All the information we need to completely characterize $\mathcal{P}$, hence $Y$, can be encoded into some numerical data that we describe next.

From Proposition \ref{Halphen} we know that we obtain the rational elliptic surface $Y$ from $\mathcal{P}$ by blowing-up $\mathbb{P}^2$ at its nine base points 
$
P_1^{(1)}, \ldots, P_1^{(a_1)}, \ldots,P_k^{(1)}, \ldots, P_k^{(a_k)}
$. We call the $k-$uple $(a_1,\ldots,a_k)$ the \textbf{characteristic sequence} of $Y$ (and/or $\mathcal{P}$), where the choice of name is borrowed from \cite{laface}.

Now, the curve $B$ has $n_B$ components, where $n_B$ satisfies Equation (\ref{numbercomp}), and  since each of these components appears with a certain multiplicity, we can associate to it a \textbf{multiplicity sequence}  $
(m_1-d_1,\ldots,m_{n_B}-d_{n_B})$.  And we simply mean that we can write 
\[
B=m_1B_1+\ldots+ m_{n_B}B_{n_B}
\] where each $B_i$ is a reduced and irreducible curve of degree $d_i$ and $m_i$ is its multiplicity in $B$.

For instance, if $B$ consists of a double line and two conics, then $B$ has a multiplicity sequence $(2-1,1-2,1-2)$. It also has $(1-2,2-1,1-2)$ as a multiplicity sequence as well as $(1-2,1-2,2-1)$. Any choice is fine.

Once we choose a multiplicity sequence it is relevant to consider the intersection multiplicity of the components of $B$ and $C$ at a base point $P_j^{(1)}$ and record these numbers by arranging them in a matrix. 

Given a Halphen pencil $\mathcal{P}$ generated by $B=m_1B_1+\ldots+ m_{n_B}B_{n_B}$ and $2C$ we define its \textbf{intersection matrix} as the $n_B \times k$ matrix $A=(a_{ij})$ whose entries $a_{ij}$ are given precisely by the intersection multiplicities of $m_iB_i$ and $C$ at the point  $P_j^{(1)}$.

These three numerical data (the multiplicity and characteristic sequences and the intersection matrix) allows us to present all of our constructions\footnote{with the exception of Example \ref{exelaface}} in Tables   \ref{examplesi7i8} through \ref{examplesiistar} below.  A much more detailed geometric description of the constructions is given in the next section.

\begin{table}[H]
\centering
\hspace*{-1cm}
\begin{tabular}{|c|c|c|c|c|c|}
\hline
\bf{\makecell{Type \\ of \\ Fiber}}  & \boldmath{$lct(\mathbb{P}^2,B)$} & \bf{\makecell{Multiplicity \\ Sequence}} & \bf{\makecell{Characteristic \\ Sequence}} & \bf{\makecell{Intersection \\ Matrix}} & \bf{Example} 
\\
\hline \hline
&&&&&\\
\makecell{$I_7$} & $\frac{2}{3}$ & $\underbrace{(1-1,\ldots,1-1)}_{6}$ &  $(2,\underbrace{1,\ldots,1}_{7})$ & $\begin{pmatrix} 1&0&0&0&1&1&0&0\\2&0&0&1&0&0&0&0\\1&0&1&0&0&0&0&1\\0&0&0&0&1&0&1&1\\0&1&1&1&0&0&0&0\\0&1&0&0&0&1&1&0 \end{pmatrix}$ & \ref{examplei7}
\\
&&&&&\\\cline{1-6}
&&&&&\\
\makecell{$I_8$} & $\frac{2}{3}$  & $\underbrace{(1-1,\ldots,1-1)}_{6}$ &  $(2,2,\underbrace{1,\ldots,1}_{5})$ & $\begin{pmatrix} 1&1&0&0&1&0&0\\1&0&0&1&0&1&0\\2&0&1&0&0&0&0\\ 0&0&0&0&1&1&1\\ 0&2&0&0&0&0&1\\ 0&1&1&1&0&0&0 \end{pmatrix}$ & \ref{examplei8}
\\
&&&&&\\\cline{1-6}
&&&&&\\
\makecell{$I_9$} & $\frac{2}{3}$  & $\underbrace{(1-1,\ldots,1-1)}_{6}$ &  $(2,2,2,1,1,1)$ & $\begin{pmatrix} 2&0&1&0&0&0\\1&0&0&1&1&0\\1&1&0&0&0&1\\ 0&1&1&0&1&0\\ 0&0&2&0&0&1\\ 0&2&0&1&0&0 \end{pmatrix}$ & \ref{examplei9}
\\
&&&&&\\\cline{1-6}
\hline
\end{tabular}
\caption{Examples yielding a fiber of type $\protect I_7, I_8$ or $\protect I_9$}
\label{examplesi7i8}
\end{table}

\begin{table}[H]
\centering
\hspace*{-1cm}
\begin{tabular}{|c|c|c|c|c|}
\hline
\bf{\makecell{Type \\ of \\ Fiber}}  & \bf{\makecell{Multiplicity \\ Sequence}} & \bf{\makecell{Characteristic \\ Sequence}} & \bf{\makecell{Intersection \\ Matrix}} & \bf{Example} 
\\
\hline \hline
&&&&\\
\makecell{\\$I_1$} &  $(1-6)$ &  $\underbrace{(1,1,\ldots,1)}_{9}$ & $\begin{pmatrix} 2 & 2 & \cdots & 2 \end{pmatrix}$ & \ref{examplei1}
\\
&&&&\\\cline{1-5}
&&&&\\
\multirow{3}{*}{\makecell{\vspace{1.7cm}\\$I_2$}} & $(1-3,1-3)$ & $\underbrace{(1,1,\ldots,1)}_{9}$ & $\begin{pmatrix} 1 & 1 & \cdots & 1\\ 1 & 1 & \cdots & 1  \end{pmatrix}$  & \ref{i21-3,1-3}
\\
&&&&\\\cline{2-5}
&&&&\\
 &  $(1-4,1-2)$ &  $\underbrace{(1,1,\ldots,1)}_{9}$ & $\begin{pmatrix} 1 & 1 & \cdots & 1\\ 1 & 1 & \cdots & 1 \end{pmatrix}$ & \ref{i21-4,1-2}
\\
&&&&\\\cline{2-5}
&&&&\\
 &  $(1-5,1-1)$ &  $\underbrace{(1,1,\ldots,1)}_{9}$ & $\begin{pmatrix} 1 & 1 & \cdots & 1 \\1 & 1 & \cdots & 1 \end{pmatrix}$ & \ref{i21-5,1-1}
\\
&&&&\\\cline{1-5}
&&&&\\
\multirow{3}{*}{\makecell{\vspace{1.7cm}\\$I_3$}} &  $(1-2,1-2,1-2)$ & $\underbrace{(1,1,\ldots,1)}_{9}$ & $\begin{pmatrix} 1&1&1&1&1&1&0&0&0\\1&1&1&0&0&0&1&1&1\\0&0&0&1&1&1&1&1&1  \end{pmatrix}$  & \ref{i31-2,1-2,1-2}
\\
&&&&\\\cline{2-5}
&&&&\\
 &  $(1-4,1-1,1-1)$ &  $\underbrace{(1,1,\ldots,1)}_{9}$ & $\begin{pmatrix} 2&2&2&1&1&1&1&1&1\\0&0&0&1&1&1&0&0&0\\0&0&0&0&0&0&1&1&1  \end{pmatrix}$ & \ref{i31-4,1-1,1-1}
\\
&&&&\\\cline{2-5}
&&&&\\
 &  $(1-3,1-2,1-1)$ &  $\underbrace{(1,1,\ldots,1)}_{9}$ & $\begin{pmatrix} 2&1&1&1&1&1&1&1&0\\0&1&1&1&1&1&0&0&1\\0&0&0&0&0&0&1&1&1  \end{pmatrix}$ & \ref{i31-3,1-2,1-1}
\\
&&&&\\\cline{1-5}
&&&&\\
\multirow{3}{*}{\makecell{\vspace{1cm}\\$I_4$}} &  $(1-3,1-1,1-1,1-1)$ & $\underbrace{(1,1,\ldots,1)}_{9}$ & $\begin{pmatrix} 2&1&1&1&1&1&1&1&0\\0&1&1&0&0&0&0&0&1\\0&0&0&1&1&0&0&0&1\\ 0&0&0&0&0&1&1&1&0  \end{pmatrix}$  & \ref{i41-3,1-1,1-1,1-1}
\\
&&&&\\\cline{2-5}
&&&&\\
 &  $(1-2,1-2,1-1,1-1)$ &  $\underbrace{(1,1,\ldots,1)}_{9}$ & $\begin{pmatrix} 1&1&1&1&1&1&0&0&0\\1&1&1&1&0&0&1&1&0\\0&0&0&0&0&1&1&0&1\\0&0&0&0&1&0&0&1&1  \end{pmatrix}$ & \ref{i41-2,1-2,1-1,1-1}
\\
&&&&\\\cline{2-5}
\hline
\end{tabular}
\caption{Examples yielding a fiber of type $\protect I_n, n\leq 4$}
\label{examplesin}
\end{table}

\begin{table}[H]
\centering
\hspace*{-1cm}
\begin{tabular}{|c|c|c|c|c|c|}
\hline
\bf{\makecell{Type \\ of \\ Fiber}}  & \boldmath{$lct(\mathbb{P}^2,B)$} & \bf{\makecell{Multiplicity \\ Sequence}} & \bf{\makecell{Characteristic \\ Sequence}} & \bf{\makecell{Intersection \\ Matrix}} & \bf{Example} 
\\
\hline \hline
&&&&&\\
\makecell{$II$} & $\frac{5}{6}$ & $(1-6)$ &  $(\underbrace{1,\ldots,1}_{9})$ & $\begin{pmatrix} 2& 2 & \cdots & 2 \end{pmatrix}$ & \ref{exampleii}
\\
&&&&&\\\cline{1-6}
&&&&&\\
\multirow{3}{*}{\makecell{\vspace{1.5cm}\\$III$}} & $\frac{3}{4}$  & $(1-3,1-3)$ &  $(\underbrace{1,\ldots,1}_{9})$ & $\begin{pmatrix} 2&1&1&1&1&1&1&1&0\\0&1&1&1&1&1&1&1&2 \end{pmatrix}$ & \ref{iii1-3,1-3}
\\
&&&&&\\\cline{2-6}
&&&&&\\
 & $\frac{3}{4}$  & $(1-4,1-2)$ &  $(\underbrace{1,\ldots,1}_{9})$ & $\begin{pmatrix} 2&2&2&1&1&1&1&1&1\\0&0&0&1&1&1&1&1&1 \end{pmatrix}$ & \ref{iii1-4,1-2}
\\
&&&&&\\\cline{2-6}
&&&&&\\
 & $\frac{3}{4}$  & $(1-5,1-1)$ &  $(\underbrace{1,\ldots,1}_{9})$ & $\begin{pmatrix} 2&2&2&2&2&2&1&1&1\\0&0&0&0&0&0&1&1&1 \end{pmatrix}$ & \ref{iii1-5,1-1}
\\
&&&&&\\\cline{1-6}
&&&&&\\
\multirow{3}{*}{\makecell{\vspace{1.5cm}\\$IV$}} & $\frac{2}{3}$  & $(1-2,1-2,1-2)$ &  $(\underbrace{1,\ldots,1}_{9})$ & $\begin{pmatrix} 1&1&1&1&1&1&0&0&0\\1&1&1&0&0&0&1&1&1\\0&0&0&1&1&1&1&1&1 \end{pmatrix}$ & \ref{iv1-2,1-2,1-2}
\\
&&&&&\\\cline{2-6}
&&&&&\\
 & $\frac{2}{3}$  & $(1-4,1-1,1-1)$ &  $(\underbrace{1,\ldots,1}_{9})$ & $\begin{pmatrix} 2&2&2&1&1&1&1&1&1\\0&0&0&1&1&1&0&0&0\\0&0&0&0&0&0&1&1&1 \end{pmatrix}$ & \ref{iv1-4,1-1,1-1}
\\
&&&&&\\\cline{2-6}
&&&&&\\
 & $\frac{2}{3}$  & $(1-3,1-2,1-1)$ &  $(\underbrace{1,\ldots,1}_{9})$ & $\begin{pmatrix} 2&1&1&1&1&1&1&1&0\\0&1&1&1&1&1&0&0&1\\0&0&0&0&0&0&1&1&1 \end{pmatrix}$ & \ref{iv1-3,1-2,1-1}
\\
&&&&&\\\cline{2-6}
\hline
\end{tabular}
\caption{Examples yielding a fiber of type $\protect II,III$ or $\protect IV$}
\label{examplesii,iii,iv}
\end{table}

\begin{rmk}
All the examples listed in Tables \ref{examplesin} and \ref{examplesinstar} are such that $lct(\mathbb{P}^2,B)=lct(Y,F)$.
\end{rmk}

\begin{table}[H]
\centering
\hspace*{-1cm}
\begin{tabular}{|c|c|c|c|c|}
\hline
\bf{\makecell{Type \\ of \\ Fiber}}  & \bf{\makecell{Multiplicity \\ Sequence}} & \bf{\makecell{Characteristic \\ Sequence}} & \bf{\makecell{Intersection \\ Matrix}} & \bf{Example} 
\\
\hline \hline
&&&&\\
\multirow{4}{*}{\makecell{\vspace{2.5cm}\\$I_0^*$}} &  $(1-1,1-1,1-1,1-1,2-1)$ &  $\underbrace{(1,1,\ldots,1)}_{9}$ & $\begin{pmatrix} 1&1&1&0&0&0&0&0&0\\1&0&0&1&1&0&0&0&0\\0&1&0&1&0&1&0&0&0\\0&0&1&0&1&1&0&0&0\\0&0&0&0&0&0&2&2&2 \end{pmatrix}$ & \ref{iostarfivelines}
\\
&&&&\\\cline{2-5}
&&&&\\
&  $(2-2,1-2)$ &  $(2,2,2,1,1,1)$ & $\begin{pmatrix}2&2&2&2&2&2 \\2&2&2&0&0&0 \end{pmatrix}$ & \ref{iostar2-2,1-2}
\\
&&&&\\\cline{2-5}
&&&&\\
&  $(2-1,1-2,1-2)$ &  $(2,2,\underbrace{1,\ldots,1}_{5})$ & $\begin{pmatrix}2&2&0&0&0&0&2\\2&0&1&1&1&1&0\\0&2&1&1&1&1&0\end{pmatrix}$ & \ref{iostar2-1,1-2,1-2}
\\
&&&&\\\cline{2-5}
&&&&\\
&  $(2-1,1-2,1-1,1-1)$ &  $(2,\underbrace{1,\ldots,1}_{7})$ & $\begin{pmatrix}2&2&2&0&0&0&0&0\\2&0&0&0&1&1&1&1\\0&0&0&1&1&1&0&0&\\0&0&0&1&0&0&1&1\end{pmatrix}$ & \ref{iostar2-1,1-2,1-1,1-1}
\\
&&&&\\\cline{1-5}
&&&&\\
\makecell{$I_1^*$} & $(2-2,1-2)$ & $(2,2,3,1,1)$ & $\begin{pmatrix} 2&2&4&2&2\\2&2&2&0&0   \end{pmatrix}$  & \ref{i1star2-2,1-2}
\\
&&&&\\\cline{1-5}
&&&&\\
\makecell{$I_2^*$} & $(2-2,1-1,1-1)$ & $(1,1,3,3,1)$ & $\begin{pmatrix} 2&2&4&4&0\\0&0&2&0&1\\0&0&0&2&1   \end{pmatrix}$  & \ref{i2star2-2,1-1,1-1}
\\
&&&&\\\cline{1-5}
&&&&\\
\makecell{$I_3^*$} & $(2-1,2-1,1-1,1-1)$ & $(3,2,2,1,1)$ & $\begin{pmatrix} 4&2&0&0&0\\2&0&2&0&2\\0&2&0&1&0\\0&0&2&1&0   \end{pmatrix}$  & \ref{i3star2-1,2-1,1-1,1-1}
\\
&&&&\\\cline{1-5}
&&&&\\
\makecell{$I_4^*$} & $(2-1,2-1,1-1,1-1)$ & $(3,3,2,1)$ & $\begin{pmatrix} 2&4&0&0\\0&2&2&2\\1&0&2&0\\3&0&0&0   \end{pmatrix}$  & \ref{i4star2-1,2-1,1-1,1-1}
\\
&&&&\\\cline{1-5}
\hline
\end{tabular}
\caption{Examples yielding a fiber of type $\protect I_n^*, n\leq 4$}
\label{examplesinstar}
\end{table}

\begin{table}[H]
\centering
\hspace*{-1cm}
\begin{tabular}{|c|c|c|c|c|c|}
\hline
\bf{\makecell{Type \\ of \\ Fiber}} & \boldmath{$lct(\mathbb{P}^2,B)$} & \bf{\makecell{Multiplicity \\ Sequence}} & \bf{\makecell{Characteristic \\ Sequence}} & \bf{\makecell{Intersection \\ Matrix}} & \bf{Example} 
\\
\hline \hline
&&&&&\\
\multirow{10}{*}{\makecell{\vspace{7cm}\\$IV^*$}} & $\frac{1}{2}$ & $(2-1,2-1,1-1,1-1)$ &  $(3,2,1,1,1,1)$ & $\begin{pmatrix} 2&2&0&0&2&0\\ 2&0&2&0&0&2\\ 2&0&0&1&0&0\\0&2&0&1&0&0 \end{pmatrix}$ & \ref{ivstar2-1,2-1,1-1,1-1} 
\\
&&&&&\\\cline{2-6}
&&&&&\\
& $\frac{2}{5}$ &$(2-1,2-1,1-2)$ &  $(2,2,3,1,1)$ & $\begin{pmatrix} 2&0&2&2&0 \\0&2&2&0&2\\2&2&2&0&0  \end{pmatrix}$ & \ref{ivstar2-1,2-1,1-2} 
\\
&&&&&\\\cline{2-6}
&&&&&\\
& $\frac{3}{7}$ &$(2-2,1-1,1-1)$ &  $(4,2,1,1,1)$ & $\begin{pmatrix} 4&2&2&2&2\\3&0&0&0&0\\1&2&0&0&0  \end{pmatrix}$ & \ref{ivstar2-2,1-1,1-1} 
\\
&&&&&\\\cline{2-6}
&&&&&\\
& $\frac{1}{3}$ &$(3-2)$  & $(3,3,3)$ & $\begin{pmatrix} 6 & 6 & 6  \end{pmatrix}$ & \ref{exeivstar} 
\\
&&&&&\\\cline{2-6}
&&&&&\\
& $\frac{1}{3}$ &$(3-1,1-2,1-1)$  & $(2,2,3,1,1)$ & $\begin{pmatrix} 3&3&3&0&0\\1&0&3&1&1\\0&1&0&1&1 \end{pmatrix}$  & \ref{ivstar3-1,1-2,1-1}
\\
&&&&&\\\cline{2-6}
&&&&&\\
& $\frac{1}{3}$ &$(3-1,2-1,1-1)$  & $(3,3,1,1,1)$ & $\begin{pmatrix} 3&6&0&0&0\\0&0&2&2&2\\3&0&0&0&0 \end{pmatrix}$  & \ref{ivstar3-1,2-1,1-1} 
\\
&&&&&\\\cline{2-6}
&&&&&\\
& $\frac{1}{3}$ &$(3-1,1-1,1-1,1-1)$  & $(2,2,2,1,1,1)$ & $\begin{pmatrix} 3&3&3&0&0&0\\1&0&0&1&1&0\\0&1&0&1&0&1\\0&0&1&0&1&1 \end{pmatrix}$ & \ref{ivstar3-1,1-1,1-1,1-1} 
\\
&&&&&\\\cline{2-6}
&&&&&\\
& $\frac{1}{3}$ &$(3-1,1-3)$  & $(3,3,2,1)$ & $\begin{pmatrix} 3&3&3&0\\3&3&1&2  \end{pmatrix}$  & \ref{ivstar3-1,1-3}
\\
&&&&&\\\cline{2-6}
\hline
\end{tabular}
\caption{Examples yielding a fiber of type $\protect IV^*$}
\label{examplesivstar}
\end{table}

\begin{table}[H]
\centering
\begin{tabular}{|c|c|c|c|c|c|}
\hline
\bf{\makecell{Type \\ of \\ Fiber}} & \boldmath{$lct(\mathbb{P}^2,B)$} & \bf{\makecell{Multiplicity \\ Sequence}} & \bf{\makecell{Characteristic \\ Sequence}} & \bf{\makecell{Intersection \\ Matrix}} & \bf{Example} 
\\
\hline \hline
&&&&&\\
\multirow{10}{*}{\makecell{\vspace{9cm}\\$III^*$}} & $\frac{2}{5}$ & $(2-1,1-3,1-1)$ &  $(1,6,1,1)$ & $\begin{pmatrix} 0 & 6 & 0 & 0\\ 2 & 5 & 1 & 1\\ 0 & 1 & 1 & 1 \end{pmatrix}$ & \ref{iiistar2-1,1-3,1-1} 
\\
&&&&&\\\cline{2-6}
&&&&&\\
& $\frac{5}{12}$ &$(2-2,1-2)$ &  $(7,1,1)$ & $\begin{pmatrix} 8 & 2 & 2 \\ 6 & 0 & 0  \end{pmatrix}$ & \ref{iiistar2-2,1-2} 
\\
&&&&&\\\cline{2-6}
&&&&&\\
& $\frac{1}{3}$ &$(3-2)$ &  $(3,6)$ & $\begin{pmatrix} 6 & 12  \end{pmatrix}$ & \ref{exeiiistar} 
\\
&&&&&\\\cline{2-6}
&&&&&\\
& $\frac{1}{3}$ &$(3-1,3-1)$  & $(3,3,3)$ & $\begin{pmatrix} 6 & 0 & 3 \\ 0 & 6 & 3 \end{pmatrix}$ & \ref{iiistar3-1,3-1} 
\\
&&&&&\\\cline{2-6}
&&&&&\\
& $\frac{1}{3}$ &$(3-1,2-1,1-1)$  & $(5,2,1,1)$ & $\begin{pmatrix} 9 & 0 & 0 & 0\\ 0 & 2 & 2 & 2\\ 1 & 2 & 0 & 0 \end{pmatrix}$ & \ref{iiistar3-1,2-1,1-1} 
\\
&&&&&\\\cline{2-6}
&&&&&\\
& $\frac{1}{3}$ &$(3-1,2-1,1-1)$  & $(4,3,1,1)$ & $\begin{pmatrix} 3 & 6 & 0 & 0\\ 2 & 0 & 2 & 2\\ 3 & 0 & 0 & 0 \end{pmatrix}$  & \ref{iiistar3-1,2-1,1-1conc}
\\
&&&&&\\\cline{2-6}
&&&&&\\
& $\frac{1}{3}$ &$(3-1,1-2,1-1)$  & $(5,1,1,2)$ & $\begin{pmatrix} 6 & 0 & 0 & 3\\ 4 & 1 & 1 & 0\\ 0 & 1 & 1 & 1 \end{pmatrix}$  & \ref{iiistar3-1,1-2,1-1} 
\\
&&&&&\\\cline{2-6}
&&&&&\\
& $\frac{1}{3}$ &$(3-1,1-3)$  & $(1,5,3)$ & $\begin{pmatrix} 0 & 6 & 3 \\ 2 & 4 & 3  \end{pmatrix}$  & \ref{iiistar3-1,1-3}
\\
&&&&&\\\cline{2-6}
&&&&&\\
& $\frac{1}{4}$ &$(4-1,1-2)$  & $(4,3,2)$ & $\begin{pmatrix} 4 & 4 & 4 \\ 4 & 2 & 0  \end{pmatrix}$   & \ref{iiistar1-2,4-1}
\\
&&&&&\\\cline{2-6}
&&&&&\\
& $\frac{1}{4}$ &$(4-1,1-1,1-1)$ & $(3,3,2,1)$ & $\begin{pmatrix} 4 & 4 & 4 & 0 \\ 2 & 0 & 0 & 1 \\ 0 & 2 & 0 & 1 \end{pmatrix}$  & \ref{iiistar4-1,1-1,1-1} 
\\
&&&&&\\\cline{2-6}
\hline
\end{tabular}
\caption{All possible example yielding a fiber of type $\protect III^*$}
\label{examplesiiistar}
\end{table}

\begin{table}[H]
\centering
\hspace*{-1cm}
\begin{tabular}{|c|c|c|c|c|c|}
\hline
\bf{\makecell{Type of \\ Fiber}} & \boldmath{$lct(\mathbb{P}^2,B)$} & \bf{\makecell{Multiplicity \\ Sequence}} & \bf{\makecell{Characteristic \\ Sequence}} & \bf{\makecell{Intersection \\ Matrix}} & \bf{Example} 
\\
\hline \hline
&&&&&\\
\multirow{5}{*}{\makecell{\vspace{3.5cm}\\$II^*$}} & $\frac{1}{3}$ &$(3-2)$  & $(9)$ & $\begin{pmatrix} 18 \end{pmatrix}$  & \ref{iistartripleconic} 
\\
&&&&&\\\cline{2-6}
&&&&&\\
& $\frac{1}{3}$ &$(3-1,3-1)$  & $(6,3)$ & $\begin{pmatrix} 9 & 0 \\ 3 & 6 \end{pmatrix}$ & \ref{iistartwotriplelines} 
\\
&&&&&\\\cline{2-6}
&&&&&\\
& $\frac{1}{3}$ &$(3-1,1-3)$  & $(1,8)$ & $\begin{pmatrix} 0 & 9 \\ 2 & 7 \end{pmatrix}$ & \ref{iistartriplelinecubic} 
\\
&&&&&\\\cline{2-6}
&&&&&\\
& $\frac{1}{4}$ &$(4-1,1-2)$  & $(7,2)$ & $\begin{pmatrix} 8 & 4 \\ 0 & 6 \end{pmatrix}$ & \ref{iistarline4} 
\\
&&&&&\\\cline{2-6}
&&&&&\\
& $\frac{1}{5}$ &$(5-1,1-1)$  & $(4,5)$ & $\begin{pmatrix} 5 & 10 \\ 3 & 0 \end{pmatrix}$ & \ref{iistarline5} 
\\
&&&&&\\\cline{2-6}
\hline
&&&&&\\
\multirow{5}{*}{\makecell{\vspace{3cm}\\$IV^*$}} & $\frac{4}{9}$ & $(2-2,1-2)$ &  $(6,1,1,1)$ & $\begin{pmatrix} 6&2&2&2\\ 6&0&0&0 \end{pmatrix}$ & \ref{ivstar2-2,1-2}
\\
&&&&&\\\cline{2-6}
&&&&&\\
& $\frac{3}{7}$ & $(2-1,1-2,1-1,1-1)$ &  $(4,1,1,1,1,1)$ & $\begin{pmatrix} 4&0&0&0&0&2\\ 3&1&1&1&0&0 \\ 1&1&0&0&1&0 \\ 0&0&1&1&1&0 \end{pmatrix}$ & \ref{ivstar2-1,1-2,1-1,1-1}
\\
&&&&&\\\cline{2-6}
&&&&&\\
& $\frac{4}{9}$ & $(2-1,1-3,1-1)$ &  $(5,1,1,1,1)$ & $\begin{pmatrix} 6&0&0&0&0\\ 4&1&1&1&2 \\ 0&1&1&1&0 \end{pmatrix}$ & \ref{ivstar2-1,1-3,1-1}
\\
&&&&&\\\cline{2-6}
&&&&&\\
& $\frac{3}{7}$ & $(2-1,1-2,1-2)$ &  $(5,1,1,1,1)$ & $\begin{pmatrix} 6&0&0&0&0\\ 4&1&1&1&2 \\ 0&1&1&1&0 \end{pmatrix}$ & \ref{ivstar2-1,1-2,1-2}
\\
&&&&&\\\cline{2-6}
\hline
\end{tabular}
\caption{All possible examples yielding a fiber of type $\protect II^*$ and more examples yielding a fiber of type $\protect IV^*$}
\label{examplesiistar}
\end{table}

\section{Geometric Descriptions}
\label{gd}

We now provide a detailed geometric description of all the examples we constructed. Our exposition is organized as follows: All the examples of Halphen pencils yielding a fixed type of singular fiber are grouped together within a subsection. The title of each subsection indicates the type of fiber. Further, we adopt the following pattern: inside parenthesis (and next to the numeration of each example) we describe the components of the curve $B$. 

For completeness we also present some constructions that can be found in the literature and proper references are mentioned. However, the vast majority of the examples are new and these are marked with a dagger ($\dagger$). 

\subsection{Type $I_n$}

\begin{exe}[A rational sextic \cite{kimu1}]
Let $\mathcal{S}=\{P_1,\ldots,P_{10}\}$ be the set consisting of the ten nodes of an irreducible rational plane curve $B$ of degree $6$. Any choice of nine points in $\mathcal{S}$ defines a Halphen pencil of index two. Blowing-up the chosen nine points, the strict transform of $B$ becomes a singular fiber of type $I_1$. 
\label{examplei1}
\end{exe}

\begin{rmk} 
Note that if $F$ is of type $I_1$ (resp. $II$), then $B$ must be an irreducible rational curve of degree six with a single node (resp. cusp) and $9$ ordinary double points. Thus, there are $19$ (resp. $17$) (see e.g. \cite{kang},\cite{gradolato}) independent sextics in $\mathbb{P}^2$ so that blowing-up its nine nodes we obtain a rational elliptic surface of index two with a singular fiber of type $I_1$ (resp. $II$).
\label{curvefiberii}
\end{rmk}

\begin{exe}[Two cubics \cite{kimu1}]
Consider two nodal cubics $C_1$ and $C_2$. Denote the two nodes by $P_1$ and $P_2$. The two cubics intersect at nine other points $P_3,\ldots,P_9,P_{10},P_{11}$. 

Blowing-up the nine points $P_1,\ldots,P_9$ we obtain a rational elliptic surface of index two with a fiber of type $I_2$, namely the strict transform of $B=C_1+C_2$.
\label{i21-3,1-3}
\end{exe}

\begin{exen}[A quartic and a conic]
Consider a quartic $Q_1$ with three nodes $P_1,P_2$ and $P_3$ and a conic $Q_2$ that intersects $Q_1$ at other eight points $P_4,\ldots,P_9,P_{10},P_{11}$. If we blow-up the nine points $P_1,\ldots,P_9$ we obtain a rational elliptic surface of index two with a fiber of type $I_2$. Such fiber is given by the strict transform of $B=Q_1+Q_2$.
\label{i21-4,1-2}
\end{exen}

\begin{exen}[A quintic and a line]
Choose a quintic $Q$ with six nodes $P_1,\ldots,P_6$ and a line $L$ that intersects $Q$ at other five points $P_7,\ldots,P_9,P_{10},P_{11}$. We get a Halphen pencil of index two generated by $B=Q+L$ and the unique cubic through the nine points $P_1,\ldots,P_9$. Thus, blowing-up the nine base points we obtain a rational elliptic surface of index two. The strict transform of $B$ is a singular fiber of type $I_2$.
\label{i21-5,1-1}
\end{exen}

\begin{exe}[Three conics \cite{kimu1}]
Consider three conics $C_1,C_2$ and $C_3$. They intersect at twelve points. Among those, we can choose three points so that if we blow-up the remaining nine points $P_1,\ldots,P_9$, the strict transform of $B=C_1+C_2+C_3$ becomes a fiber of type $I_3$ in the corresponding rational elliptic surface. 

\label{i31-2,1-2,1-2}
\end{exe}

\begin{exen}[A quartic with three nodes and two lines]
Let $Q$ be a rational plane quartic with three nodes $P_1,P_2$ and $P_3$. Consider two lines $L_1$ and $L_2$ that intersect $Q$ at other eight points $P_4,\ldots,P_9,P_{10},P_{11}$ and denote by $P_{12}$ the point of intersection of $L_1$ and $L_2$, which we can assume to be distinct from the other $P_i$. Without loss of generality, we may further assume $P_{10} \in L_1$ and $P_{11}\in L_2$. Then, blowing-up the nine points $P_1,\ldots,P_9$ we obtain a rational elliptic surface of index two with a singular fiber of type $I_3$. Namely, the strict transform of $B=Q+L_1+L_2$. 
\label{i31-4,1-1,1-1}
\end{exen}

\begin{exen}[A cubic with a node, a conic and a line]
Let $C'$ be a nodal cubic (hence a rational curve). Let $Q$ be a conic and let $L$ be a line. Generically, these curves determine twelve distinct intersection points (of multiplicity $2$). We can choose three of these points so that if we blow-up the remaining nine points $P_1,\ldots,P_9$, the strict transform of $B=C'+Q+L$ is a singular fiber of type $I_3$ in the corresponding rational elliptic surface. 
\label{i31-3,1-2,1-1}
\end{exen}

\begin{exen}[A cubic with a node and three lines]
Consider a nodal cubic $C'$ and three distinct lines $L_1,L_2$ and $L_3$. Generically, these curves determine $13$ distinct intersection points: each line intersects the cubic at three points and any two pair of lines intersect at a point. We can then choose four of these points so that if we blow-up the remaining nine points $P_1,\ldots,P_9$, the strict transform of $B=C'+L_1+L_2+L_3$ becomes a singular fiber of type $I_4$. 
\label{i41-3,1-1,1-1,1-1}
\end{exen}

\begin{exen}[Two conics and two lines]
Consider two distinct conics $Q_1$ and $Q_2$ and a pair of distinct lines $L_1$ and $L_2$. There is a way of choosing four points, among their $13$ intersection points, so that  if we blow-up the remaining nine points the strict transform of $B=Q_1+Q_2+L_1+L_2$ is a singular fiber of type $I_4$ in the corresponding rational elliptic surface.  
\label{i41-2,1-2,1-1,1-1}
\end{exen}

\begin{exen}[Three concurrent lines and three other lines]
Choose six lines $L_1,\ldots,L_6$ intersecting as in the picture below

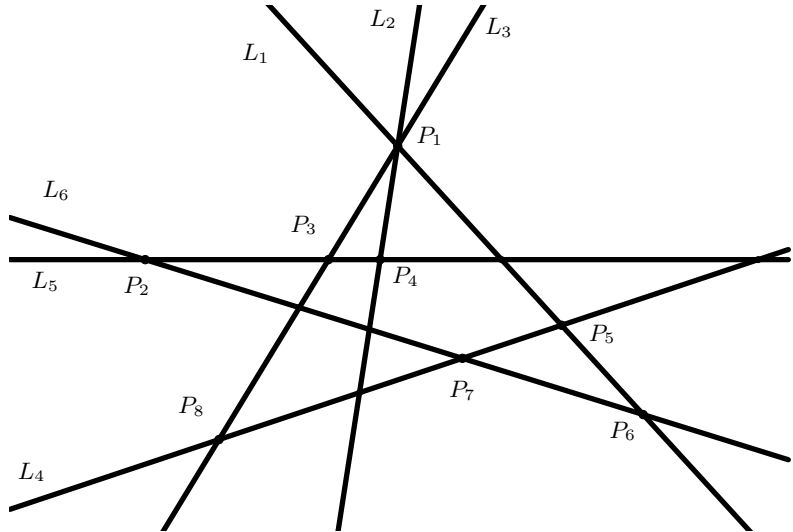
\begin{figure}[H]
\centering
\definecolor{zzccqq}{rgb}{0,0,0}
\definecolor{qqwwzz}{rgb}{0,0,0}
\definecolor{ffwwqq}{rgb}{0,0,0}
\begin{tikzpicture}[line cap=round,line join=round,>=triangle 45,x=1.0cm,y=1.0cm, scale=0.45]
\clip(-2.,-8.) rectangle (22.,7.5);
\draw [line width=2.pt,color=qqwwzz,domain=-2.:21.] plot(\x,{(--64.372-8.68*\x)/-5.28});
\draw [line width=2.pt,color=qqwwzz,domain=-2.:21.] plot(\x,{(--66.4004-5.3*\x)/4.84});
\draw [line width=2.pt,color=qqwwzz,domain=-2.:21.] plot(\x,{(--67.9668-3.38*\x)/-10.12});
\draw [line width=2.pt,color=qqwwzz,domain=-2.:21.] plot(\x,{(--65.20409255397594-7.293452557464656*\x)/-1.1285620951308672});
\draw [line width=2.pt,color=qqwwzz,domain=-2.:21.] plot(\x,{(-0.-0.*\x)/1.});
\draw [line width=2.pt,color=qqwwzz,domain=-2.:21.] plot(\x,{(--9.161905790838375-4.580952895419188*\x)/14.711738115816768});
\begin{scriptsize}
\draw [fill=ffwwqq] (4.18,-5.32) circle (3.5pt);
\draw[color=ffwwqq] (3.37,-4.34) node {$P_8$};
\draw [fill=ffwwqq] (9.46,3.36) circle (3.5pt);
\draw[color=ffwwqq] (10.41,3.66) node {$P_1$};
\draw [fill=ffwwqq] (14.3,-1.94) circle (3.5pt);
\draw[color=ffwwqq] (15.5,-2.2) node {$P_5$};
\draw[color=qqwwzz] (12.45,6.86) node {$L_3$};
\draw[color=qqwwzz] (5.27,6.1) node {$L_1$};
\draw[color=qqwwzz] (-1.35,-6.28) node {$L_4$};
\draw [fill=zzccqq] (8.331437904869134,-3.9334525574646566) circle (2.5pt);
\draw[color=qqwwzz] (9.03,7.12) node {$L_2$};
\draw[color=qqwwzz] (-0.95,-0.64) node {$L_5$};
\draw [fill=ffwwqq] (2.,0.) circle (3.5pt);
\draw[color=ffwwqq] (1.75,-0.78) node {$P_2$};
\draw [fill=ffwwqq] (16.711738115816768,-4.580952895419188) circle (3.5pt);
\draw[color=ffwwqq] (16.11,-5.06) node {$P_6$};
\draw[color=qqwwzz] (-0.63,2.04) node {$L_6$};
\draw [fill=zzccqq] (6.553621382250244,-1.417910303421947) circle (2.5pt);
\draw [fill=zzccqq] (8.621070032169598,-2.0616741336655195) circle (2.5pt);
\draw [fill=ffwwqq] (11.371485452688862,-2.918100708489293) circle (3.5pt);
\draw[color=ffwwqq] (11.37,-3.86) node {$P_7$};
\draw [fill=zzccqq] (12.528377358490568,0.) circle (2.5pt);
\draw [fill=ffwwqq] (8.940085925045372,0.) circle (3.5pt);
\draw[color=ffwwqq] (9.69,-0.5) node {$P_4$};
\draw [fill=ffwwqq] (7.416129032258064,0.) circle (3.5pt);
\draw[color=ffwwqq] (6.69,1.02) node {$P_3$};
\draw [fill=zzccqq] (20.10852071005917,0.) circle (2.5pt);
\end{scriptsize}
\end{tikzpicture}
\caption{Six lines yielding a fiber of type $I_7$}
\end{figure}
and so that we can choose a smooth cubic that is tangent to $L_2$ at $P_1$ (with multiplicity $2$) and that passes through $P_2,\ldots,P_8$. The blow-up of $\mathbb{P}^2$ at $P_1^{(1)},P_1^{(2)},P_2^{(1)},\ldots,P_8^{(1)}$ determines a rational elliptic surface of index two with a singular fiber of type $I_7$.  
\label{examplei7}
\end{exen}

\begin{exen}[Two pairs of three concurrent lines]
Choose six lines $L_1,\ldots,L_6$ as in the picture below:

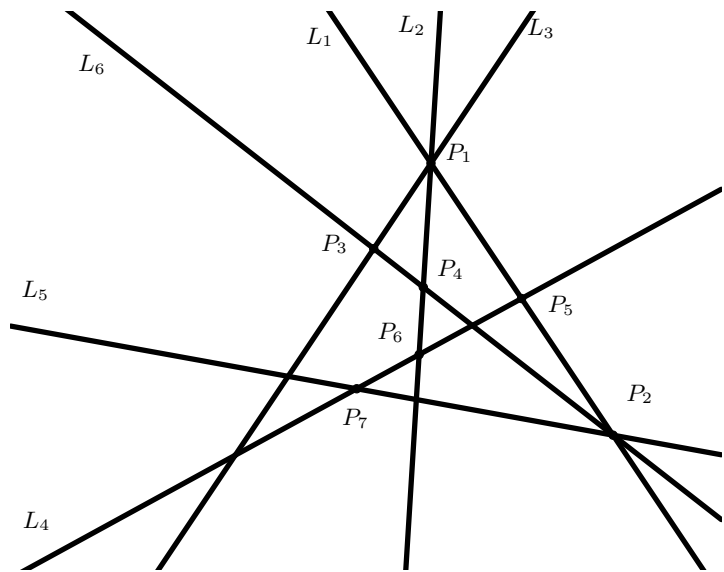
\begin{figure}[H]
\centering
\definecolor{qqwwzz}{rgb}{0,0,0}
\definecolor{zzccqq}{rgb}{0,0,0}
\definecolor{ffwwqq}{rgb}{0,0,0}
\begin{tikzpicture}[line cap=round,line join=round,>=triangle 45,x=1.0cm,y=1.0cm,scale=0.45]
\clip(-4.,-8.) rectangle (17.,8.5);
\draw [line width=2.pt,color=qqwwzz,domain=-4.:17.] plot(\x,{(--36.166-6.3*\x)/-4.22});
\draw [line width=2.pt,color=qqwwzz,domain=-4.:17.] plot(\x,{(--89.2168-8.04*\x)/5.38});
\draw [line width=2.pt,color=qqwwzz,domain=-4.:17.] plot(\x,{(--14.772--1.74*\x)/-9.6});
\draw [line width=2.pt,color=qqwwzz,domain=-4.:17.] plot(\x,{(--57.09896400484654-6.986575082155839*\x)/-0.43199954672640484});
\draw [line width=2.pt,color=qqwwzz,domain=-4.:17.] plot(\x,{(--47.54755943421352-5.515533753976824*\x)/7.070991675907589});
\draw [line width=2.pt,color=qqwwzz,domain=-4.:17.] plot(\x,{(-51.40050816134187--4.6320881898609745*\x)/8.45874582990119});
\begin{scriptsize}
\draw [fill=ffwwqq] (8.42,4.) circle (3.5pt);
\draw[color=ffwwqq] (9.25,4.3) node {$P_1$};
\draw [fill=zzccqq] (4.2,-2.3) circle (2.0pt);
\draw [fill=ffwwqq] (13.8,-4.04) circle (3.5pt);
\draw[color=ffwwqq] (14.59,-2.82) node {$P_2$};
\draw[color=qqwwzz] (11.67,7.98) node {$L_3$};
\draw[color=qqwwzz] (5.13,7.74) node {$L_1$};
\draw[color=qqwwzz] (-3.27,0.24) node {$L_5$};
\draw [fill=zzccqq] (7.988000453273595,-2.986575082155839) circle (2.0pt);
\draw[color=qqwwzz] (7.85,8.08) node {$L_2$};
\draw [fill=ffwwqq] (6.729008324092412,1.475533753976824) circle (3.5pt);
\draw[color=ffwwqq] (5.57,1.62) node {$P_3$};
\draw[color=qqwwzz] (-1.59,6.92) node {$L_6$};
\draw [fill=zzccqq] (2.637871085521697,-4.6320881898609745) circle (2.0pt);
\draw [fill=ffwwqq] (11.096616915422887,0.) circle (3.5pt);
\draw[color=ffwwqq] (12.27,-0.2) node {$P_5$};
\draw[color=qqwwzz] (-3.23,-6.58) node {$L_4$};
\draw [fill=zzccqq] (9.641921691063908,-0.7966046863331621) circle (2.0pt);
\draw [fill=ffwwqq] (8.193281806074575,0.33336725384596544) circle (3.5pt);
\draw[color=ffwwqq] (9.,0.86) node {$P_4$};
\draw [fill=ffwwqq] (8.070193400244822,-1.6572977725159532) circle (3.5pt);
\draw[color=ffwwqq] (7.23,-1.) node {$P_6$};
\draw [fill=ffwwqq] (6.22597639659359,-2.667208221882588) circle (3.5pt);
\draw[color=ffwwqq] (6.19,-3.54) node {$P_7$};
\end{scriptsize}
\end{tikzpicture}
\caption{Six lines yielding a fiber of type $I_8$}
\end{figure}
and such that we choose a smooth cubic that is tangent to $L_3$ at $P_1$ (with multiplicity $2$), is tangent to $L_5$ at $P_2$ (with multiplicity $2$) and that passes through $P_3,\ldots,P_7$. The blow-up of $\mathbb{P}^2$ at $P_1^{(1)},P_1^{(2)},P_2^{(1)},P_2^{(2)},\ldots,P_7^{(1)}$ gives rise to a rational elliptic surface of index two with a singular fiber of type $I_8$. 
\label{examplei8}
\end{exen}

\begin{exe}[Six Lines \cite{kimu2}]
Choose six lines $L_1,\ldots,L_6$ intersecting as in the picture below:
\begin{figure}[H]
\centering
\begin{tikzpicture}[line cap=round,line join=round,>=triangle 45,x=1.0cm,y=1.0cm,scale=0.6]
\clip(3.,-8.) rectangle (20.,6.);
\draw [line width=2.pt,domain=3.:20.] plot(\x,{(--68.4-8.72*\x)/-4.7});
\draw [line width=2.pt,domain=3.:20.] plot(\x,{(-53.2152--0.28*\x)/10.96});
\draw [line width=2.pt,domain=3.:20.] plot(\x,{(-109.44--8.44*\x)/-6.26});
\draw [line width=2.pt] (10.000474129981082,-8.) -- (10.000474129981082,6.);
\draw [line width=2.pt,domain=3.:20.] plot(\x,{(-61.087869056460725--5.847825991252279*\x)/-7.657158009302099});
\draw [line width=2.pt,domain=3.:20.] plot(\x,{(-59.93582428012012--4.015450511314855*\x)/8.18939334113377});
\begin{scriptsize}
\draw [fill=black] (10.,4.) circle (2.5pt);
\draw[color=black] (10.73,3.82) node {$P_1$};
\draw [fill=black] (5.3,-4.72) circle (2.5pt);
\draw[color=black] (5.85,-5.5) node {$P_2$};
\draw [fill=black] (16.26,-4.44) circle (2.5pt);
\draw[color=black] (16.53,-3.76) node {$P_3$};
\draw[color=black] (12,5) node {$L_3$};
\draw[color=black] (19,-3.76) node {$L_4$};
\draw[color=black] (8,5) node {$L_1$};
\draw [fill=black] (10.001019367720465,-4.599900965058237) circle (2.5pt);
\draw[color=black] (10.65,-5.26) node {$P_5$};
\draw [fill=black] (8.602841990697902,1.407825991252278) circle (2.5pt);
\draw[color=black] (7.75,1.36) node {$P_6$};
\draw[color=black] (11,-7) node {$L_2$};
\draw[color=black] (4,4) node {$L_5$};
\draw[color=black] (17,2) node {$L_6$};
\draw [fill=black] (10.00076039870729,-2.4151075665459985) circle (2.0pt);
\draw[color=black] (10.71,-2.68) node {$P_4$};
\end{scriptsize}
\end{tikzpicture}
\caption{Six lines yielding a fiber of type $I_9$}
\end{figure}
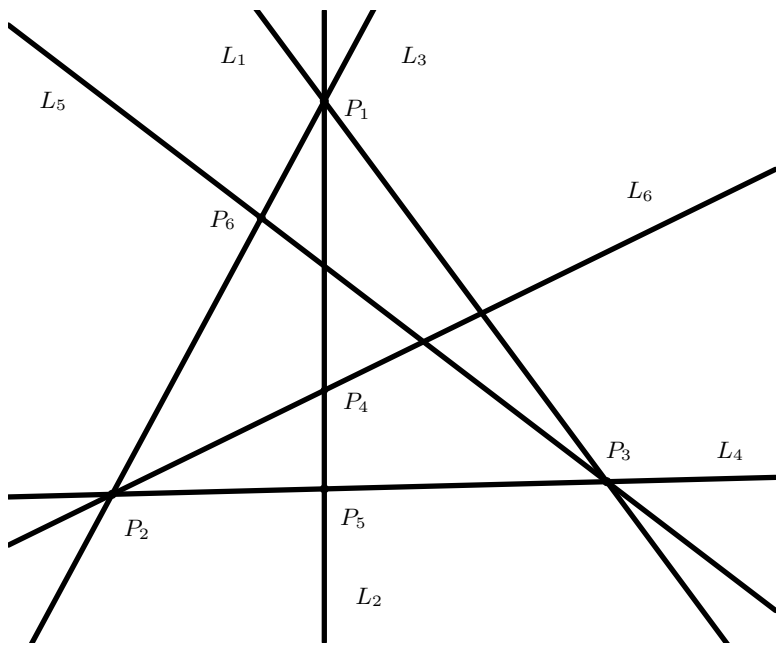
and such that we can choose a smooth cubic $C$ that is tangent to $L_1$ at $P_1$ (with multiplicity $2$), is tangent to $L_5$ at $P_3$ (with multiplicity $2$), is tangent to $L_6$ at $P_2$ (with multiplicity $2$) and that also passes through $P_4,P_5$ and $P_6$. Then, blowing-up $\mathbb{P}^2$ at the points
\[
P_1^{(1)},P_1^{(2)},P_2^{(1)},P_2^{(2)},P_3^{(1)},P_3^{(2)},P_4^{(1)},P_5^{(1)},P_6^{(1)}\]
yields a rational elliptic surface of index two with a fiber of type $I_9$.
\label{examplei9}
\end{exe}

\subsection{Type $II$}

\begin{exen}[A rational sextic]
Let $B$ denote a rational sextic curve with exactly a single cusp and nine nodes, say $P_1,\ldots, P_9$. Note that such curve indeed exists (see Remark \ref{curvefiberii}). Letting $C$ denote a smooth cubic through the nine points $P_1,\ldots, P_9$ we obtain a Halphen pencil of index two, namely the pencil generated by $B$ and $2C$. Blowing-up $\mathbb{P}^2$ at its nine base points $P_1,\ldots,P_9$, but not the cusp, the strict transform of $B$ is a singular fiber of type $II$ in the corresponding elliptic surface. 
\label{exampleii}
\end{exen}

\subsection{Type $III$}

\begin{exen}[Two cubics]
Let $C_1$ and $C_2$ be two distinct nodal cubics which are tangent with multiplicity two at a single point and assume $C_1$ and $C_2$ do not intersect at their nodes.

Denote by $B$ the sextic curve defined by $C_1+C_2$ and let $C$ denote a smooth cubic through the two nodes and the seven points where $C_1$ and $C_2$ intersect transversally. We obtain a Halphen pencil of index two generated by $B$ and $2C$. Blowing-up $\mathbb{P}^2$ at the nine base points (but not the tangency point), the strict transform of $B$ is a type $III$ singular fiber in the corresponding elliptic surface. 
\label{iii1-3,1-3}
\end{exen}

\begin{exen}[A quartic and a conic]
Consider a quartic $Q_1$ with three nodes $P_1,P_2$ and $P_3$ and a conic $Q_2$ that intersects $Q_1$ at other seven points $P_4,\ldots,P_9,P_{10}$, where $P_{10}$ is a tangency point. If we blow-up the nine points $P_1,\ldots,P_9$ we obtain a rational elliptic surface of index two with a fiber of type $III$, namely the strict transform of $B=Q_1+Q_2$.
\label{iii1-4,1-2} 
\end{exen}

\begin{exen}[A quintic and a line]
Choose a quintic $Q$ with six nodes $P_1,\ldots,P_6$ and a line $L$ that intersects $Q$ at other four points $P_7,\ldots,P_9,P_{10}$, where $P_{10}$ is a tangency point. Consider the Halphen pencil of index two generated by $B=Q+L$ and $2C$, where $C$ is the unique cubic through the nine points $P_1,\ldots,P_9$. Blowing-up the nine base points we obtain a rational elliptic surface of index two and the strict transform of $B$ is a singular fiber of type $III$. 
\label{iii1-5,1-1}
\end{exen}

\subsection{Type $IV$}

\begin{exen}[Three conics]
Let $C_1,C_2$ and $C_3$ be three different conics through a common fixed point $P$. Denote by $B$ the sextic curve defined by $C_1+C_2+C_3$ and note that, generically, the three conics intersect at other nine points, say $P_1,\ldots, P_9$. Letting $C$ denote a smooth cubic through the points $P_1,\ldots,P_9$ we obtain a Halphen pencil of degree two generated by $B$ and $2C$. Blowing-up $\mathbb{P}^2$ at the nine base points $P_1,\ldots,P_9$, but not $P$, the strict transform of $B$ is a type $IV$ singular fiber in the corresponding elliptic surface.
\label{iv1-2,1-2,1-2}
\end{exen}

\begin{exen}[A quartic with three nodes and two lines]
Let $Q$ be a plane quartic with three nodes $P_1,P_2$ and $P_3$, then $Q$ is rational. Consider two lines $L_1$ and $L_2$ that intersect $Q$ at other seven points $P_4,\ldots,P_9,P_{10}$. That is, assume $L_1,L_2$ and $Q$ pass through a common point $P_{10}$. Then, blowing-up the nine points $P_1,\ldots,P_9$ we obtain a rational elliptic surface of index two with a singular fiber of type $IV$, which is given by the strict transform of $B=Q+L_1+L_2$. 
\label{iv1-4,1-1,1-1}
\end{exen}

\begin{exen}[A cubic with a node, a conic and a line]
Let $C'$ be a nodal cubic (hence a rational curve). Let $Q$ be a conic and let $L$ be a line. Assume these curves pass through a common point and hence determine other nine distinct intersection points (of multiplicity $2$). Blowing-up these nine double points, the strict transform of $B=C'+Q+L$ is a singular fiber of type $IV$ in the corresponding rational elliptic surface. 
\label{iv1-3,1-2,1-1}
\end{exen}

\subsection{Type $I_n^*$}

\begin{exe}[Five lines \cite{dz}.]
Take five lines in the plane in general position, say $L_1,L_2,L_3, L_4$ and $L_5$. Choose three distinct points in $L_5$ that do not lie in any of the other four lines. Consider the six intersection points $L_i\cap L_j$ for $i,j\neq 5$ and choose a cubic $C$ through these nine points. Let $B=L_1+L_2+L_3+L_4+2L_5$, then the pencil generated by $B$ and $2C$ is a Halphen pencil of index $2$ and the corresponding rational elliptic surface contains a singular fiber of type $I_0^*$. Namely, the strict transform of $B$. 
\label{iostarfivelines}
\end{exe}

\begin{exen}[A double conic and another conic]
Let $Q$ be a (smooth) conic and choose three distinct points on it, say $P_1,P_2$ and $P_3$. We can construct a cubic $C$ so that $C$ is tangent to $Q$ at $P_i$ with multiplicity two. Let $Q'$ be another conic through $P_1,P_2$ and $P_3$ so that $Q'$ intersects $Q$ at a fourth point $P$, then $P\notin C$. We can construct $Q'$ (and $C$) so that $Q'$ intersects $C$ at three other points, say $P_4,P_5$ and $P_6$. Then the pencil generated by $B=2Q'+Q$ and $2C$ is a Halphen pencil of index two which yields a fiber of type $I_0^*$ in the associated rational elliptic surface.

Concretely, we can choose coordinates in $\mathbb{P}^2$ so that $Q$ is the conic given by $x^2+yz=0$ and we have $P_1=(0:0:1),P_2=(0:1:0)$ and $P_3=(1:-1:1)$. Then we can let $C$ be the cubic given by 
\[
y(z(2x+y-z)+x^2+yz)=0
\]
and we can choose $Q'$ to be the conic given by $xy+2xz+yz=0$. 

Blowing-up $\mathbb{P}^2$ at the points $
P_1^{(1)},P_1^{(2)},P_2^{(1)},P_2^{(2)},P_3^{(1)},P_3^{(2)},P_4^{(1)},P_5^{(1)},P_6^{(1)}
$  we obtain the following (dual) configuration of rational curves:
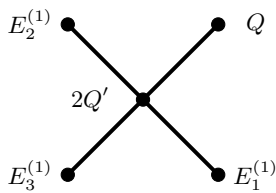
\begin{figure}[H]
\centering
\begin{tikzpicture}[line cap=round,line join=round,>=triangle 45,x=1.0cm,y=1.0cm]
\clip(-2,-2.) rectangle (2.,2.);
\draw [line width=1.5pt] (-1.,1.)-- (0.,0.);
\draw [line width=1.5pt] (0.,0.)-- (-1.,-1.);
\draw [line width=1.5pt] (0.,0.)-- (1.,1.);
\draw [line width=1.5pt] (0.,0.)-- (1.,-1.);
\begin{scriptsize}
\draw [fill=black] (0.,0.) circle (2.5pt);
\draw[color=black] (-0.7,0.) node {$2Q'$};
\draw [fill=black] (1.,1.) circle (2.5pt);
\draw[color=black] (1.5,1) node {$Q$};
\draw [fill=black] (1.,-1.) circle (2.5pt);
\draw[color=black] (1.5,-1) node {$E_1^{(1)}$};
\draw [fill=black] (-1.,1.) circle (2.5pt);
\draw[color=black] (-1.5,1) node {$E_2^{(1)}$};
\draw [fill=black] (-1.,-1.) circle (2.5pt);
\draw[color=black] (-1.5,-1) node {$E_3^{(1)}$};
\end{scriptsize}
\end{tikzpicture}
\end{figure}
\label{iostar2-2,1-2}
\end{exen}

\begin{exen}[A double line and two conics]
Let $Q_1$ and $Q_2$ be two (smooth) conics intersecting at four distinct points, say $P_3,P_4,P_5$ and $P_6$. Let $P_1$ be a point in $Q_1$ that is not in $Q_2$ and let $P_2$ be a point in $Q_2$ that is not in $Q_1$. Let $L$ be the line joining $P_1$ and $P_2$. We can choose $P_1$ and $P_2$ so that the $P_i$ for $i=3,\ldots,6$ are not in $L$ and $L$ intersects both $Q_1$ and $Q_2$ at a second point different than the $P_i$ for $i=3,\ldots,6$.

Now, we can construct a cubic $C$ through $P_1,\ldots,P_6$ such that the intersection multiplicity of $C$ and $Q_i$ at $P_i$ is two, for $i=1,2$, and $C$ intersects $L$ at a third point, say $P_7$.

The pencil generated by $B=2L+Q_1+Q_2$ and $2C$ is a Halphen pencil of index two and blowing-up $\mathbb{P}^2$ at its nine base points $
P_1^{(1)},P_1^{(2)},P_2^{(1)},P_2^{(2)},P_3^{(1)},P_4^{(1)},P_5^{(1)},P_6^{(1)},P_7^{(1)}
$ we obtain the following (dual) configuration of rational curves:

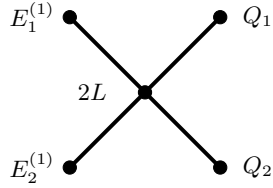
\begin{figure}[H]
\centering
\begin{tikzpicture}[line cap=round,line join=round,>=triangle 45,x=1.0cm,y=1.0cm]
\clip(-2,-2.) rectangle (2.,2.);
\draw [line width=1.5pt] (-1.,1.)-- (0.,0.);
\draw [line width=1.5pt] (0.,0.)-- (-1.,-1.);
\draw [line width=1.5pt] (0.,0.)-- (1.,1.);
\draw [line width=1.5pt] (0.,0.)-- (1.,-1.);
\begin{scriptsize}
\draw [fill=black] (0.,0.) circle (2.5pt);
\draw[color=black] (-0.7,0.) node {$2L$};
\draw [fill=black] (1.,1.) circle (2.5pt);
\draw[color=black] (1.5,1) node {$Q_1$};
\draw [fill=black] (1.,-1.) circle (2.5pt);
\draw[color=black] (1.5,-1) node {$Q_2$};
\draw [fill=black] (-1.,1.) circle (2.5pt);
\draw[color=black] (-1.5,1) node {$E_1^{(1)}$};
\draw [fill=black] (-1.,-1.) circle (2.5pt);
\draw[color=black] (-1.5,-1) node {$E_2^{(1)}$};
\end{scriptsize}
\end{tikzpicture}
\end{figure}

That is, we obtain a fiber of type $I_0^*$ in the corresponding rational elliptic surface.

Concretely, we can choose coordinates in $\mathbb{P}^2$ so that $Q_1$ is the conic given by $x^2+yz=0$ and $Q_2$ is the conic given by $y^2+xz=0$. Then we can let $P_1=(1:-1:1)$ and $P_2=(-4:-2:1)$ so that $C$ is the cubic 
\[
x^3 - 4x^2y + 2xy^2 + y^3 - 2x^2z + 2xyz - 5y^2z - xz^2 - 4yz^2=0
\]
and $L$ is the line $x-5y-6z=0$.
\label{iostar2-1,1-2,1-2}
\end{exen}

\begin{lema}[{\cite[Problem 5.41]{fultoncurves}}]
Let $C$ be a smooth cubic, let $P_1,\ldots,P_{3d}$ be points on $C$ (not necessarily distinct) and choose the group law $\oplus$ on $C$ with a flex point $O$ as the origin. These points satisfy the equation $P_1\oplus \ldots \oplus P_{3d}=0$ if and only if there exists a plane curve $D$ of degree $d$ intersecting $C$ precisely at them (where a certain number of repetitions means the intersection multiplicity of $C$ and $D$ at that point).
\label{sum3d}
\end{lema}

\begin{exen}[A double line, a conic and two lines]
Let $C$ be a smooth cubic. Let $P$ be one of its flex points, which we fix  as the origin for the natural group law $\oplus$. Let $\varepsilon_2$ be a 2-torsion point (w.r.t $\oplus$) and choose another point $P_4\in C$. Let $P_1\doteq P_4-\varepsilon_2$. Then there exist points $P_2,P_3,P_5,P_6\in C$ (all distinct) such that
\[
P_1 \oplus P_2 \oplus P_3=0 \quad  P_4 \oplus P_5 \oplus P_6=0 \quad \mbox{and} \quad P_4 \oplus P_7 \oplus P_8=0
\]
By construction we have 
\begin{eqnarray*}
P_2 \oplus P_3 \oplus P_4&=&\varepsilon_2\\
P_1 \oplus P_5 \oplus P_6&=&-\varepsilon_2=\varepsilon_2\\
P_1 \oplus P_7 \oplus P_8&=&-\varepsilon_2=\varepsilon_2
\end{eqnarray*}
In particular, $2P_1\oplus P_2 \oplus \ldots \oplus P_8=3\cdot \varepsilon_2=\varepsilon_2$ and $2P_1\oplus P_5\oplus P_6\oplus P_7 \oplus P_8=2\cdot \varepsilon_2=0$.

It follows from Lemma \ref{sum3d} that we can construct three lines $L_1,L_2$ and $L_3$ such that $P_1,P_2,P_3 \in L_1$, $P_4,P_5,P_6 \in L_2$ and $P_4,P_7,P_8 \in L_3$. By Lemma \ref{sum3d} we can also construct a conic $Q$ through $P_1,P_5,P_6,P_7$ and $P_8$  so that the intersection multiplicity of $Q$ and $C$ at $P_1$ is two.

The pencil generated by $B=2L_1+L_2+L_3+Q$ and $2C$ is a Halphen pencil of index two and blowing-up $\mathbb{P}^2$ at its nine base points $
P_1^{(1)},P_1^{(2)},P_2^{(1)},P_3^{(1)},P_4^{(1)},P_5^{(1)},P_6^{(1)},P_7^{(1)},P_8^{(1)}
$ we obtain the following (dual) configuration of rational curves:

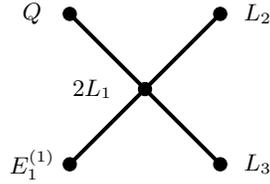
\begin{figure}[H]
\centering
\begin{tikzpicture}[line cap=round,line join=round,>=triangle 45,x=1.0cm,y=1.0cm]
\clip(-2,-2.) rectangle (2.,2.);
\draw [line width=1.5pt] (-1.,1.)-- (0.,0.);
\draw [line width=1.5pt] (0.,0.)-- (-1.,-1.);
\draw [line width=1.5pt] (0.,0.)-- (1.,1.);
\draw [line width=1.5pt] (0.,0.)-- (1.,-1.);
\begin{scriptsize}
\draw [fill=black] (0.,0.) circle (2.5pt);
\draw[color=black] (-0.7,0.) node {$2L_1$};
\draw [fill=black] (1.,1.) circle (2.5pt);
\draw[color=black] (1.5,1) node {$L_2$};
\draw [fill=black] (1.,-1.) circle (2.5pt);
\draw[color=black] (1.5,-1) node {$L_3$};
\draw [fill=black] (-1.,1.) circle (2.5pt);
\draw[color=black] (-1.5,1) node {$Q$};
\draw [fill=black] (-1.,-1.) circle (2.5pt);
\draw[color=black] (-1.5,-1) node {$E_1^{(1)}$};
\end{scriptsize}
\end{tikzpicture}
\end{figure}

That is, we obtain a fiber of type $I_0^*$ in the corresponding rational elliptic surface. 
\label{iostar2-1,1-2,1-1,1-1}
\end{exen}

\begin{exe}[Five lines, with three lines concurrent at a point \cite{exlaface}]
This next example is a geometrical description of an example communicated by Antonio Laface. The associated Jacobian is the surface $X_{141}$ in  Miranda and Persson's list \cite{extr}. 

Consider five distinct lines in $\mathbb{P}^2$, say $L_1,L_2,L_3,L_4$ and $L_5$ and assume that $L_1,L_3$ and $L_5$ are concurrent at a point $P_1$.

We can choose coordinates in $\mathbb{P}^2$ so that we have $L_1: z=0, L_2: y-z=0, L_3: x-z=0$, $L_4: x-y-z=0$ and $L_5:x=0$. These lines determine eight intersection points:

\begin{multicols}{2}
\begin{itemize}
\item $P_1=(0:1:0)$ ($L_1\cap L_3\cap L_5$)
\item $P_2=(1:0:1)$ ($L_3\cap L_4$)
\item $P_3=(1:1:0)$ ($L_1\cap L_4$) 
\item $P_4=(2:1:1)$ ($L_2\cap L_4$)
\item $P_5=(1:0:0)$ ($L_1\cap L_2$)
\item $P_6=(1:1:1)$ ($L_2\cap L_3$)
\item $Q_1=(0:1:1)$ ($L_2\cap L_5$)
\item $Q_2=(0:1:-1)$ ($L_4\cap L_5$)
\end{itemize}
\end{multicols}

Now consider the lines $R_1:x-y=0$ determined by the points $P_3$ and $P_6$; $R_2:y=0$, determined by the points $P_2$ and $P_5$; and $R_3:x-2z=0$ determined by the points $P_1$ and $P_4$

There exists a unique cubic $C$ through $P_1,\ldots,P_6$ which is tangent to $R_3$ at $P_1=(0:1:0)$ and which has a node at $P_7 \doteq (0:t:1)$  (for a parameter $t\neq \pm 1$ and $t\neq 0$).

Such curves determine five more intersection points

\begin{multicols}{2}
\begin{itemize}
\item $P_8=(0:0:1)$ 
\subitem $\{P_8\}=R_1\cap R_2\cap L_5$
\item $\{Q_3\}=R_1\cap R_3$
\item $\{Q_4\}=R_2\cap R_3$
\item $Q_5 \in R_1\cap C$
\item $Q_6\in R_2\cap C$
\end{itemize}
\end{multicols}

Consider the pencil $\mathcal{P}$ generated by $B\doteq L_1+L_2+L_3+L_4+2L_5$ and $D\doteq C+R_1+R_2+R_3$. If we blow-up of $\mathbb{P}^2$ at the points $
P_1^{(1)},P_1^{(2)},P_2^{(1)},P_3^{(1)}, \ldots, P_8^{(1)}$ we obtain a rational elliptic surface of index two with a singular fiber $F$ of type $I_1^*$ and a singular fiber $F'$ of type $I_4$. We have the following (dual) configurations of rational curves:

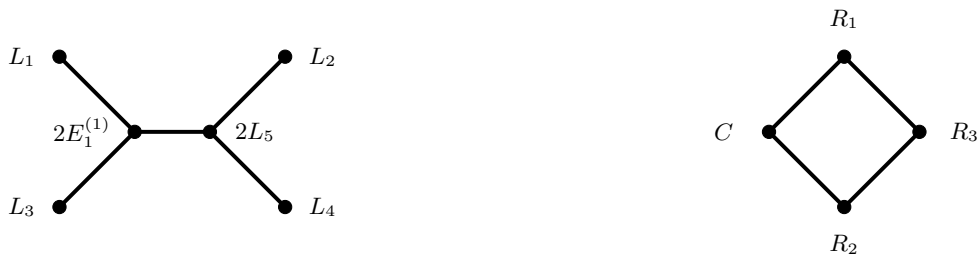
\begin{multicols}{2}
\begin{figure}[H]
\centering
\begin{tikzpicture}[line cap=round,line join=round,>=triangle 45,x=1.0cm,y=1.0cm]
\clip(-2,-2.) rectangle (4.,2.);
\draw [line width=1.5pt] (-1.,1.)-- (0.,0.);
\draw [line width=1.5pt] (0.,0.)-- (-1.,-1.);
\draw [line width=1.5pt] (0.,0.)-- (1.,0.);
\draw [line width=1.5pt] (1.,0.)-- (2.,1.);
\draw [line width=1.5pt] (1.,0.)-- (2.,-1.);
\begin{scriptsize}
\draw [fill=black] (0.,0.) circle (2.5pt);
\draw[color=black] (-0.7,0.) node {$2E_1^{(1)}$};
\draw [fill=black] (1.,0.) circle (2.5pt);
\draw[color=black] (1.6,0) node {$2L_5$};
\draw [fill=black] (2.,1.) circle (2.5pt);
\draw[color=black] (2.5,1) node {$L_2$};
\draw [fill=black] (2.,-1.) circle (2.5pt);
\draw[color=black] (2.5,-1) node {$L_4$};
\draw [fill=black] (-1.,1.) circle (2.5pt);
\draw[color=black] (-1.5,1) node {$L_1$};
\draw [fill=black] (-1.,-1.) circle (2.5pt);
\draw[color=black] (-1.5,-1) node {$L_3$};
\end{scriptsize}
\end{tikzpicture}
\end{figure}

\begin{figure}[H]
\centering
\begin{tikzpicture}[line cap=round,line join=round,>=triangle 45,x=1.0cm,y=1.0cm]
\clip(-1,-2.) rectangle (3.,2.);
\draw [line width=1.5pt] (0,0)-- (1,1);
\draw [line width=1.5pt] (0.,0.)-- (1.,-1.);
\draw [line width=1.5pt] (1.,1.)-- (2.,0.);
\draw [line width=1.5pt] (1.,-1.)-- (2.,0.);
\begin{scriptsize}
\draw [fill=black] (0.,0.) circle (2.5pt);
\draw[color=black] (-0.6,0.) node {$C$};
\draw [fill=black] (2.,0.) circle (2.5pt);
\draw[color=black] (2.6,0) node {$R_3$};
\draw [fill=black] (1.,1.) circle (2.5pt);
\draw[color=black] (1.,1.5) node {$R_1$};
\draw [fill=black] (1.,-1.) circle (2.5pt);
\draw[color=black] (1,-1.5) node {$R_2$};
\end{scriptsize}
\end{tikzpicture}
\end{figure}
\end{multicols}

\label{exelaface}
\end{exe}

\begin{exen}[A double conic and another conic]
Let $Q$ be a smooth conic and choose three (distinct) points on it, say $P_1,P_2$ and $P_3$. Then we can construct a cubic $C$ through $P_1,P_2$ and $P_3$ so that $C$ is tangent to $Q$ at each $P_i$ with multiplicity two. And we can also construct another conic $Q'$  through $P_1,P_2$ and $P_3$ so that $Q'$ is tangent to $Q$ and to $C$ at $P_3$ (with multiplicity two) and $Q'$ intersects $C$ at two other points, say $P_4$ and $P_5$.

Concretely,  we can choose coordinates in $\mathbb{P}^2$ so that $Q$ is the conic given by $x^2+yz=0$. We can let $P_1=(0:0:1),P_2=(0:1:0)$ and $P_3=(1:-1:1)$. And we can take $C$ to be the cubic given by
\[
yz(2x+y-z)-x(x^2+yz)=y^2z-yz^2-x^3+xyz=0
\]
Then $Q'$ is the conic given by $x^2+xy-xz-yz=0$ and we have that $P_4=(1:1:1)$ and $P_5=(-1:1:1)$.

By letting $B=2Q'+Q$ we have that the pencil generated by $B$ and $2C$ is a Halphen pencil of index two. Moreover, the corresponding rational elliptic surface has a fiber of type $I_1^*$.

We blow-up $\mathbb{P}^2$ at the points $
P_1^{(1)},P_1^{(2)},P_2^{(1)},P_2^{(2)},P_3^{(1)},P_3^{(2)},P_3^{(3)},P_4^{(1)},P_5^{(1)}
$, which produces the following (dual) configuration of rational curves:

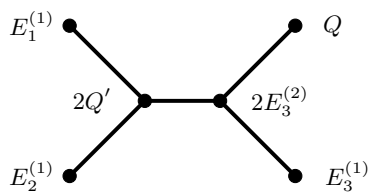
\begin{figure}[H]
\centering
\begin{tikzpicture}[line cap=round,line join=round,>=triangle 45,x=1.0cm,y=1.0cm]
\clip(-2,-2.) rectangle (4.,2.);
\draw [line width=1.5pt] (-1.,1.)-- (0.,0.);
\draw [line width=1.5pt] (0.,0.)-- (-1.,-1.);
\draw [line width=1.5pt] (0.,0.)-- (1.,0.);
\draw [line width=1.5pt] (1.,0.)-- (2.,1.);
\draw [line width=1.5pt] (1.,0.)-- (2.,-1.);
\begin{scriptsize}
\draw [fill=black] (0.,0.) circle (2.5pt);
\draw[color=black] (-0.7,0.) node {$2Q'$};
\draw [fill=black] (1.,0.) circle (2.5pt);
\draw[color=black] (1.8,0) node {$2E_3^{(2)}$};
\draw [fill=black] (2.,1.) circle (2.5pt);
\draw[color=black] (2.5,1) node {$Q$};
\draw [fill=black] (2.,-1.) circle (2.5pt);
\draw[color=black] (2.7,-1) node {$E_3^{(1)}$};
\draw [fill=black] (-1.,1.) circle (2.5pt);
\draw[color=black] (-1.5,1) node {$E_1^{(1)}$};
\draw [fill=black] (-1.,-1.) circle (2.5pt);
\draw[color=black] (-1.5,-1) node {$E_2^{(1)}$};
\end{scriptsize}
\end{tikzpicture}
\end{figure}

\label{i1star2-2,1-2}
\end{exen}

\begin{exen}[A double conic and two tangent lines]
Let $C$ be a smooth cubic. Choose two distinct points $P_3$ and $P_4$ on $C$ such that the lines $L_1$ (resp. $L_2$) tangent to $C$ at $P_3$ (resp. $P_4$) intersect at a fifth point $P_5$ that is also in $C$. Then there exists a unique conic $Q$ that is tangent to the lines $L_1$ (resp. $L_2$) at $P_3$ (resp. $P_4$). Such conic intersects $C$ at two other distinct points, say $P_1$ and $P_2$.  If we let $B=2Q+L_1+L_2$, it follows that the pencil generated by $B$ and $2C$
is a Halphen pencil of index two. Thus, blowing-up the nine points $
P_1^{(1)},P_2^{(1)},P_3^{(1)},P_3^{(2)},P_3^{(3)},P_4^{(1)},P_4^{(2)},P_4^{(3)},P_5^{(1)}
$ we obtain a rational elliptic surface of index two with a type $I_2^*$ singular fiber. 

In fact, we obtain the following (dual) configuration of rational curves:
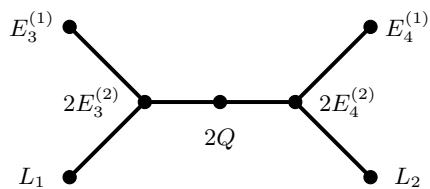
\begin{figure}[H]
\centering
\begin{tikzpicture}[line cap=round,line join=round,>=triangle 45,x=1.0cm,y=1.0cm]
\clip(-2,-2.) rectangle (4.,2.);
\draw [line width=1.5pt] (-1.,1.)-- (0.,0.);
\draw [line width=1.5pt] (0.,0.)-- (-1.,-1.);
\draw [line width=1.5pt] (0.,0.)-- (2.,0.);
\draw [line width=1.5pt] (2.,0.)-- (3.,1.);
\draw [line width=1.5pt] (2.,0.)-- (3.,-1.);
\begin{scriptsize}
\draw [fill=black] (0.,0.) circle (2.5pt);
\draw[color=black] (-0.7,0.) node {$2E_3^{(2)}$};
\draw [fill=black] (1.,0.) circle (2.5pt);
\draw[color=black] (1.,-0.5) node {$2Q$};
\draw [fill=black] (2.,0.) circle (2.5pt);
\draw[color=black] (2.7,0) node {$2E_4^{(2)}$};
\draw [fill=black] (3.,1.) circle (2.5pt);
\draw[color=black] (3.5,1) node {$E_4^{(1)}$};
\draw [fill=black] (3.,-1.) circle (2.5pt);
\draw[color=black] (3.5,-1) node {$L_2$};
\draw [fill=black] (-1.,1.) circle (2.5pt);
\draw[color=black] (-1.5,1) node {$E_3^{(1)}$};
\draw [fill=black] (-1.,-1.) circle (2.5pt);
\draw[color=black] (-1.5,-1) node {$L_1$};
\end{scriptsize}
\end{tikzpicture}
\end{figure}

\label{i2star2-2,1-1,1-1}
\end{exen}

\begin{exen}[Two double lines and two other lines]
Let $C$ be a smooth cubic. Choose $P_1\in C$ not a flex point and let $L_1$ be the tangent line to $C$ at $P_1$. Then $C$ intersects $L_1$ at another point, say $P_2$. We can choose $P_1$ so that $P_2$ is not a flex point. Let $L_3$ be the tangent line to $C$ at $P_2$ and let $P_4$ be the third intersection point. Choose a line $L_4$ through $P_4$ which is tangent to $C$ at a point $P_3$. Let $L_2$ be the line joining $P_1$ and $P_3$. Then $L_2$ intersects $C$ at a third point $P_5$. 

Letting $B=2L_1+2L_2+L_3+L_4$ we have that the pencil generated by $B$ and $2C$ is a Halphen pencil of index two, which yields a fiber of type $I_3^*$ in the corresponding rational elliptic surface. More precisely, blowing-up 
$
P_1^{(1)},P_1^{(2)},P_1^{(3)},P_2^{(1)},P_2^{(2)},P_3^{(1)},P_3^{(2)},P_4^{(1)},P_5^{(1)}
$ gives the following (dual) configuration of rational curves:
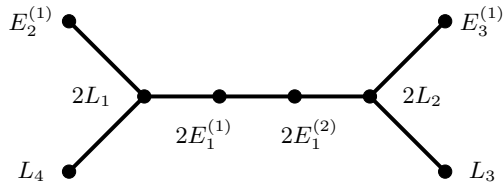
\begin{figure}[H]
\centering
\begin{tikzpicture}[line cap=round,line join=round,>=triangle 45,x=1.0cm,y=1.0cm]
\clip(-2,-2.) rectangle (5.,2.);
\draw [line width=1.5pt] (-1.,1.)-- (0.,0.);
\draw [line width=1.5pt] (0.,0.)-- (-1.,-1.);
\draw [line width=1.5pt] (0.,0.)-- (3.,0.);
\draw [line width=1.5pt] (3.,0.)-- (4.,1.);
\draw [line width=1.5pt] (3.,0.)-- (4.,-1.);
\begin{scriptsize}
\draw [fill=black] (0.,0.) circle (2.5pt);
\draw[color=black] (-0.7,0.) node {$2L_1$};
\draw [fill=black] (1.,0.) circle (2.5pt);
\draw[color=black] (0.8,-0.5) node {$2E_1^{(1)}$};
\draw [fill=black] (2.,0.) circle (2.5pt);
\draw[color=black] (2.2,-0.5) node {$2E_1^{(2)}$};
\draw [fill=black] (3.,0.) circle (2.5pt);
\draw[color=black] (3.7,0) node {$2L_2$};
\draw [fill=black] (4.,1.) circle (2.5pt);
\draw[color=black] (4.5,1) node {$E_3^{(1)}$};
\draw [fill=black] (4.,-1.) circle (2.5pt);
\draw[color=black] (4.5,-1) node {$L_3$};
\draw [fill=black] (-1.,1.) circle (2.5pt);
\draw[color=black] (-1.5,1) node {$E_2^{(1)}$};
\draw [fill=black] (-1.,-1.) circle (2.5pt);
\draw[color=black] (-1.5,-1) node {$L_4$};
\end{scriptsize}
\end{tikzpicture}
\end{figure}
\label{i3star2-1,2-1,1-1,1-1}
\end{exen}

\begin{exen}[Two double lines and two other lines]
Let $C$ be a smooth cubic. Let $P_1\in C$ be a flex point and let $L_4$ be the corresponding inflection line. Let $L_1$ be a line through $P_1$ which is tangent to $C$ at $P_2$ and let $L_3$ be another line through $P_1$ which is tangent to $C$ at $P_3\neq P_2$. Let $L_2$ be the line joining $P_2$ and $P_3$. Then $L_2$ intersects $C$ at a third point $P_4$

If we let $B=2L_1+2L_2+L_3+L_4$, then the pencil generated by $B$ and $2C$ is a Halphen pencil of index two, which gives a fiber of type $I_4^*$ in the corresponding rational elliptic surface. The blow-up of $\mathbb{P}^2$ at
$
P_1^{(1)},P_1^{(2)},P_1^{(3)},P_2^{(1)},P_2^{(2)},P_2^{(3)},P_3^{(1)},P_3^{(2)},P_4^{(1)}
$ yields the following (dual) configuration of rational curves:
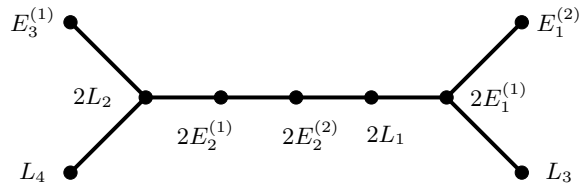
\begin{figure}[H]
\centering
\begin{tikzpicture}[line cap=round,line join=round,>=triangle 45,x=1.0cm,y=1.0cm]
\clip(-2,-2.) rectangle (6.,2.);
\draw [line width=1.5pt] (-1.,1.)-- (0.,0.);
\draw [line width=1.5pt] (0.,0.)-- (-1.,-1.);
\draw [line width=1.5pt] (0.,0.)-- (4.,0.);
\draw [line width=1.5pt] (4.,0.)-- (5.,1.);
\draw [line width=1.5pt] (4.,0.)-- (5.,-1.);
\begin{scriptsize}
\draw [fill=black] (0.,0.) circle (2.5pt);
\draw[color=black] (-0.7,0.) node {$2L_2$};
\draw [fill=black] (1.,0.) circle (2.5pt);
\draw[color=black] (0.8,-0.5) node {$2E_2^{(1)}$};
\draw [fill=black] (2.,0.) circle (2.5pt);
\draw[color=black] (2.2,-0.5) node {$2E_2^{(2)}$};
\draw [fill=black] (3.,0.) circle (2.5pt);
\draw[color=black] (3.2,-0.5) node {$2L_1$};
\draw [fill=black] (4.,0.) circle (2.5pt);
\draw[color=black] (4.7,0) node {$2E_1^{(1)}$};
\draw [fill=black] (5.,1.) circle (2.5pt);
\draw[color=black] (5.5,1) node {$E_1^{(2)}$};
\draw [fill=black] (5.,-1.) circle (2.5pt);
\draw[color=black] (5.5,-1) node {$L_3$};
\draw [fill=black] (-1.,1.) circle (2.5pt);
\draw[color=black] (-1.5,1) node {$E_3^{(1)}$};
\draw [fill=black] (-1.,-1.) circle (2.5pt);
\draw[color=black] (-1.5,-1) node {$L_4$};
\end{scriptsize}
\end{tikzpicture}
\end{figure}

\label{i4star2-1,2-1,1-1,1-1}
\end{exen}

\subsection{Type $IV^*$}

We now construct all possible examples of Halphen pencils of index two that yield a fiber of type $IV^*$ in the corresponding rational elliptic surface (Theorem \ref{allpossibleivstar}).

\begin{defi}
Given a cubic $C$, a conic $Q$ and a point $P\in C$, we say $Q$ is an osculating conic of $C$ at $P$ if $I_P(Q,C)\geq 5$, where $I_P(Q,C)$ denotes the intersection multiplicity of $Q$ and $C$ at $P$.
\end{defi}

\begin{defi}
Given a cubic $C$, any point on it where a tangent conic intersects $C$ with multiplicity six is called a sextactic point.  If $C$ is smooth, there are exactly $27$ such points and if $C$ is nodal, then there only $3$ sextactic points (see e.g. \cite{cayley5},\cite{cayley6}).
\label{sextactic}
\end{defi}

\begin{exen}[A double conic and a conic]
Consider a smooth cubic $C$ and let $P_1$ be a sextactic point. Let $Q_1$ be the corresponding osculating conic. Assume we can construct another conic $Q_2$ so that $Q_2$ is tangent to both $Q_1$ and $C$ at $P_1$ with multiplicity three, $Q_2$ intersects $C$ at other three points $P_2,P_3,P_4$. Then the fourth intersection point between the two conics is different than the $P_i$'s.

Letting $B=Q_1+2Q_2$ we have that the pencil generated by $B$ and $2C$ is a Halphen pencil of index two and the corresponding rational elliptic surface has a fiber of type $IV^*$.

For instance, let $C$ be the cubic given by $xz^2+y^2z+x^3=0$, then we can let $P_1=(0:0:1)$ and we have that $Q_1$ is the conic $y^2+xz=0$. Choosing $Q_2$ to be the conic $xy+y^2+xz=0$ we get the desired pencil.

Blowing-up $\mathbb{P}^2$ at the points $P_1^{(1)},\ldots,P_1^{(6)},P_2^{(1)},P_3^{(1)},P_4^{(1)}$ we obtain the following (dual) configuration of rational curves:

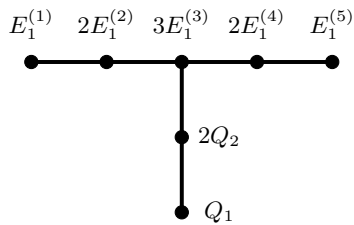
\begin{figure}[H]
\centering
\begin{tikzpicture}[line cap=round,line join=round,>=triangle 45,x=1.0cm,y=1.0cm]
\clip(-3.5,-3.) rectangle (3.,1.);
\draw [line width=1.5pt] (-2.,0.)-- (2.,0.);
\draw [line width=1.5pt] (0.,0.)-- (0.,-2.);
\begin{scriptsize}
\draw [fill=black] (0.,-2.) circle (2.5pt);
\draw[color=black] (0.5,-2.) node {$Q_1$};
\draw [fill=black] (0.,-1.) circle (2.5pt);
\draw[color=black] (0.5,-1.) node {$2Q_2$};
\draw [fill=black] (-2.,0.) circle (2.5pt);
\draw[color=black] (-2.,.5) node {$E_1^{(1)}$};
\draw [fill=black] (-1.,0.) circle (2.5pt);
\draw[color=black] (-1.,.5) node {$2E_1^{(2)}$};
\draw [fill=black] (0.,0.) circle (2.5pt);
\draw[color=black] (0.,.5) node {$3E_1^{(3)}$};
\draw [fill=black] (1.,0.) circle (2.5pt);
\draw[color=black] (1.,0.5) node {$2E_1^{(4)}$};
\draw [fill=black] (2.,0.) circle (2.5pt);
\draw[color=black] (2.,.5) node {$E_1^{(5)}$};
\end{scriptsize}
\end{tikzpicture}
\end{figure}
\label{ivstar2-2,1-2}
\end{exen}

\begin{exen}[A double line, a conic and two lines]
Let $Q$ be a (smooth) conic and choose $P_1\in Q$. Let $T$ be the tangent line to $Q$ at $P_1$. Let $L_1$ be a line through $P_1$, intersecting $Q$ at a second point $P_2$. Choose two other points $P_3$ and $P_4$ in $Q$, let $L_2$ be the line joining them and let $\{P_5\}=L_1\cap L_2$. Assume we can construct a cubic $C$ through $P_1,\ldots,P_5$ which is tangent to $Q$ (resp. $T$) with multiplicity $3$ (resp. $2$.). Then $C$ intersects $T$ at another point $P_6$.

Letting $B=2T+Q+L_1+L_2$ we have that the pencil generated by $B$ and $2C$ is a Halphen pencil of index two and the corresponding rational elliptic surface has a fiber of type $IV^*$.

For instance, we can choose coordinates so that $Q$ is the conic $y^2+xz=0$ and we can choose $P_1=(0:0:1)$. Then $T$ is the line $x=0$. Choosing $L_1$ to be the line $x+y=0$ we have that $P_2=(-1:-1:1)$. Now, if we choose $P_4$ and $P_5$ so that $L_2$ is the line $x+y+z$, then $P_5=(-1:1:0)$ and $C$ is the cubic $x^3 + y^3 + 2xyz + y^2z + xz^2=0$. Thus, $P_6$ is the point $(0:1:-1)$.

Blowing-up $\mathbb{P}^2$ at the points $P_1^{(1)},\ldots,P_1^{(4)},P_2^{(1)},P_3^{(1)},P_4^{(1)},P_5^{(1)},P_6^{(1)}$ we obtain the following (dual) configuration of rational curves:

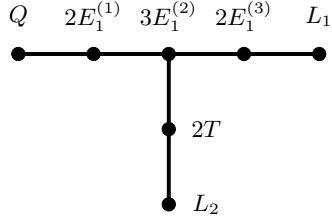
\begin{figure}[H]
\centering
\begin{tikzpicture}[line cap=round,line join=round,>=triangle 45,x=1.0cm,y=1.0cm]
\clip(-3.5,-3.) rectangle (3.,1.);
\draw [line width=1.5pt] (-2.,0.)-- (2.,0.);
\draw [line width=1.5pt] (0.,0.)-- (0.,-2.);
\begin{scriptsize}
\draw [fill=black] (0.,-2.) circle (2.5pt);
\draw[color=black] (0.5,-2.) node {$L_2$};
\draw [fill=black] (0.,-1.) circle (2.5pt);
\draw[color=black] (0.5,-1.) node {$2T$};
\draw [fill=black] (-2.,0.) circle (2.5pt);
\draw[color=black] (-2.,.5) node {$Q$};
\draw [fill=black] (-1.,0.) circle (2.5pt);
\draw[color=black] (-1.,.5) node {$2E_1^{(1)}$};
\draw [fill=black] (0.,0.) circle (2.5pt);
\draw[color=black] (0.,.5) node {$3E_1^{(2)}$};
\draw [fill=black] (1.,0.) circle (2.5pt);
\draw[color=black] (1.,0.5) node {$2E_1^{(3)}$};
\draw [fill=black] (2.,0.) circle (2.5pt);
\draw[color=black] (2.,.5) node {$L_1$};
\end{scriptsize}
\end{tikzpicture}
\end{figure}

\label{ivstar2-1,1-2,1-1,1-1}
\end{exen}

\begin{exen}[A double line, a cubic and another line]
Let $D$ be a nodal cubic and denote its node by $P_5$. Let $P_1$ be a flex point of $D$ and denote the corresponding inflection line by $L$. Let $L'$ be a line that intersects $D$ at three other points $P_2,P_3$ and $P_4$. Assume we can construct a cubic $C$ through $P_1,\ldots,P_5$ so that $C$ is tangent to $D$ (resp. $L$) at $P_1$ with multiplicity $4$ (resp. $3$).

For instance, let $D$ be the nodal cubic $y^2z-x^2(x+z)=0$. Then $P_5=(0:0:1)$ and we can let $P_1=(0:1:0)$ so that $L$ is the line $z=0$. Choosing $L'$ to be the line $x+y+z=0$ we have that $C$ is the cubic $xyz+xz^2+y^2z-x^3=0$.

Letting $B=2L+L'+D$ we have that the pencil generated by $B$ and $2C$ is a Halphen pencil of index two and the corresponding rational elliptic surface has a fiber of type $IV^*$.

Blowing-up $\mathbb{P}^2$ at the points $P_1^{(1)},\ldots,P_1^{(5)},P_2^{(1)},P_3^{(1)},P_4^{(1)},P_5^{(1)}$ we obtain the following (dual) configuration of rational curves:

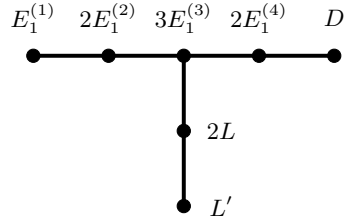
\begin{figure}[H]
\centering
\begin{tikzpicture}[line cap=round,line join=round,>=triangle 45,x=1.0cm,y=1.0cm]
\clip(-3.5,-3.) rectangle (3.,1.);
\draw [line width=1.5pt] (-2.,0.)-- (2.,0.);
\draw [line width=1.5pt] (0.,0.)-- (0.,-2.);
\begin{scriptsize}
\draw [fill=black] (0.,-2.) circle (2.5pt);
\draw[color=black] (0.5,-2.) node {$L'$};
\draw [fill=black] (0.,-1.) circle (2.5pt);
\draw[color=black] (0.5,-1.) node {$2L$};
\draw [fill=black] (-2.,0.) circle (2.5pt);
\draw[color=black] (-2.,.5) node {$E_1^{(1)}$};
\draw [fill=black] (-1.,0.) circle (2.5pt);
\draw[color=black] (-1.,.5) node {$2E_1^{(2)}$};
\draw [fill=black] (0.,0.) circle (2.5pt);
\draw[color=black] (0.,.5) node {$3E_1^{(3)}$};
\draw [fill=black] (1.,0.) circle (2.5pt);
\draw[color=black] (1.,0.5) node {$2E_1^{(4)}$};
\draw [fill=black] (2.,0.) circle (2.5pt);
\draw[color=black] (2.,.5) node {$D$};
\end{scriptsize}
\end{tikzpicture}
\end{figure}
\label{ivstar2-1,1-3,1-1}
\end{exen}

\begin{exen}[A double line and two conics]
Let $C$ be a smooth cubic. Let $P_2$ be a flex point. There exists a line $L$ through $P_2$ which is tangent to $C$ at another point $P_1$. Then $P_1$ is a sextactic (see Definition \ref{sextactic}) point of $C$.

In fact, by Lemma \ref{sum3d} we have
\[
2P_1\oplus P_2=0 \qquad 3P_2=0
\]
hence $3(2P_1\oplus P_2)=6P_1=0$, where $\oplus$ denotes the group law with another flex point taken as the origin. Again, using Lemma \ref{sum3d} we conclude there exists an osculating conic which is  tangent to $C$ with multiplicity at $P_1$.

Concretely, we can choose coordinates in $\mathbb{P}^2$ so that $C$ is the cubic given by 
\[
y^2z=x(x-z)(x-\alpha\cdot z) \quad \alpha \in \mathbb{C}\backslash \{0,1\}
\]
and $C$ has a flex point at $P_2=(0:1:0)$. The line $x=0$ is tangent to $C$  at $P_1=(0:0:1)$ and the flex $P_2$ is a point in that line.

Now, let $\varepsilon_2$ be a two torsion point of $C$. Using the same argument as in Example \ref{iostar2-1,1-2,1-1,1-1}, we can always find three points $P_3,P_4$ and $P_5$ in $C$ so that $P_3\oplus P_4 \oplus P_5=\varepsilon_2$. In particular, 
$2P_3\oplus 2P_4 \oplus 2P_5=0$ and we claim we must have
\begin{equation}
3P_1\oplus P_3 \oplus P_4 \oplus P_5=0
\label{eqc1}
\end{equation}
and 
\begin{equation}
P_1\oplus 2P_2 \oplus P_3 \oplus P_4 \oplus P_5=0
\label{eqc2}
\end{equation}

In fact, if one of these sums is non zero, then adding the two equations we obtain
\[
0 \neq 4P_1\oplus 2P_2\oplus 2P_3\oplus 2P_4 \oplus 2P_5=4P_1\oplus 2P_2=0
\]
a contradiction.

Applying Lemma \ref{sum3d} two Equations (\ref{eqc1}) and (\ref{eqc2}) we conclude there exists two conics $Q$ and $Q'$ so that: $P_1,P_3,P_4,P_5\in Q$, the cubic $C$ is tangent $Q$ at $P_1$ with multiplicity three, $P_1,P_2,P_3,P_4,P_5\in Q'$ and the cubic $C$ is tangent $Q$ at $P_2$ with multiplicity two.

Note that, by construction, $L$ is also tangent to $Q$ at $P_1$.

Letting $B=2L+Q+Q'$ we have that the pencil generated by $B$ and $2C$ is a Halphen pencil of index two and the corresponding rational elliptic has a fiber of type $IV^*$.

More precisely, blowing-up $\mathbb{P}^2$ at the points $P_1^{(1)},\ldots,P_1^{(4)},P_2^{(1)},P_2^{(2)},P_3^{(1)},P_4^{(1)},P_5^{(1)}$ we obtain the following (dual) configuration of rational curves:

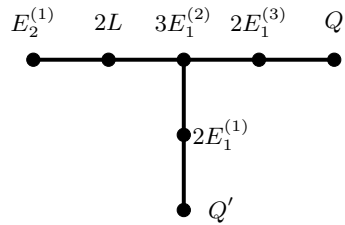
\begin{figure}[H]
\centering
\begin{tikzpicture}[line cap=round,line join=round,>=triangle 45,x=1.0cm,y=1.0cm]
\clip(-3.5,-3.) rectangle (3.,1.);
\draw [line width=1.5pt] (-2.,0.)-- (2.,0.);
\draw [line width=1.5pt] (0.,0.)-- (0.,-2.);
\begin{scriptsize}
\draw [fill=black] (0.,-2.) circle (2.5pt);
\draw[color=black] (0.5,-2.) node {$Q'$};
\draw [fill=black] (0.,-1.) circle (2.5pt);
\draw[color=black] (0.5,-1.) node {$2E_1^{(1)}$};
\draw [fill=black] (-2.,0.) circle (2.5pt);
\draw[color=black] (-2.,.5) node {$E_2^{(1)}$};
\draw [fill=black] (-1.,0.) circle (2.5pt);
\draw[color=black] (-1.,.5) node {$2L$};
\draw [fill=black] (0.,0.) circle (2.5pt);
\draw[color=black] (0.,.5) node {$3E_1^{(2)}$};
\draw [fill=black] (1.,0.) circle (2.5pt);
\draw[color=black] (1.,0.5) node {$2E_1^{(3)}$};
\draw [fill=black] (2.,0.) circle (2.5pt);
\draw[color=black] (2.,.5) node {$Q$};
\end{scriptsize}
\end{tikzpicture}
\end{figure}

\label{ivstar2-1,1-2,1-2}
\end{exen}

\begin{exen}[Two double lines and two other lines]
Let $Q$ be a smooth conic. And choose three distinct points on $Q$ say $P_1,P_2$ and $P_3$.  For each $i=1,2$ let $T_i$ be the tangent line to $Q$ at $P_i$. Let $L_i$ be the line joining $P_1$ and $P_i$, for $i=2,3$. And let $L$ be a line through $\{P_4\}=T_1\cap T_2 $ different than the $T_i$ and such that $P_3 \notin L$. Then $L$ intersects both $L_2$ and $L_3$ at two other points $P_5\in L_2$ and $P_6\in L_3$.

Letting $C$ be the cubic $Q+L$ and $B$ be the sextic $T_1+T_2+2L_2+2L_3$ we have that the pencil $\mathcal{P}$ generated by $B$ and $2C$ is a Halphen pencil of index two which yields a fiber of type $IV^*$ in the associated elliptic surface. 

Blowing-up $\mathbb{P}^2$ at the nine points $
P_1^{(1)},P_1^{(2)},P_1^{(3)},P_2^{(1)},P_2^{(2)},P_3^{(1)},P_4^{(1)},P_5^{(1)},P_6^{(1)}$
we obtain the following (dual) configuration of rational curves:

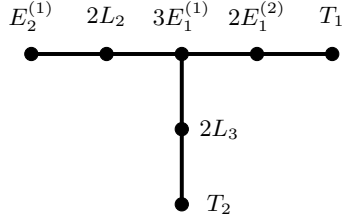
\begin{figure}[H]
\centering
\begin{tikzpicture}[line cap=round,line join=round,>=triangle 45,x=1.0cm,y=1.0cm]
\clip(-3.5,-3.) rectangle (3.,1.);
\draw [line width=1.5pt] (-2.,0.)-- (2.,0.);
\draw [line width=1.5pt] (0.,0.)-- (0.,-2.);
\begin{scriptsize}
\draw [fill=black] (0.,-2.) circle (2.5pt);
\draw[color=black] (0.5,-2.) node {$T_2$};
\draw [fill=black] (0.,-1.) circle (2.5pt);
\draw[color=black] (0.5,-1.) node {$2L_3$};
\draw [fill=black] (-2.,0.) circle (2.5pt);
\draw[color=black] (-2.,.5) node {$E_2^{(1)}$};
\draw [fill=black] (-1.,0.) circle (2.5pt);
\draw[color=black] (-1.,.5) node {$2L_2$};
\draw [fill=black] (0.,0.) circle (2.5pt);
\draw[color=black] (0.,.5) node {$3E_1^{(1)}$};
\draw [fill=black] (1.,0.) circle (2.5pt);
\draw[color=black] (1.,0.5) node {$2E_1^{(2)}$};
\draw [fill=black] (2.,0.) circle (2.5pt);
\draw[color=black] (2.,.5) node {$T_1$};
\end{scriptsize}
\end{tikzpicture}
\end{figure}

\label{ivstar2-1,2-1,1-1,1-1}
\end{exen}

\begin{exen}[Two double lines and a conic]
Let $Q$ be a smooth conic. And choose three distinct points on $Q$ say $P_1,P_2$ and $P_3$.  For each $i=1,2,3$ let $L_i$ be the tangent line to $Q$ at $P_i$. Let $L$ (resp. $R$) be the lines joining $P_1$ and $P_3$ (resp. $P_2$ and $P_3$). And let $\{P_4\}= L\cap L_2$ and $\{P_5\}= R\cap L_1$. 

Then the cubic $C=L_1+L_2+L_3$ is such that the intersection multiplicity of $Q$ and $C$ at $P_i$, for $i=1,2,3$ is two and the pencil $\mathcal{P}$ generated by $B=Q+2L+2R$ and $2C$ is a Halphen pencil of index two which yields a fiber of type $IV^*$ in the associated elliptic surface. In fact the Jacobian fibration of such surface is the surface $X_{431}$ in Miranda and Persson's list \cite{extr}.

Concretely, we can choose coordinates in $\mathbb{P}^2$ so that $Q$ is given by $
x^2-yz=0$,  $P_1=(0:0:1),P_2=(0:1:0)$ and $P_3=(1:-1:-1)$. Then $L_1$ is the line $y=0$, $L_2$ is the line $z=0$ and $L_3$ is the line $2x+y+z=0$. And, therefore, $L$ and $R$ are the lines $x+y=0$ and $x+z=0$, respectively. Moreover, $P_4=(1:-1:0)$ and $P_5=(1:0:-1)$.

Blowing-up the nine points $
P_1^{(1)},P_1^{(2)},P_2^{(1)},P_2^{(2)},P_3^{(1)},P_3^{(2)},P_3^{(3)},P_4^{(1)},P_5^{(1)}
$ we obtain the following (dual) configuration of rational curves:

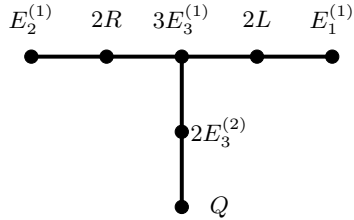
\begin{figure}[H]
\centering
\begin{tikzpicture}[line cap=round,line join=round,>=triangle 45,x=1.0cm,y=1.0cm]
\clip(-3.5,-3.) rectangle (3.,1.);
\draw [line width=1.5pt] (-2.,0.)-- (2.,0.);
\draw [line width=1.5pt] (0.,0.)-- (0.,-2.);
\begin{scriptsize}
\draw [fill=black] (0.,-2.) circle (2.5pt);
\draw[color=black] (0.5,-2.) node {$Q$};
\draw [fill=black] (0.,-1.) circle (2.5pt);
\draw[color=black] (0.5,-1.) node {$2E_3^{(2)}$};
\draw [fill=black] (-2.,0.) circle (2.5pt);
\draw[color=black] (-2.,.5) node {$E_2^{(1)}$};
\draw [fill=black] (-1.,0.) circle (2.5pt);
\draw[color=black] (-1.,.5) node {$2R$};
\draw [fill=black] (0.,0.) circle (2.5pt);
\draw[color=black] (0.,.5) node {$3E_3^{(1)}$};
\draw [fill=black] (1.,0.) circle (2.5pt);
\draw[color=black] (1.,0.5) node {$2L$};
\draw [fill=black] (2.,0.) circle (2.5pt);
\draw[color=black] (2.,.5) node {$E_1^{(1)}$};
\end{scriptsize}
\end{tikzpicture}
\end{figure}

\label{ivstar2-1,2-1,1-2}
\end{exen}

\begin{exen}[A double conic and two lines]
Let $C$ be a smooth cubic. Let $L_1$ be an inflection line of $C$ at a point $P_1$ and choose a line $L_2$ through $P_1$ which is tangent to $C$ at another point $P_2$. We can construct a conic $Q$ through $P_1$ and $P_2$ so that $Q$ is tangent to $C$ at $P_1$ with multiplicity two and $Q$ meets $C$ transversally at $P_2$. Moreover, $Q$ intersects $C$ at other three points, say $P_3,P_4$ and $P_5$.

Concretely, choose coordinates in $\mathbb{P}^2$ so that $C$ is the cubic given by 
\[
y^2z=x(x-z)(x-\alpha\cdot z) \quad \alpha \in \mathbb{C}\backslash \{0,1\}
\]
Then we can let $L_1$ be the line $z=0$ and hence $P_1=(0:1:0)$ and we can let $L_2$ be either one of the  lines $x=0, x-z=0$ or $x-\alpha\cdot z=0$. 

If we choose $L_2$ as $x=0$, then $P_2 = (0:0:1)$ and, similarly, if we take $L_2$ as $x-z=0$ (resp. $x-\alpha\cdot z=0$), then $P_2=(1:0:1)$  (resp.  $P_2=(\alpha:0:1)$). 

Say we choose  $L_2$ to be the line $x=0$, then we can let $Q$ be the conic $x^2+yz=0$.

Now, the pencil $\mathcal{P}$ generated by $B=2Q+L_1+L_2$ and $2C$ is a Halphen pencil of index two that yields a fiber of type $IV^*$ in the corresponding rational elliptic surface.

More precisely, blowing-up $
P_1^{(1)},\ldots,P_1^{(4)},P_2^{(1)},P_2^{(2)},P_3^{(1)},P_4^{(1)},P_5^{(1)}
$ we obtain the following (dual) configuration of rational curves:

\begin{figure}[H]
\centering
\begin{tikzpicture}[line cap=round,line join=round,>=triangle 45,x=1.0cm,y=1.0cm]
\clip(-3.5,-3.) rectangle (3.,1.);
\draw [line width=1.5pt] (-2.,0.)-- (2.,0.);
\draw [line width=1.5pt] (0.,0.)-- (0.,-2.);
\begin{scriptsize}
\draw [fill=black] (0.,-2.) circle (2.5pt);
\draw[color=black] (0.5,-2.) node {$L_1$};
\draw [fill=black] (0.,-1.) circle (2.5pt);
\draw[color=black] (0.5,-1.) node {$2E_1^{(3)}$};
\draw [fill=black] (-2.,0.) circle (2.5pt);
\draw[color=black] (-2.,.5) node {$E_2^{(1)}$};
\draw [fill=black] (-1.,0.) circle (2.5pt);
\draw[color=black] (-1.,.5) node {$2Q$};
\draw [fill=black] (0.,0.) circle (2.5pt);
\draw[color=black] (0.,.5) node {$3E_1^{(2)}$};
\draw [fill=black] (1.,0.) circle (2.5pt);
\draw[color=black] (1.,0.5) node {$2E_1^{(1)}$};
\draw [fill=black] (2.,0.) circle (2.5pt);
\draw[color=black] (2.,.5) node {$L_2$};
\end{scriptsize}
\end{tikzpicture}
\end{figure}

\label{ivstar2-2,1-1,1-1}
\end{exen}

\begin{exe}[A triple conic {\cite[I.5.11]{pisa}}] In this example we consider a rational elliptic surface of index two whose Jacobian is the surface $X_{431}$ in  Miranda and Persson's list \cite{extr}.

Let $Q\subset \mathbb{P}^2$ be a smooth conic and choose three distinct points $P_1,P_2$ and $P_3$ on $Q$. Let $L_i$ be the line tangent to $Q$ at $P_i$ and consider the pencil generated by $B=3Q$ and $2C$, where $C=L_1+L_2+L_3$. Then the associated rational elliptic surface has as its Jacobian the surface $X_{431}$.

Note that we need to blow-up each of the three points three times. That is, to construct the desired surface we blow-up $\mathbb{P}^2$ at $
P_1^{(1)},P_1^{(2)},P_1^{(3)},P_2^{(1)},P_2^{(2)},P_2^{(3)},P_3^{(1)},P_3^{(2)},P_3^{(3)}
$, which produces three disjoint chains of $(-2)$-curves, each of length $2$ and formed by exceptional divisors over the corresponding three points. Pictorially, 

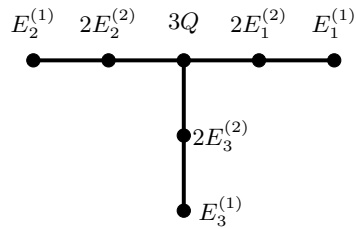
\begin{figure}[H]
\centering
\begin{tikzpicture}[line cap=round,line join=round,>=triangle 45,x=1.0cm,y=1.0cm]
\clip(-3.5,-3.) rectangle (3.,1.);
\draw [line width=1.5pt] (-2.,0.)-- (2.,0.);
\draw [line width=1.5pt] (0.,0.)-- (0.,-2.);
\begin{scriptsize}
\draw [fill=black] (0.,-2.) circle (2.5pt);
\draw[color=black] (0.5,-2.) node {$E_3^{(1)}$};
\draw [fill=black] (0.,-1.) circle (2.5pt);
\draw[color=black] (0.5,-1.) node {$2E_3^{(2)}$};
\draw [fill=black] (-2.,0.) circle (2.5pt);
\draw[color=black] (-2.,.5) node {$E_2^{(1)}$};
\draw [fill=black] (-1.,0.) circle (2.5pt);
\draw[color=black] (-1.,.5) node {$2E_2^{(2)}$};
\draw [fill=black] (0.,0.) circle (2.5pt);
\draw[color=black] (0.,.5) node {$3Q$};
\draw [fill=black] (1.,0.) circle (2.5pt);
\draw[color=black] (1.,0.5) node {$2E_1^{(2)}$};
\draw [fill=black] (2.,0.) circle (2.5pt);
\draw[color=black] (2.,.5) node {$E_1^{(1)}$};
\end{scriptsize}
\end{tikzpicture}
\end{figure}

\label{exeivstar}
\end{exe}

\begin{exen}[A triple line, a conic and another line]
Choose two (distinct) lines $L_1$ and $L_2$ and a smooth conic $Q$ in general position. Let $\{P_2\}=L_1\cap L_2$, let $\{P_2,P_4\}=L_1\cap Q$ and let $\{P_1,P_3\}=L_3\cap Q$. We can find a cubic $C$ so that $P_1,P_2,P_3,P_4,P_5\in C$ and $C$ is tangent to $Q$ at $P_3$ with multiplicity three.

Concretely, we can choose coordinates in $\mathbb{P}^2$ so that $Q$ is the conic $x^2+yz+xz=0$ and $L_1$ and $L_2$ are the lines $x+2y+z=0$ and $x=0$, respectively. 

Then $P_1=(0:1:0),P_2=(0:1:-2),P_3=(0:0:1),P_4=(1:0:-1)$ and $P_5=(1:-1:1)$ and we have that $C$ is the cubic given by
\[
xy(x+z)+(x^2+yz+xz)(2y+z)=0
\]
Now, the pencil generated by $B=Q+L_1+3L_2$ and $2C$ is a Halphen pencil of index two which yields a fiber of type $IV^*$ in the associated elliptic surface.

If we blow-up $\mathbb{P}^2$ at the nine points $
P_1^{(1)},P_1^{(2)},P_2^{(1)},P_2^{(2)},P_3^{(1)},P_3^{(2)},P_3^{(3)},P_4^{(1)},P_5^{(1)}
$ we obtain the following (dual) configuration of rational curves:

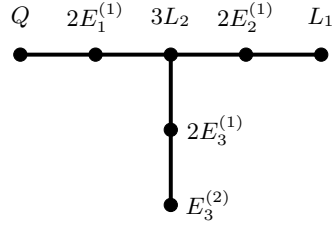
\begin{figure}[H]
\centering
\begin{tikzpicture}[line cap=round,line join=round,>=triangle 45,x=1.0cm,y=1.0cm]
\clip(-3.5,-3.) rectangle (3.,1.);
\draw [line width=1.5pt] (-2.,0.)-- (2.,0.);
\draw [line width=1.5pt] (0.,0.)-- (0.,-2.);
\begin{scriptsize}
\draw [fill=black] (0.,-2.) circle (2.5pt);
\draw[color=black] (0.5,-2.) node {$E_3^{(2)}$};
\draw [fill=black] (0.,-1.) circle (2.5pt);
\draw[color=black] (0.6,-1.) node {$2E_3^{(1)}$};
\draw [fill=black] (-2.,0.) circle (2.5pt);
\draw[color=black] (-2.,.5) node {$Q$};
\draw [fill=black] (-1.,0.) circle (2.5pt);
\draw[color=black] (-1.,.5) node {$2E_1^{(1)}$};
\draw [fill=black] (0.,0.) circle (2.5pt);
\draw[color=black] (0.,.5) node {$3L_2$};
\draw [fill=black] (1.,0.) circle (2.5pt);
\draw[color=black] (1.,0.5) node {$2E_2^{(1)}$};
\draw [fill=black] (2.,0.) circle (2.5pt);
\draw[color=black] (2.,.5) node {$L_1$};
\end{scriptsize}
\end{tikzpicture}
\end{figure}

\label{ivstar3-1,1-2,1-1}
\end{exen}

\begin{exen}[A triple line, a double line and another line]
Let $C$ be a smooth cubic and let $L_1$ be an inflection line of $C$ at a point $P_1$. We can choose another line $L_2$ through $P_1$ which is tangent to $C$ at another point $P_2$. Let  $L_3$ be a third line which intersects $C$ at three distinct points, say $P_3,P_4$ and $P_5$, all different than $P_1$ and $P_2$. Then the pencil $\mathcal{P}$ generated by $B=L_1+3L_2+2L_3$ and $2C$ is a Halphen pencil of index two and it yields a fiber of type $IV^*$ in the corresponding elliptic surface.

Concretely, we can choose coordinates in $\mathbb{P}^2$ so that $C$ is the cubic given by
\[
y^2z=x(x-z)(x-\alpha\cdot z) \qquad \alpha\in \mathbb{C}\backslash\{0,1\}
\]
we can let $L_1$ be the line $z=0$ (hence $P_1 = (0:1:0)$)  and we can choose $L_2$ to be either one of the  lines $x=0, x-z=0$ or $x-\alpha\cdot z=0$. 

If we choose $L_2$ as $x=0$, then $P_2 = (0:0:1)$ and we can let $L_3$ be the line $x+y+z=0$.

Now, if we blow-up $\mathbb{P}^2$ at the nine base points
\[
P_1^{(1)},P_1^{(2)},P_1^{(3)},P_2^{(1)},P_2^{(2)},P_2^{(3)},P_3^{(1)},P_4^{(1)},P_5^{(1)}
\]
we obtain the following (dual) configuration of rational curves:

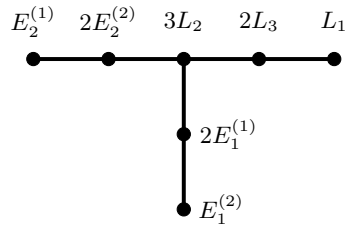
\begin{figure}[H]
\centering
\begin{tikzpicture}[line cap=round,line join=round,>=triangle 45,x=1.0cm,y=1.0cm]
\clip(-3.5,-3.) rectangle (3.,1.);
\draw [line width=1.5pt] (-2.,0.)-- (2.,0.);
\draw [line width=1.5pt] (0.,0.)-- (0.,-2.);
\begin{scriptsize}
\draw [fill=black] (0.,-2.) circle (2.5pt);
\draw[color=black] (0.5,-2.) node {$E_1^{(2)}$};
\draw [fill=black] (0.,-1.) circle (2.5pt);
\draw[color=black] (0.6,-1.) node {$2E_1^{(1)}$};
\draw [fill=black] (-2.,0.) circle (2.5pt);
\draw[color=black] (-2.,.5) node {$E_2^{(1)}$};
\draw [fill=black] (-1.,0.) circle (2.5pt);
\draw[color=black] (-1.,.5) node {$2E_2^{(2)}$};
\draw [fill=black] (0.,0.) circle (2.5pt);
\draw[color=black] (0.,.5) node {$3L_2$};
\draw [fill=black] (1.,0.) circle (2.5pt);
\draw[color=black] (1.,0.5) node {$2L_3$};
\draw [fill=black] (2.,0.) circle (2.5pt);
\draw[color=black] (2.,.5) node {$L_1$};
\end{scriptsize}
\end{tikzpicture}
\end{figure}

\label{ivstar3-1,2-1,1-1}
\end{exen}

\begin{exen}[A triple line and three more lines]
Consider four (distinct) lines $L_1,L_2,L_3$ and $L_4$ in general position. That is, such that the $L_i$ determine six intersection points, say $P_1,\ldots,P_6$. Now, choose a cubic $C$ through these six points so that $C$ intersects each of the lines transversally, i.e. the $L_i$ are not tangent lines to $C$.

The pencil $\mathcal{P}$ generated by $B=L_1+L_2+L_3+3L_4$ and $2C$ is a Halphen pencil of index two and it yields a fiber of type $IV^*$ in the corresponding rational elliptic surface. 

Concretely, we blow-up $\mathbb{P}^2$ at the base points 
$
P_1^{(1)},P_1^{(2)},P_2^{(1)},P_2^{(2)},P_3^{(1)},P_3^{(2)},P_4^{(1)},P_5^{(1)},P_6^{(1)}
$ of $\mathcal{P}$ so that we obtain the following (dual) configuration of rational curves:

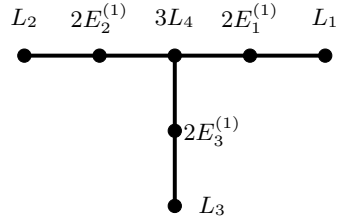
\begin{figure}[H]
\centering
\begin{tikzpicture}[line cap=round,line join=round,>=triangle 45,x=1.0cm,y=1.0cm]
\clip(-3.5,-3.) rectangle (3.,1.);
\draw [line width=1.5pt] (-2.,0.)-- (2.,0.);
\draw [line width=1.5pt] (0.,0.)-- (0.,-2.);
\begin{scriptsize}
\draw [fill=black] (0.,-2.) circle (2.5pt);
\draw[color=black] (0.5,-2.) node {$L_3$};
\draw [fill=black] (0.,-1.) circle (2.5pt);
\draw[color=black] (0.5,-1.) node {$2E_3^{(1)}$};
\draw [fill=black] (-2.,0.) circle (2.5pt);
\draw[color=black] (-2.,.5) node {$L_2$};
\draw [fill=black] (-1.,0.) circle (2.5pt);
\draw[color=black] (-1.,.5) node {$2E_2^{(1)}$};
\draw [fill=black] (0.,0.) circle (2.5pt);
\draw[color=black] (0.,.5) node {$3L_4$};
\draw [fill=black] (1.,0.) circle (2.5pt);
\draw[color=black] (1.,0.5) node {$2E_1^{(1)}$};
\draw [fill=black] (2.,0.) circle (2.5pt);
\draw[color=black] (2.,.5) node {$L_1$};
\end{scriptsize}
\end{tikzpicture}
\end{figure}

If we take $B$ and $C$ as in the picture below, then we obtain a rational elliptic surface whose Jacobian is the surface $X_{431}$ in  Miranda and Persson's list \cite{extr}.

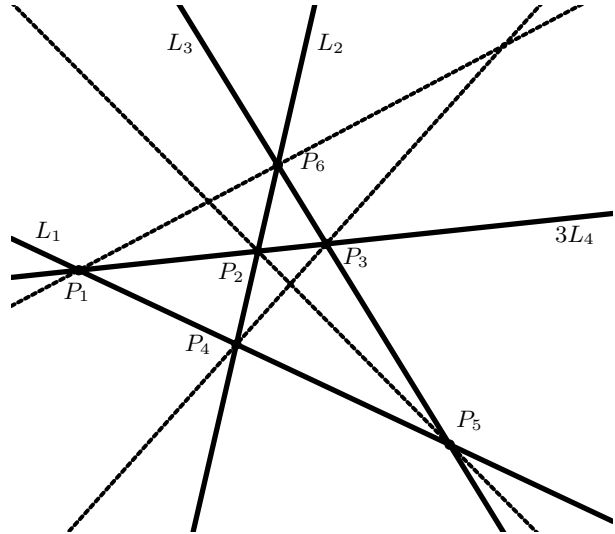
\begin{figure}[H]
\centering
\definecolor{uuuuuu}{rgb}{0,0,0}
\definecolor{ffzzqq}{rgb}{0,0,0}
\definecolor{qqwwzz}{rgb}{0,0,0}
\definecolor{qqzzqq}{rgb}{0,0,0}
\begin{tikzpicture}[line cap=round,line join=round,>=triangle 45,x=1.0cm,y=1.0cm,scale=0.5]
\clip(-8.,-4.) rectangle (8.,10.);
\draw [line width=2.pt,color=qqzzqq,domain=-8.:8.] plot(\x,{(--23.8168--0.7*\x)/6.58});
\draw [line width=2.pt,color=qqwwzz,domain=-8.:8.] plot(\x,{(--14.034-5.34*\x)/3.28});
\draw [line width=2.pt,color=qqwwzz,domain=-8.:8.] plot(\x,{(-0.4176--4.64*\x)/-9.86});
\draw [line width=2.pt,color=qqwwzz,domain=-8.:8.] plot(\x,{(-10.677209899381689-4.771255215290925*\x)/-1.100366901351058});
\draw [line width=1.5pt,dash pattern=on 1.5pt off 1.5pt,color=ffzzqq,domain=-8.:8.] plot(\x,{(-7.73181405659147-2.671803364193506*\x)/-2.389917851088798});
\draw [line width=1.5pt,dash pattern=on 1.5pt off 1.5pt,color=ffzzqq,domain=-8.:8.] plot(\x,{(--10.271181167054962-5.146558768612737*\x)/5.098347575040273});
\draw [line width=1.5pt,dash pattern=on 1.5pt off 1.5pt,color=ffzzqq,domain=-8.:8.] plot(\x,{(-33.016330665580284-2.7994518510974187*\x)/-5.29044905026226});
\begin{scriptsize}
\draw [fill=black] (-6.2,2.96) circle (3.5pt);
\draw[color=black] (-6.25,2.38) node {$P_1$};
\draw [fill=black] (0.38,3.66) circle (3.5pt);
\draw[color=black] (1.17,3.3) node {$P_3$};
\draw[color=qqzzqq] (7,3.9) node {$3L_4$};
\draw [fill=black] (3.66,-1.68) circle (3.5pt);
\draw[color=black] (4.19,-1.04) node {$P_5$};
\draw[color=qqwwzz] (-3.5,9) node {$L_3$};
\draw[color=qqwwzz] (-7,4) node {$L_1$};
\draw [fill=black] (-0.9095509497377404,5.759451851097419) circle (3.5pt);
\draw[color=black] (0.03,5.84) node {$P_6$};
\draw [fill=black] (-2.0099178510887983,0.9881966358064938) circle (3.5pt);
\draw[color=black] (-3.03,0.98) node {$P_4$};
\draw[color=qqwwzz] (0.5,9) node {$L_2$};
\draw [fill=black] (-1.4383475750402723,3.466558768612737) circle (3.5pt);
\draw[color=black] (-2.15,2.9) node {$P_2$};
\draw [fill=uuuuuu] (-2.7467243181060432,4.787307834965069) circle (2.0pt);
\draw [fill=uuuuuu] (5.104593093313228,8.941848376256457) circle (2.0pt);
\draw [fill=uuuuuu] (-0.5737368146161825,2.5937720421286414) circle (2.0pt);
\end{scriptsize}
\end{tikzpicture}
\caption{The two generators of $\mathcal{P}$ yielding a fiber of type $IV^*$ and a multiple fiber of type $I_3$}
\end{figure}
\label{ivstar3-1,1-1,1-1,1-1}
\end{exen}

\begin{exen}[A triple line and a cubic]
Let $D:d=0$ be a nodal cubic with node at a point $P_4$. Let $L_1:l_1=0$ and $L_2:l_2=0$ be two of its inflections lines at points $P_1$ and $P_2$ ($\neq P_4$), respectively.  And let $L_3$ be a line through the node $P_4$ which does not contain the flex points $P_1$ and $P_2$. Then the cubic $C$ given by $l_1l_2l_3+d=0$ is such that the intersection multiplicity of $D$ and $C$ at $P_i$ for $i=1,2$ is
\[
I_{P_i}(C,D)=I_{P_i}(l_i,d)=3
\]
and, by construction, the node $P_4$ lies on it.

Now let $L$ be the line joining $P_1$ and $P_2$. Then $L$ intersects $D$ at a third (flex) point $P_3$ and we have that the pencil $\mathcal{P}$ generated by $B=D+3L$ and $2C$ is a Halphen pencil of index two which yields a fiber of type $IV^*$ in the corresponding elliptic surface.

Blowing-up $\mathbb{P}^2$ at the nine base points $
P_1^{(1)},P_1^{(2)},P_1^{(3)},P_2^{(1)},P_2^{(2)},P_2^{(3)},P_3^{(1)},P_3^{(2)},P_4^{(1)}
$ we obtain the following (dual) configuration of rational curves:

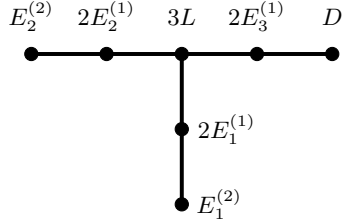
\begin{figure}[H]
\centering
\begin{tikzpicture}[line cap=round,line join=round,>=triangle 45,x=1.0cm,y=1.0cm]
\clip(-3.5,-3.) rectangle (3.,1.);
\draw [line width=1.5pt] (-2.,0.)-- (2.,0.);
\draw [line width=1.5pt] (0.,0.)-- (0.,-2.);
\begin{scriptsize}
\draw [fill=black] (0.,-2.) circle (2.5pt);
\draw[color=black] (0.5,-2.) node {$E_1^{(2)}$};
\draw [fill=black] (0.,-1.) circle (2.5pt);
\draw[color=black] (0.6,-1.) node {$2E_1^{(1)}$};
\draw [fill=black] (-2.,0.) circle (2.5pt);
\draw[color=black] (-2.,.5) node {$E_2^{(2)}$};
\draw [fill=black] (-1.,0.) circle (2.5pt);
\draw[color=black] (-1.,.5) node {$2E_2^{(1)}$};
\draw [fill=black] (0.,0.) circle (2.5pt);
\draw[color=black] (0.,.5) node {$3L$};
\draw [fill=black] (1.,0.) circle (2.5pt);
\draw[color=black] (1.,0.5) node {$2E_3^{(1)}$};
\draw [fill=black] (2.,0.) circle (2.5pt);
\draw[color=black] (2.,.5) node {$D$};
\end{scriptsize}
\end{tikzpicture}
\end{figure}

\label{ivstar3-1,1-3}
\end{exen}

\subsection{Type $III^*$}

We now construct all possible examples of Halphen pencils of index two that yield a fiber of type $III^*$ in the corresponding rational elliptic surface (Theorem \ref{allpossibleiiistar}).

\begin{exen}[A double line, a cubic and another line]
Let $D$ be a nodal cubic and denote its node by $P_1$. Let $P_2$ be a flex point of $D$ and denote the corresponding inflection line by $L_1$. Let $L_2$ be a line through $P_2$ so that $L_2$ intersects $D$ at two other points, say $P_3$ and $P_4$. We can construct a cubic $C$ through $P_1,\ldots,P_4$ so that $C$ is tangent to $D$ (resp. $L_1$) at $P_2$ with multiplicity five (resp. three).

Concretely, let $D$ be the nodal cubic given by $y^2z-x^2(x+z)=0$. Then $P_1=(0:0:1)$ and we can let $L_1$ be the line $z=0$, hence $P_2=(0:1:0)$. Thus we can take $L_2$ to be the line $x-z=0$. And, further, we have that $P_4=(1:\sqrt{2}:1)$ and $P_5=(1:-\sqrt{2}:1)$. Choosing $C$ to be the cubic given by
\[
y^2z-x(x^2+z^2)=0
\]
we have that all the points $P_1,\ldots,P_4$ lie in $C$ and, moreover, the intersection multiplicity of $C$ and $D$ (resp. $L_1$) at $P_2$ is five (resp. three).

Now, the pencil $\mathcal{P}$ generated by $B=D+2L_1+L_2$ and $2C$ is a Halphen pencil of index two which yields a fiber of type $III^*$ in the corresponding rational elliptic surface. More precisely, blowing-up $
P_1^{(1)},P_2^{(1)}, \ldots, P_2^{(6)},P_3^{(1)},P_4^{(1)}
$ we obtain the following (dual) configuration of rational curves

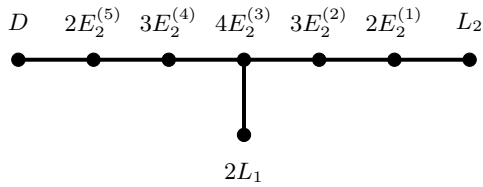
\begin{figure}[H]
\centering
\begin{tikzpicture}[line cap=round,line join=round,>=triangle 45,x=1.0cm,y=1.0cm]
\clip(-4.,-2.) rectangle (4.,1.);
\draw [line width=1.5pt] (-3.,0.)-- (3.,0.);
\draw [line width=1.5pt] (0.,0.)-- (0.,-1.);
\begin{scriptsize}
\draw [fill=black] (0.,-1.) circle (2.5pt);
\draw[color=black] (0.,-1.5) node {$2L_1$};
\draw [fill=black] (-3.,0.) circle (2.5pt);
\draw[color=black] (-3.,.5) node {$D$};
\draw [fill=black] (-2.,0.) circle (2.5pt);
\draw[color=black] (-2.,.5) node {$2E_2^{(5)}$};
\draw [fill=black] (-1.,0.) circle (2.5pt);
\draw[color=black] (-1.,.5) node {$3E_2^{(4)}$};
\draw [fill=black] (0.,0.) circle (2.5pt);
\draw[color=black] (0.,.5) node {$4E_2^{(3)}$};
\draw [fill=black] (1.,0.) circle (2.5pt);
\draw[color=black] (1.,0.5) node {$3E_2^{(2)}$};
\draw [fill=black] (2.,0.) circle (2.5pt);
\draw[color=black] (2.,.5) node {$2E_2^{(1)}$};
\draw [fill=black] (3.,0.) circle (2.5pt);
\draw[color=black] (3.,0.5) node {$L_2$};
\end{scriptsize}
\end{tikzpicture}
\end{figure}

\label{iiistar2-1,1-3,1-1}
\end{exen}

\begin{exen}[A double conic and another conic]
Let $Q$ be a conic and choose a point $P_1\in Q$. We can construct another conic $Q'$ and a smooth cubic $C$ so that
$Q$ is tangent to both  $C$ and $Q'$ at $P_1$ with full multiplicity and, moreover, the intersection multiplicity of $Q'$ and $C$ at $P_1$ is four and $Q'$ intersects $C$ at two other points, say $P_2$ and $P_3$.
 
Concretely, choose coordinates in $\mathbb{P}^2$ so that $Q$ is the conic given by $x^2+yz=0$ and let $P_1$ be the point $(0:0:1)$. Then we can let $Q'$ be the conic given by $x^2+yz+y^2=0$ and we can let $C$ be the cubic given by
\[
y^3+z(x^2+yz)=0
\]
Thus, $P_2=(\alpha:1:1)$ and $P_3=(-\alpha:1:1)$, where $\alpha^2+2=0$

Now, the pencil $\mathcal{P}$ generated by $B=2Q'+Q$ and $2C$ is a Halphen pencil of index two such that the corresponding elliptic surface has a fiber of type $III^*$.

Blowing-up $\mathbb{P}^2$ at $
P_1^{(1)},\ldots,P_1^{(7)},P_2^{(1)},P_3^{(1)}
$ we obtain the following (dual) configuration of rational curves:

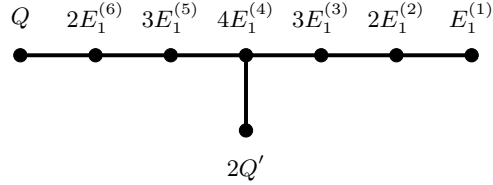
\begin{figure}[H]
\centering
\begin{tikzpicture}[line cap=round,line join=round,>=triangle 45,x=1.0cm,y=1.0cm]
\clip(-4.,-2.) rectangle (4.,1.);
\draw [line width=1.5pt] (-3.,0.)-- (3.,0.);
\draw [line width=1.5pt] (0.,0.)-- (0.,-1.);
\begin{scriptsize}
\draw [fill=black] (0.,-1.) circle (2.5pt);
\draw[color=black] (0.,-1.5) node {$2Q'$};
\draw [fill=black] (-3.,0.) circle (2.5pt);
\draw[color=black] (-3.,.5) node {$Q$};
\draw [fill=black] (-2.,0.) circle (2.5pt);
\draw[color=black] (-2.,.5) node {$2E_1^{(6)}$};
\draw [fill=black] (-1.,0.) circle (2.5pt);
\draw[color=black] (-1.,.5) node {$3E_1^{(5)}$};
\draw [fill=black] (0.,0.) circle (2.5pt);
\draw[color=black] (0.,.5) node {$4E_1^{(4)}$};
\draw [fill=black] (1.,0.) circle (2.5pt);
\draw[color=black] (1.,0.5) node {$3E_1^{(3)}$};
\draw [fill=black] (2.,0.) circle (2.5pt);
\draw[color=black] (2.,.5) node {$2E_1^{(2)}$};
\draw [fill=black] (3.,0.) circle (2.5pt);
\draw[color=black] (3.,0.5) node {$E_1^{(1)}$};
\end{scriptsize}
\end{tikzpicture}
\end{figure}

\label{iiistar2-2,1-2}
\end{exen}

\begin{exen}[A triple conic] In this new example we construct a rational elliptic surface whose Jacobian is the surface $X_{321}$ in  Miranda and Persson's list \cite{extr}.

Let $Q\subset \mathbb{P}^2$ be a (smooth) conic. Then, there exists a line $L$ (resp. a conic $R$) that is  tangent to $Q$ with full multiplicity $2$ (resp. 4). In fact we can assume we have determined two distinct intersection points this way. Now, generically, $L$ intersects $R$ at two other points.

Letting $C=L+R$ and $B=3Q$ we have that the pencil generated by $B$ and $2C$ is a Halphen npencil of index two. In particular, blowing-up $\mathbb{P}^2$ at the nine base points
$
P_1^{(1)},\ldots,P_1^{(3)},P_2^{(1)},\ldots,P_2^{(6)}
$ we obtain a rational elliptic surface of index two. And such surface has a type $III^*$ singular fiber.

In fact, we obtain the following (dual) configuration of rational curves:
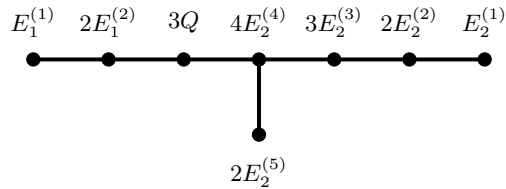
\begin{figure}[H]
\centering
\begin{tikzpicture}[line cap=round,line join=round,>=triangle 45,x=1.0cm,y=1.0cm]
\clip(-4.,-2.) rectangle (4.,1.);
\draw [line width=1.5pt] (-3.,0.)-- (3.,0.);
\draw [line width=1.5pt] (0.,0.)-- (0.,-1.);
\begin{scriptsize}
\draw [fill=black] (0.,-1.) circle (2.5pt);
\draw[color=black] (0.,-1.5) node {$2E_2^{(5)}$};
\draw [fill=black] (-3.,0.) circle (2.5pt);
\draw[color=black] (-3.,.5) node {$E_1^{(1)}$};
\draw [fill=black] (-2.,0.) circle (2.5pt);
\draw[color=black] (-2.,.5) node {$2E_1^{(2)}$};
\draw [fill=black] (-1.,0.) circle (2.5pt);
\draw[color=black] (-1.,.5) node {$3Q$};
\draw [fill=black] (0.,0.) circle (2.5pt);
\draw[color=black] (0.,.5) node {$4E_2^{(4)}$};
\draw [fill=black] (1.,0.) circle (2.5pt);
\draw[color=black] (1.,0.5) node {$3E_2^{(3)}$};
\draw [fill=black] (2.,0.) circle (2.5pt);
\draw[color=black] (2.,.5) node {$2E_2^{(2)}$};
\draw [fill=black] (3.,0.) circle (2.5pt);
\draw[color=black] (3.,0.5) node {$E_2^{(1)}$};
\end{scriptsize}
\end{tikzpicture}
\end{figure}

\label{exeiiistar}
\end{exen}

\begin{exen}[Two triple lines]
Consider two (distinct) lines $L_1$ and $L_2$ and let $P_3$ be their intersection point. Choose a cubic $C$ which intersects $L_1$ and $L_2$ at $P_3$ with multiplicity one and which is tangent to each $L_i$ at a point $P_i$ (with multiplicity two). The pencil $\mathcal{P}$ generated by $B=3L_1+3L_2$ and $2C$ is a Halphen pencil of index two and it yields a fiber of type $III^*$ in the corresponding rational elliptic surface. 

Concretely, we blow-up $\mathbb{P}^2$ at the base points 
$
P_1^{(1)},P_1^{(2)},P_1^{(3)},P_2^{(1)},P_2^{(2)},P_2^{(3)},P_3^{(1)},P_3^{(2)},P_3^{(3)}
$ of $\mathcal{P}$ so that we obtain the following (dual) configuration of rational curves:

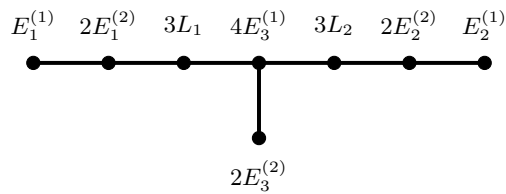
\begin{figure}[H]
\centering
\begin{tikzpicture}[line cap=round,line join=round,>=triangle 45,x=1.0cm,y=1.0cm]
\clip(-4.,-2.) rectangle (4.,1.);
\draw [line width=1.5pt] (-3.,0.)-- (3.,0.);
\draw [line width=1.5pt] (0.,0.)-- (0.,-1.);
\begin{scriptsize}
\draw [fill=black] (0.,-1.) circle (2.5pt);
\draw[color=black] (0.,-1.5) node {$2E_3^{(2)}$};
\draw [fill=black] (-3.,0.) circle (2.5pt);
\draw[color=black] (-3.,.5) node {$E_1^{(1)}$};
\draw [fill=black] (-2.,0.) circle (2.5pt);
\draw[color=black] (-2.,.5) node {$2E_1^{(2)}$};
\draw [fill=black] (-1.,0.) circle (2.5pt);
\draw[color=black] (-1.,.5) node {$3L_1$};
\draw [fill=black] (0.,0.) circle (2.5pt);
\draw[color=black] (0.,.5) node {$4E_3^{(1)}$};
\draw [fill=black] (1.,0.) circle (2.5pt);
\draw[color=black] (1.,0.5) node {$3L_2$};
\draw [fill=black] (2.,0.) circle (2.5pt);
\draw[color=black] (2.,.5) node {$2E_2^{(2)}$};
\draw [fill=black] (3.,0.) circle (2.5pt);
\draw[color=black] (3.,0.5) node {$E_2^{(1)}$};
\end{scriptsize}
\end{tikzpicture}
\end{figure}

If we take $B$ and $C$ as in the picture below, we obtain a rational elliptic surface whose Jacobian is the surface $X_{321}$ in  Miranda and Persson's list \cite{extr}.

\begin{figure}[H]
\centering
\definecolor{qqwwzz}{rgb}{0,0,0}
\definecolor{ffzzqq}{rgb}{0,0,0}
\begin{tikzpicture}[line cap=round,line join=round,>=triangle 45,x=1.0cm,y=1.0cm, scale=1.4]
\clip(-2.5,-1.5) rectangle (2.5,3.);
\draw [line width=2.pt,color=ffzzqq] (0.,0.) circle (1.cm);
\draw [line width=2.4pt,color=qqwwzz,domain=-2.:2.5] plot(\x,{(--1.--0.6485521326471825*\x)/0.7611702380143298});
\draw [line width=2.4pt,color=qqwwzz,domain=-2.:2.5] plot(\x,{(--1.-0.9306084606647819*\x)/0.3660162468239969});
\draw [line width=2.pt,color=ffzzqq,domain=-2.:2.5] plot(\x,{(-0.39901798854207815-2.55716702712656*\x)/-0.8788459358933687});
\begin{scriptsize}
\draw[color=ffzzqq] (-0.4,0.3) node {$2C$};
\draw [fill=black] (-0.6485521326471825,0.7611702380143298) circle (2.5pt);
\draw[color=black] (-1.1,0.9) node {$P_1$};
\draw [fill=black] (0.9306084606647819,0.3660162468239969) circle (2.5pt);
\draw[color=black] (1.2,0.6) node {$P_2$};
\draw[color=qqwwzz] (-2,-0.9700534759358264) node {$3L_1$};
\draw[color=qqwwzz] (-0.5,2.8) node {$3L_2$};
\draw [fill=black] (0.41782868440681503,1.6697758541753163) circle (2.5pt);
\draw[color=black] (0.9,1.6) node {$P_3$};
\draw [fill=black] (-0.46101725148655365,-0.8873911729512435) circle (2.0pt);
\draw [fill=black] (0.18190643756755348,0.983315843445778) circle (2.0pt);
\end{scriptsize}
\end{tikzpicture}
\caption{The two generators of $\mathcal{P}$ yielding a fiber of type $III^*$ and a multiple fiber of type $I_2$}
\end{figure}
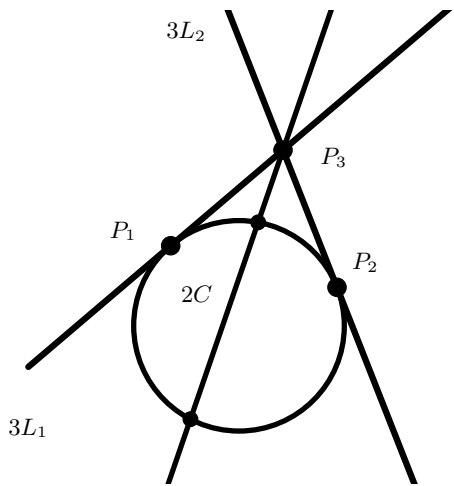

\label{iiistar3-1,3-1}
\end{exen}

\begin{exen}[A triple line, a double line and another line]
Let $C$ be a smooth cubic. Let $L_1$ be an inflection line of $C$ at a point $P_1$ and choose  a line $L_2$ through  $P_1$ which is tangent to $C$ at another point $P_2$. Let $L_3$ be any line through $P_2$ which intersects $C$ at another two points, say $P_3$ and $P_4$.

Then the pencil $\mathcal{P}$ generated by $B=3L_1+L_2+2L_3$ and $2C$ is a Halphen pencil of index two which yields a fiber of type $III^*$ in the associated rational elliptic surface.

Concretely, we can choose coordinates in $\mathbb{P}^2$ so that $C$ is the cubic given by
\[
y^2z=x(x-z)(x-\alpha\cdot z) \qquad \alpha\in \mathbb{C}\backslash\{0,1\}
\]
we can let $L_1$ be the line $z=0$ (hence $P_1 = (0:1:0)$)  and we can choose $L_2$ to be either one of the  lines $x=0, x-z=0$ or $x-\alpha\cdot z=0$. If we choose $L_2$ as $x=0$, then $P_2 = (0:0:1)$ and we can let $L_3$ be the line $y=0$.

Now, if we blow-up $\mathbb{P}^2$ at the nine base points
$
P_1^{(1)},\ldots,P_1^{(5)},P_2^{(1)},P_2^{(2)},P_3^{(1)},P_4^{(1)}
$ we obtain the following (dual) configuration of rational curves:

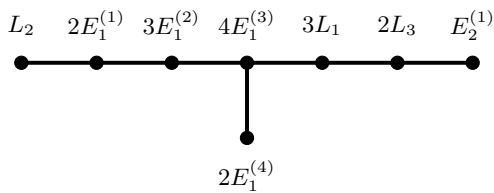
\begin{figure}[H]
\centering
\begin{tikzpicture}[line cap=round,line join=round,>=triangle 45,x=1.0cm,y=1.0cm]
\clip(-4.,-2.) rectangle (4.,1.);
\draw [line width=1.5pt] (-3.,0.)-- (3.,0.);
\draw [line width=1.5pt] (0.,0.)-- (0.,-1.);
\begin{scriptsize}
\draw [fill=black] (0.,-1.) circle (2.5pt);
\draw[color=black] (0.,-1.5) node {$2E_1^{(4)}$};
\draw [fill=black] (-3.,0.) circle (2.5pt);
\draw[color=black] (-3.,.5) node {$L_2$};
\draw [fill=black] (-2.,0.) circle (2.5pt);
\draw[color=black] (-2.,.5) node {$2E_1^{(1)}$};
\draw [fill=black] (-1.,0.) circle (2.5pt);
\draw[color=black] (-1.,.5) node {$3E_1^{(2)}$};
\draw [fill=black] (0.,0.) circle (2.5pt);
\draw[color=black] (0.,.5) node {$4E_1^{(3)}$};
\draw [fill=black] (1.,0.) circle (2.5pt);
\draw[color=black] (1.,0.5) node {$3L_1$};
\draw [fill=black] (2.,0.) circle (2.5pt);
\draw[color=black] (2.,.5) node {$2L_3$};
\draw [fill=black] (3.,0.) circle (2.5pt);
\draw[color=black] (3.,0.5) node {$E_2^{(1)}$};
\end{scriptsize}
\end{tikzpicture}
\end{figure}

\label{iiistar3-1,2-1,1-1}
\end{exen}

\begin{exen}[A triple line, a double line and another line concurrent at a point]
Consider three lines $L_1,L_2$ and $L_3$ concurrent at a point $P_1$ and choose a cubic $C$ as in the picture below

\begin{figure}[H]
\centering
\definecolor{qqwwzz}{rgb}{0,0,0}
\definecolor{ffzzqq}{rgb}{0,0,0}
\begin{tikzpicture}[line cap=round,line join=round,>=triangle 45,x=1.0cm,y=1.0cm,scale=0.8]
\clip(-6.,-4.) rectangle (6.,6.);
\draw[line width=2.4pt,color=ffzzqq,smooth,samples=100,domain=-6.0:6.0] plot(\x,{(\x)^(3.0)+2.0-4.0*(\x)});
\draw [line width=2.pt,color=qqwwzz,domain=-6.:6.] plot(\x,{(--3.302853447212932-1.7455965768312502*\x)/1.});
\draw [line width=2.pt,color=qqwwzz,domain=-6.:6.] plot(\x,{(-8.422827577703453--5.017613692674999*\x)/1.});
\draw [line width=2.pt,color=qqwwzz,domain=-6.:6.] plot(\x,{(--2.6953026222093515-0.9447763427834985*\x)/3.824783217058637});
\begin{scriptsize}
\draw[color=ffzzqq] (-2.,-3) node {2C};
\draw [fill=black] (-0.8668724287476887,4.816062991384289) circle (3.0pt);
\draw[color=black] (-0.35,5.3) node {$P_2$};
\draw[color=qqwwzz] (4.5,-3) node {$3L_2$};
\draw [fill=black] (1.7337448574953773,0.27643435887021717) circle (3.0pt);
\draw[color=black] (2.5,0.5) node {$P_1$};
\draw[color=qqwwzz] (3.5,5.3) node {$L_1$};
\draw [fill=black] (-2.0910383595632593,1.2212107016537157) circle (3.0pt);
\draw[color=black] (-2.6,0.88) node {$P_4$};
\draw[color=qqwwzz] (-5,2.5) node {$2L_3$};
\draw [fill=black] (0.357293502067882,0.6164375966485536) circle (3.0pt);
\draw[color=black] (0.,0.) node {$P_3$};
\end{scriptsize}
\end{tikzpicture}
\caption{The two generators of $\mathcal{P}$ yielding a fiber of type $III^*$ and a multiple fiber of type $I_0$}
\end{figure}
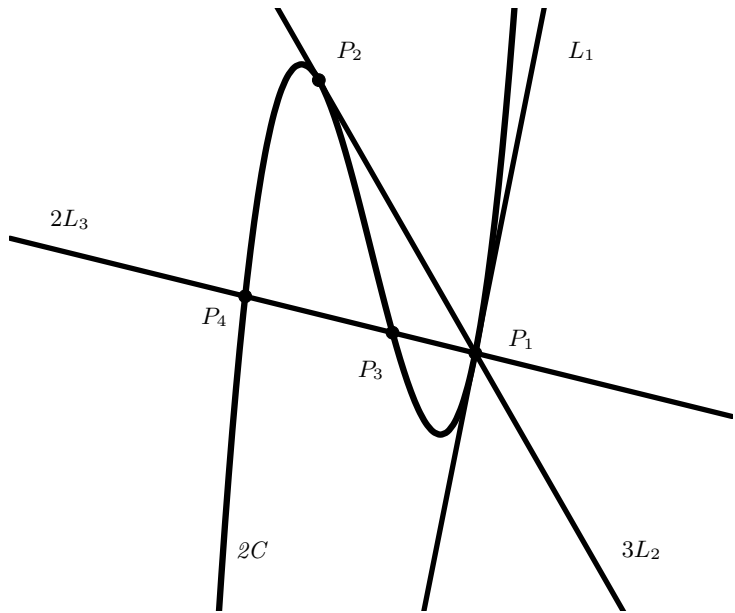

That is, choose a cubic $C$ so that $C$ is tangent to $L_1$ at $P_1$ with full multiplicity, $C$ is tangent to $L_2$ at a point $P_2 (\neq P_1)$ (with multiplicity two) and it intersects $L_3$ at two other points $P_3$ and $P_4$.

The pencil $\mathcal{P}$ generated by $B=L_1+3L_2+2L_3$ and $2C$ is a Halphen pencil of index two and such pencil yields a fiber of type $III^*$ in the associated rational elliptic surface.

Concretely, we blow-up $\mathbb{P}^2$ at the base points $
P_1^{(1)},P_1^{(2)},P_1^{(3)},P_1^{(4)},P_2^{(1)},P_2^{(2)},P_2^{(3)},P_3^{(1)},P_4^{(1)}
$ of $\mathcal{P}$ so that we obtain the following (dual) configuration of rational curves:

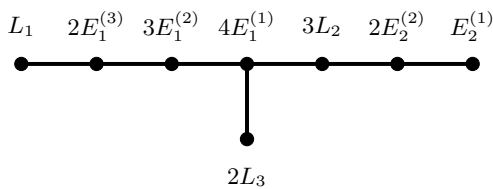
\begin{figure}[H]
\centering
\begin{tikzpicture}[line cap=round,line join=round,>=triangle 45,x=1.0cm,y=1.0cm]
\clip(-4.,-2.) rectangle (4.,1.);
\draw [line width=1.5pt] (-3.,0.)-- (3.,0.);
\draw [line width=1.5pt] (0.,0.)-- (0.,-1.);
\begin{scriptsize}
\draw [fill=black] (0.,-1.) circle (2.5pt);
\draw[color=black] (0.,-1.5) node {$2L_3$};
\draw [fill=black] (-3.,0.) circle (2.5pt);
\draw[color=black] (-3.,.5) node {$L_1$};
\draw [fill=black] (-2.,0.) circle (2.5pt);
\draw[color=black] (-2.,.5) node {$2E_1^{(3)}$};
\draw [fill=black] (-1.,0.) circle (2.5pt);
\draw[color=black] (-1.,.5) node {$3E_1^{(2)}$};
\draw [fill=black] (0.,0.) circle (2.5pt);
\draw[color=black] (0.,.5) node {$4E_1^{(1)}$};
\draw [fill=black] (1.,0.) circle (2.5pt);
\draw[color=black] (1.,0.5) node {$3L_2$};
\draw [fill=black] (2.,0.) circle (2.5pt);
\draw[color=black] (2.,.5) node {$2E_2^{(2)}$};
\draw [fill=black] (3.,0.) circle (2.5pt);
\draw[color=black] (3.,0.5) node {$E_2^{(1)}$};
\end{scriptsize}
\end{tikzpicture}
\end{figure}

\label{iiistar3-1,2-1,1-1conc}
\end{exen}

\begin{exen}[A triple line, a conic and a line]
Let $Q$ be a (smooth) conic. Choose a point $P_1$ in $Q$ and let $L_1$ be the tangent line to $Q$ at $P_1$. Choose two other points in $Q$, say $P_2$ and $P_3$, and let $L_2$ be the line joining them. Let $P_4$ be the intersection point between $L_1$ and $L_2$. We can construct a cubic $C$ through these four points so that $C$ is tangent to $Q$ (resp. $L_1$) at $P_1$ with multiplicity four (resp. two).

The pencil $\mathcal{P}$ generated by $B=3L_1+Q+L_2$ and $2C$ is a Halphen pencil of index two which yields a fiber of type $III^*$ in the corresponding rational elliptic surface.

Concretely, we can choose coordinates in $\mathbb{P}^2$ so that  $Q$ is the conic given by $x^2+yz=0$ and we have  $P_1=(0:1:0),P_2=(-1:-1:1)$ and $P_3=(0:0:1)$. Then $L_1$ is the line $z=0$, $L_2$ is the line $x+y=0, P_4=(-1:1:0)$ and $C$  is the cubic given by
\[
(x+z)xz+(x^2+yz)(x+y)=0
\]

If we blow-up $\mathbb{P}^2$ at the nine base points $
P_1^{(1)},\ldots,P_1^{(5)},P_2^{(1)},P_3^{(1)},P_4^{(1)},P_4^{(2)}
$ we obtain the following (dual) configuration of rational curves:

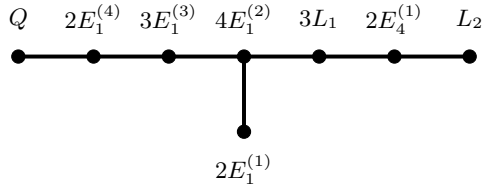
\begin{figure}[H]
\centering
\begin{tikzpicture}[line cap=round,line join=round,>=triangle 45,x=1.0cm,y=1.0cm]
\clip(-4.,-2.) rectangle (4.,1.);
\draw [line width=1.5pt] (-3.,0.)-- (3.,0.);
\draw [line width=1.5pt] (0.,0.)-- (0.,-1.);
\begin{scriptsize}
\draw [fill=black] (0.,-1.) circle (2.5pt);
\draw[color=black] (0.,-1.5) node {$2E_1^{(1)}$};
\draw [fill=black] (-3.,0.) circle (2.5pt);
\draw[color=black] (-3.,.5) node {$Q$};
\draw [fill=black] (-2.,0.) circle (2.5pt);
\draw[color=black] (-2.,.5) node {$2E_1^{(4)}$};
\draw [fill=black] (-1.,0.) circle (2.5pt);
\draw[color=black] (-1.,.5) node {$3E_1^{(3)}$};
\draw [fill=black] (0.,0.) circle (2.5pt);
\draw[color=black] (0.,.5) node {$4E_1^{(2)}$};
\draw [fill=black] (1.,0.) circle (2.5pt);
\draw[color=black] (1.,0.5) node {$3L_1$};
\draw [fill=black] (2.,0.) circle (2.5pt);
\draw[color=black] (2.,.5) node {$2E_4^{(1)}$};
\draw [fill=black] (3.,0.) circle (2.5pt);
\draw[color=black] (3.,0.5) node {$L_2$};
\end{scriptsize}
\end{tikzpicture}
\end{figure}

\label{iiistar3-1,1-2,1-1}
\end{exen}

\begin{exen}[A triple line and a cubic]
Let $D:d=0$ be a nodal cubic and let $P_1$ denote its node. Let $P_2$ be a point in $D$ which is not a flex and let $L:l=0$ denote the tangent line to $D$ at $P_2$. Let $P_3$ be the third intersection point between $L$ and $D$ and let $L':l'=0$ denote the line joining $P_1$ and $P_3$.

Then the cubic $C$ given by $l^2l'+d=0$ is such that the intersection multiplicity of $D$ and $C$ at $P=P_2$ (resp. $P=P_3$) is $4$ (resp. $3$). Moreover, by construction, the node $P_1$ lies in $C$.

Concretely, if $D$ is the nodal cubic given by $y^2z=x^2(x+z)$ we have that $P_1=(0:0:1)$ and we can let $P_2=(1:0:-1)$ so that $L$ is the line $x+z=0$. Then $P_3=(0:1:0)$ and $L'$ is the line $x=0$. Thus, $C$ is the cubic given by
\[
z(y^2+x^2+xz)=0
\]
Note that $C$ consists of a line $(z=0)$ and a conic $(y^2+x^2+xz=0)$. Moreover, the line is an inflection line of $D$ and the node $P_1$ lies in the conic.

Now, the pencil $\mathcal{P}$ generated by $B=3L+D$ and $2C$ is a Halphen pencil of index two and the associated rational elliptic surface has a fiber of type $III^*$.

More precisely, blowing-up $\mathbb{P}^2$ at $
P_1^{(1)},P_2^{(1)},\ldots,P_2^{(5)},P_3^{(1)},\ldots,P_3^{(3)}
$ we obtain the following (dual) configuration of rational curves:

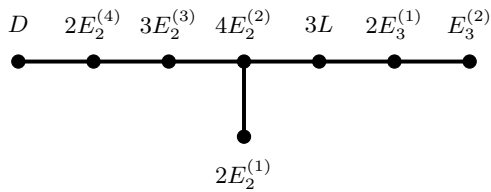
\begin{figure}[H]
\centering
\begin{tikzpicture}[line cap=round,line join=round,>=triangle 45,x=1.0cm,y=1.0cm]
\clip(-4.,-2.) rectangle (4.,1.);
\draw [line width=1.5pt] (-3.,0.)-- (3.,0.);
\draw [line width=1.5pt] (0.,0.)-- (0.,-1.);
\begin{scriptsize}
\draw [fill=black] (0.,-1.) circle (2.5pt);
\draw[color=black] (0.,-1.5) node {$2E_2^{(1)}$};
\draw [fill=black] (-3.,0.) circle (2.5pt);
\draw[color=black] (-3.,.5) node {$D$};
\draw [fill=black] (-2.,0.) circle (2.5pt);
\draw[color=black] (-2.,.5) node {$2E_2^{(4)}$};
\draw [fill=black] (-1.,0.) circle (2.5pt);
\draw[color=black] (-1.,.5) node {$3E_2^{(3)}$};
\draw [fill=black] (0.,0.) circle (2.5pt);
\draw[color=black] (0.,.5) node {$4E_2^{(2)}$};
\draw [fill=black] (1.,0.) circle (2.5pt);
\draw[color=black] (1.,0.5) node {$3L$};
\draw [fill=black] (2.,0.) circle (2.5pt);
\draw[color=black] (2.,.5) node {$2E_3^{(1)}$};
\draw [fill=black] (3.,0.) circle (2.5pt);
\draw[color=black] (3.,0.5) node {$E_3^{(2)}$};
\end{scriptsize}
\end{tikzpicture}
\end{figure}

\label{iiistar3-1,1-3}
\end{exen}

\begin{exen}[A line with multiplicity four and a conic]
Consider either a smooth or nodal cubic $C$. Choose smooth points $P_1,P_2\in C$ so that there exists a conic $Q$ which is tangent to $C$ at $P_1$ (resp. $P_2$) with multiplicity $4$ (resp. $2$). Let $L$ be the line joining $P_1$ and $P_2$ and let $P_3$ be the third intersection point between $L$ and $C$. Then the pencil $\mathcal{P}$ generated by $B=Q+4L$ and $2C$ is a Halphen pencil of index two which yields a fiber of type $III^*$ in the associated rational elliptic surface.

For instance, consider the cubic $C$ given by $x^2z+(x^2+yz)(y+z)=0$ and let $P_1=(0:1:0)$ and $P_2=(0:0:1)$. Then $L:x=0$ and $P_3=(0:1:-1)$ and we can take $Q: x^2+yz=0$.

Now, if we blow-up $\mathbb{P}^2$ at the nine base points
$
P_1^{(1)},\ldots,P_1^{(4)},P_2^{(1)},\ldots,P_2^{(3)},P_3^{(1)},P_3^{(2)}
$ we obtain the following (dual) configuration of rational curves:

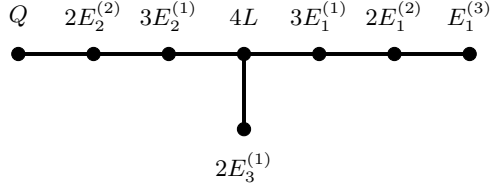
\begin{figure}[H]
\centering
\begin{tikzpicture}[line cap=round,line join=round,>=triangle 45,x=1.0cm,y=1.0cm]
\clip(-4.,-2.) rectangle (4.,1.);
\draw [line width=1.5pt] (-3.,0.)-- (3.,0.);
\draw [line width=1.5pt] (0.,0.)-- (0.,-1.);
\begin{scriptsize}
\draw [fill=black] (0.,-1.) circle (2.5pt);
\draw[color=black] (0.,-1.5) node {$2E_3^{(1)}$};
\draw [fill=black] (-3.,0.) circle (2.5pt);
\draw[color=black] (-3.,.5) node {$Q$};
\draw [fill=black] (-2.,0.) circle (2.5pt);
\draw[color=black] (-2.,.5) node {$2E_2^{(2)}$};
\draw [fill=black] (-1.,0.) circle (2.5pt);
\draw[color=black] (-1.,.5) node {$3E_2^{(1)}$};
\draw [fill=black] (0.,0.) circle (2.5pt);
\draw[color=black] (0.,.5) node {$4L$};
\draw [fill=black] (1.,0.) circle (2.5pt);
\draw[color=black] (1.,0.5) node {$3E_1^{(1)}$};
\draw [fill=black] (2.,0.) circle (2.5pt);
\draw[color=black] (2.,.5) node {$2E_1^{(2)}$};
\draw [fill=black] (3.,0.) circle (2.5pt);
\draw[color=black] (3.,0.5) node {$E_1^{(3)}$};
\end{scriptsize}
\end{tikzpicture}
\end{figure}

\label{iiistar1-2,4-1}
\end{exen}

\begin{exen}[A line with multiplicity four and two other lines]
Consider either a smooth or nodal cubic $C$ and let $P_4$ be a flex point of $C$. We can always choose two lines $L_1$ and $L_2$ through $P_4$ which are tangent to $C$ at two other points $P_1$ and $P_2$, respectively. Moreover, if $L_3$ is the line joining $P_1$ and $P_2$, then $C$ intersects $L_3$ at a third point $P_3$ and we have that the pencil $\mathcal{P}$ generated by $B=L_1+L_2+4L_3$ and $2C$ is a Halphen pencil of index two with base points
\[
P_1^{(1)},\ldots,P_1^{(3)},P_2^{(1)},\ldots,P_2^{(3)},P_3^{(1)},P_3^{(2)},P_4^{(1)}
\]
Blowing-up $\mathbb{P}^2$ at these nine base points yields a fiber of type $III^*$ in the associated rational elliptic surface. Explicitly, we obtain the following (dual) configuration of rational curves:

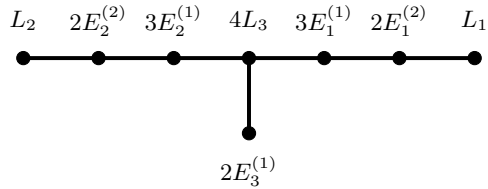
\begin{figure}[H]
\centering
\begin{tikzpicture}[line cap=round,line join=round,>=triangle 45,x=1.0cm,y=1.0cm]
\clip(-4.,-2.) rectangle (4.,1.);
\draw [line width=1.5pt] (-3.,0.)-- (3.,0.);
\draw [line width=1.5pt] (0.,0.)-- (0.,-1.);
\begin{scriptsize}
\draw [fill=black] (0.,-1.) circle (2.5pt);
\draw[color=black] (0.,-1.5) node {$2E_3^{(1)}$};
\draw [fill=black] (-3.,0.) circle (2.5pt);
\draw[color=black] (-3.,.5) node {$L_2$};
\draw [fill=black] (-2.,0.) circle (2.5pt);
\draw[color=black] (-2.,.5) node {$2E_2^{(2)}$};
\draw [fill=black] (-1.,0.) circle (2.5pt);
\draw[color=black] (-1.,.5) node {$3E_2^{(1)}$};
\draw [fill=black] (0.,0.) circle (2.5pt);
\draw[color=black] (0.,.5) node {$4L_3$};
\draw [fill=black] (1.,0.) circle (2.5pt);
\draw[color=black] (1.,0.5) node {$3E_1^{(1)}$};
\draw [fill=black] (2.,0.) circle (2.5pt);
\draw[color=black] (2.,.5) node {$2E_1^{(2)}$};
\draw [fill=black] (3.,0.) circle (2.5pt);
\draw[color=black] (3.,0.5) node {$L_1$};
\end{scriptsize}
\end{tikzpicture}
\end{figure}

Note that, concretely, we can choose coordinates in $\mathbb{P}^2$ so that $C$ is the cubic given by
\[
y^2z=x(x-z)(x-\alpha\cdot z) \qquad \alpha\in \mathbb{C}\backslash\{0,1\}
\]
we can let $P_4 = (0:1:0)$  and we can choose $L_1$ and $L_2$ to be the lines $x=0$ and $x-z=0$. Then $P_1 = (0:0:1), P_2=(1:0:1), L_3$ is the line $y=0$ and $P_3=(\alpha:0:1)$.
\label{iiistar4-1,1-1,1-1}
\end{exen}

\subsection{Type $II^*$}

We now construct all possible examples of Halphen pencils of index two that yield a fiber of type $II^*$ in the corresponding rational elliptic surface (Theorem \ref{allpossibleiistar}).

\begin{exe}[A triple conic \cite{fuji98}]
We begin with an example of a rational elliptic surface whose Jacobian is the surface $X_{211}$ in  Miranda and Persson's list \cite{extr}.

Let $C$ be a cubic with a node and  let $P_0$ be an inflection point of $C$ that we take as the identity for the group law. Choose  another point $P$ in $C$ satisfying $6P=P_0$. Then there exists a conic $Q$ tangent to $C$ at $P$ with multiplicity $6$ and to the pencil generated by $B=3Q$ and $2C$ we can associate a rational elliptic fibration $Y\to \mathbb{P}^1$ of index two with $II^*+\,_{2}I_1+I_1$ singular fibers. 

Concretely, we blow-up $\mathbb{P}^2$  at the nine points $P_1^{(1)},\ldots,P_1^{(9)}$ where $P_1^{(1)}=P$. The strict transform of $C$ is the multiple fiber and the strict transform of $Q$ is the component of multiplicity $3$ in the $II^*$ fiber that intersects the component of multiplicity $6$. 

More precisely, we get the following (dual) configuration of rational curves:

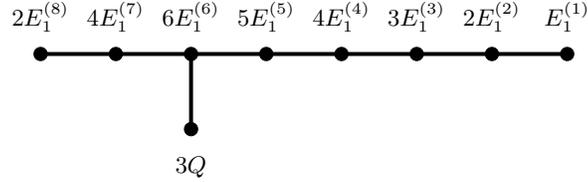
\begin{figure}[H]
\centering
\begin{tikzpicture}[line cap=round,line join=round,>=triangle 45,x=1.0cm,y=1.0cm]
\clip(-3.,-2.5) rectangle (5.5,1.);
\draw [line width=1.5pt] (-2.,0.)-- (5.,0.);
\draw [line width=1.5pt] (0.,0.)-- (0.,-1.);
\begin{scriptsize}
\draw [fill=black] (0.,-1.) circle (2.5pt);
\draw[color=black] (0.,-1.5) node {$3Q$};
\draw [fill=black] (-2.,0.) circle (2.5pt);
\draw[color=black] (-2.,.5) node {$2E_1^{(8)}$};
\draw [fill=black] (-1.,0.) circle (2.5pt);
\draw[color=black] (-1.,.5) node {$4E_1^{(7)}$};
\draw [fill=black] (0.,0.) circle (2.5pt);
\draw[color=black] (0.,.5) node {$6E_1^{(6)}$};
\draw [fill=black] (1.,0.) circle (2.5pt);
\draw[color=black] (1.,0.5) node {$5E_1^{(5)}$};
\draw [fill=black] (2.,0.) circle (2.5pt);
\draw[color=black] (2.,.5) node {$4E_1^{(4)}$};
\draw [fill=black] (3.,0.) circle (2.5pt);
\draw[color=black] (3.,0.5) node {$3E_1^{(3)}$};
\draw [fill=black] (4.,0.) circle (2.5pt);
\draw[color=black] (4,.5) node {$2E_1^{(2)}$};
\draw [fill=black] (5.,0.) circle (2.5pt);
\draw[color=black] (5.,0.5) node {$E_1^{(1)}$};
\end{scriptsize}
\end{tikzpicture}
\end{figure}

\label{iistartripleconic}
\end{exe}

\begin{exen}[Two triple lines]
Let $C$ be either a smooth or nodal cubic. Let $L_1$ be an inflection line of $C$ at a point $P_1$ and let $L_2$ be a line through $P_1$ which is tangent to $C$ at another point $P_2$.

Then the pencil $\mathcal{P}$ generated by $B=3L_1+3L_2$ and $2C$ is a Halphen pencil of index two which yields a fiber of type $II^*$ in the associated rational elliptic surface.

Concretely, (if $C$ is smooth) we can choose coordinates in $\mathbb{P}^2$ so that $C$ is the cubic given by
\[
y^2z=x(x-z)(x-\alpha\cdot z) \qquad \alpha\in \mathbb{C}\backslash\{0,1\}
\]
we can let $L_1$ be the line $z=0$ (hence $P_1 = (0:1:0)$)  and we can choose $L_2$ to be either one of the  lines $x=0, x-z=0$ or $x-\alpha\cdot z=0$. 

If we choose $L_2$ as $x=0$, then $P_2 = (0:0:1)$ and, similarly, if we take $L_2$ as $x-z=0$ (resp. $x-\alpha\cdot z=0$), then $P_2=(1:0:1)$  (resp.  $P_2=(\alpha:0:1)$).

In any case we blow-up $\mathbb{P}^2$ at the nine base points $
P_1^{(1)},\ldots,P_1^{(6)},P_2^{(1)},P_2^{(2)},P_2^{(3)}
$ and we obtain the following (dual) configuration of rational curves:

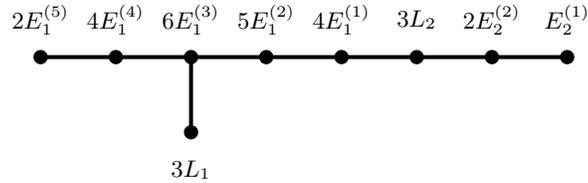
\begin{figure}[H]
\centering
\begin{tikzpicture}[line cap=round,line join=round,>=triangle 45,x=1.0cm,y=1.0cm]
\clip(-3.,-2.5) rectangle (5.5,1.);
\draw [line width=1.5pt] (-2.,0.)-- (5.,0.);
\draw [line width=1.5pt] (0.,0.)-- (0.,-1.);
\begin{scriptsize}
\draw [fill=black] (0.,-1.) circle (2.5pt);
\draw[color=black] (0.,-1.5) node {$3L_1$};
\draw [fill=black] (-2.,0.) circle (2.5pt);
\draw[color=black] (-2.,.5) node {$2E_1^{(5)}$};
\draw [fill=black] (-1.,0.) circle (2.5pt);
\draw[color=black] (-1.,.5) node {$4E_1^{(4)}$};
\draw [fill=black] (0.,0.) circle (2.5pt);
\draw[color=black] (0.,.5) node {$6E_1^{(3)}$};
\draw [fill=black] (1.,0.) circle (2.5pt);
\draw[color=black] (1.,0.5) node {$5E_1^{(2)}$};
\draw [fill=black] (2.,0.) circle (2.5pt);
\draw[color=black] (2.,.5) node {$4E_1^{(1)}$};
\draw [fill=black] (3.,0.) circle (2.5pt);
\draw[color=black] (3.,0.5) node {$3L_2$};
\draw [fill=black] (4.,0.) circle (2.5pt);
\draw[color=black] (4,.5) node {$2E_2^{(2)}$};
\draw [fill=black] (5.,0.) circle (2.5pt);
\draw[color=black] (5.,0.5) node {$E_2^{(1)}$};
\end{scriptsize}
\end{tikzpicture}
\end{figure}

\label{iistartwotriplelines}
\end{exen}

\begin{exen}[A triple line and a cubic]
Let $D: d=0$ be a nodal cubic and let $P_1$ denote its node. Let $L:l=0$ be an inflection line of  $D$ and denote the flex point by $P_2$. Let $L':l'=0$ be the line joining $P_1$ and $P_2$. 

Then the cubic $C$ given by $l^2l'+d=0$ is such that the intersection multiplicity of $D$ and $C$ at $P_2$ is $7$ and, by construction, the node $P_1$ lies on it. We also have that $L$ is also an inflection line of $C$ at $P_2$.

Concretely, we can choose as $D$ the nodal cubic given by $
y^2z=x^2(x+z)
$, then $P_1=(0:0:1)$ and we can choose $L$ to be the line $z=0$ so that $P_2=(0:1:0)$. Then $L'$ is the line $x=0$ and $C$ has equation
\[
z^2x+y^2z-x^3-x^2z=0
\]

Now, the pencil $\mathcal{P}$ generated by $B=D+3L$ and $2C$ is a Halphen pencil of index two which yields a fiber of type $II^*$ in the associated rational elliptic surface.

If we blow-up $\mathbb{P}^2$ at the nine base points $P_1^{(1)},P_2^{(1)},\ldots,P_2^{(8)}$ then we obtain the following (dual) configuration of rational curves:

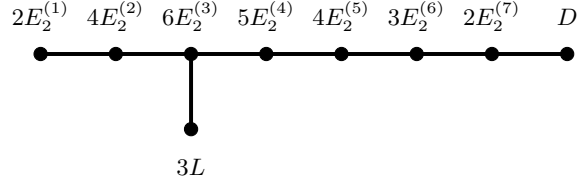
\begin{figure}[H]
\centering
\begin{tikzpicture}[line cap=round,line join=round,>=triangle 45,x=1.0cm,y=1.0cm]
\clip(-3.,-2.5) rectangle (5.5,1.);
\draw [line width=1.5pt] (-2.,0.)-- (5.,0.);
\draw [line width=1.5pt] (0.,0.)-- (0.,-1.);
\begin{scriptsize}
\draw [fill=black] (0.,-1.) circle (2.5pt);
\draw[color=black] (0.,-1.5) node {$3L$};
\draw [fill=black] (-2.,0.) circle (2.5pt);
\draw[color=black] (-2.,.5) node {$2E_2^{(1)}$};
\draw [fill=black] (-1.,0.) circle (2.5pt);
\draw[color=black] (-1.,.5) node {$4E_2^{(2)}$};
\draw [fill=black] (0.,0.) circle (2.5pt);
\draw[color=black] (0.,.5) node {$6E_2^{(3)}$};
\draw [fill=black] (1.,0.) circle (2.5pt);
\draw[color=black] (1.,0.5) node {$5E_2^{(4)}$};
\draw [fill=black] (2.,0.) circle (2.5pt);
\draw[color=black] (2.,.5) node {$4E_2^{(5)}$};
\draw [fill=black] (3.,0.) circle (2.5pt);
\draw[color=black] (3.,0.5) node {$3E_2^{(6)}$};
\draw [fill=black] (4.,0.) circle (2.5pt);
\draw[color=black] (4,.5) node {$2E_2^{(7)}$};
\draw [fill=black] (5.,0.) circle (2.5pt);
\draw[color=black] (5.,0.5) node {$D$};
\end{scriptsize}
\end{tikzpicture}
\end{figure}

\label{iistartriplelinecubic}
\end{exen}

\begin{exen}[A line with multiplicity four and a conic]
Let $C$ be either a smooth or nodal cubic. Choose a sextactic point $P_1\in C$ (see Definition \ref{sextactic}). And let $Q$ be the corresponding osculating conic at $P_1$. Choose a line $L$ which is tangent to both $Q$ and $C$ at $P_1$ and let $P_2$ be the third point of intersection between $L$ and $C$. Then the pencil $\mathcal{P}$ generated by $B=Q+4L$ and $2C$ is a Halphen pencil of index two which yields a fiber of type $II^*$ in the associated rational elliptic surface.

For instance, consider the cubic $C$ given by
\[
-3x^3+xz^2+y^2z+2xy^2=x^3+(y^2-2x^2+xz)\cdot (2x+z)=0
\]
Let $P_1=(0:0:1)$, let $Q: y^2-2x^2+xz=0$ and let $L:x=0$. Then the intersection multiplicity of $Q$ and $C$ at $P_1$ is $6$ and we have that $P_2=(0:1:0)$ is a flex point with inflection line $2x+z=0$.

Now, if we blow-up $\mathbb{P}^2$ at $
P_1^{(1)},\ldots,P_1^{(7)},P_2^{(1)},P_2^{(2)}
$ we obtain the following (dual) configuration of rational curves:

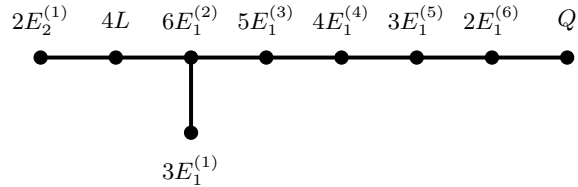
\begin{figure}[H]
\centering
\begin{tikzpicture}[line cap=round,line join=round,>=triangle 45,x=1.0cm,y=1.0cm]
\clip(-3.,-2.5) rectangle (5.5,1.);
\draw [line width=1.5pt] (-2.,0.)-- (5.,0.);
\draw [line width=1.5pt] (0.,0.)-- (0.,-1.);
\begin{scriptsize}
\draw [fill=black] (0.,-1.) circle (2.5pt);
\draw[color=black] (0.,-1.5) node {$3E_1^{(1)}$};
\draw [fill=black] (-2.,0.) circle (2.5pt);
\draw[color=black] (-2.,.5) node {$2E_2^{(1)}$};
\draw [fill=black] (-1.,0.) circle (2.5pt);
\draw[color=black] (-1.,.5) node {$4L$};
\draw [fill=black] (0.,0.) circle (2.5pt);
\draw[color=black] (0.,.5) node {$6E_1^{(2)}$};
\draw [fill=black] (1.,0.) circle (2.5pt);
\draw[color=black] (1.,0.5) node {$5E_1^{(3)}$};
\draw [fill=black] (2.,0.) circle (2.5pt);
\draw[color=black] (2.,.5) node {$4E_1^{(4)}$};
\draw [fill=black] (3.,0.) circle (2.5pt);
\draw[color=black] (3.,0.5) node {$3E_1^{(5)}$};
\draw [fill=black] (4.,0.) circle (2.5pt);
\draw[color=black] (4,.5) node {$2E_1^{(6)}$};
\draw [fill=black] (5.,0.) circle (2.5pt);
\draw[color=black] (5.,0.5) node {$Q$};
\end{scriptsize}
\end{tikzpicture}
\end{figure}

\label{iistarline4}
\end{exen}

\begin{exen}[A line with multiplicity five and another line]
Consider either a smooth or nodal cubic $C$ and let $L_1$ be an inflection line of $C$ at a point $P_1$. We can always choose another line $L_2$ through $P_1$ which is tangent to $C$ at another point $P_2$. And the pencil $\mathcal{P}$ generated by $B=5L_2+L_1$ and $2C$ is a Halphen pencil of index two which yields a fiber of type $II^*$ in the associated rational elliptic surface.

Concretely, (if $C$ is smooth) we can choose coordinates in $\mathbb{P}^2$ so that $C$ is the cubic given by
\[
y^2z=x(x-z)(x-\alpha\cdot z) \qquad \alpha\in \mathbb{C}\backslash\{0,1\}
\]
we can let $L_1$ be the line $z=0$ (hence $P_1 = (0:1:0)$)  and we can choose $L_2$ to be either one of the  lines $x=0, x-z=0$ or $x-\alpha\cdot z=0$. 

If we choose $L_2$ as $x=0$, then $P_2 = (0:0:1)$ and, similarly, if we take $L_2$ as $x-z=0$ (resp. $x-\alpha\cdot z=0$), then $P_2=(1:0:1)$  (resp.  $P_2=(\alpha:0:1)$).

In any case we blow-up $\mathbb{P}^2$ at $
P_1^{(1)},\ldots,P_1^{(4)},P_2^{(1)},\ldots,P_2^{(5)}$
and we obtain the following (dual) configuration of rational curves:

\begin{figure}[H]
\centering
\begin{tikzpicture}[line cap=round,line join=round,>=triangle 45,x=1.0cm,y=1.0cm]
\clip(-3.,-2.5) rectangle (5.5,1.);
\draw [line width=1.5pt] (-2.,0.)-- (5.,0.);
\draw [line width=1.5pt] (0.,0.)-- (0.,-1.);
\begin{scriptsize}
\draw [fill=black] (0.,-1.) circle (2.5pt);
\draw[color=black] (0.,-1.5) node {$3E_2^{(1)}$};
\draw [fill=black] (-2.,0.) circle (2.5pt);
\draw[color=black] (-2.,.5) node {$2E_2^{(4)}$};
\draw [fill=black] (-1.,0.) circle (2.5pt);
\draw[color=black] (-1.,.5) node {$4E_2^{(3)}$};
\draw [fill=black] (0.,0.) circle (2.5pt);
\draw[color=black] (0.,.5) node {$6E_2^{(2)}$};
\draw [fill=black] (1.,0.) circle (2.5pt);
\draw[color=black] (1.,0.5) node {$5L_2$};
\draw [fill=black] (2.,0.) circle (2.5pt);
\draw[color=black] (2.,.5) node {$4E_1^{(1)}$};
\draw [fill=black] (3.,0.) circle (2.5pt);
\draw[color=black] (3.,0.5) node {$3E_1^{(2)}$};
\draw [fill=black] (4.,0.) circle (2.5pt);
\draw[color=black] (4,.5) node {$2E_1^{(3)}$};
\draw [fill=black] (5.,0.) circle (2.5pt);
\draw[color=black] (5.,0.5) node {$L_1$};
\end{scriptsize}
\end{tikzpicture}
\end{figure}

\label{iistarline5}
\end{exen}

\bibliography{references} 
\bibliographystyle{plain}

\end{document}